\documentclass[11pt]{amsart}

\newif\ifnotes\notesfalse

\ifnotes
\usepackage{xcolor}
\definecolor{mygrey}{gray}{0.50}
\newcommand{\notename}[2]{{\textcolor{mygrey}{\footnotesize{\bf (#1:} {#2}{\bf ) }}}}
\newcommand{\noteswarning}{{\begin{center} {\Large WARNING: NOTES ON}\end{center}}}

\else

\newcommand{\notename}[2]{{}}
\newcommand{\noteswarning}{{}}

\fi

\usepackage{latexsym,amssymb,amsfonts,amsmath,mathrsfs,float, hyperref}
\usepackage{thmtools,thm-restate}
\usepackage{xcolor}
\usepackage{tikz}
\hypersetup{
    colorlinks,
    linkcolor={black},
    citecolor={black},
}
\usepackage[capitalise]{cleveref}
\usepackage{doi}

\usepackage[ruled, vlined, linesnumbered]{algorithm2e}

\calclayout

\newcommand{\parset}{
	\setlength{\parskip}{3mm}
  	\setlength{\parindent}{0mm}}
\parset
	
  \renewcommand{\Pr}{\mathop{\mbox{\rm Pr}}}
  \newcommand{\Exp}{\mathop{\mathbb{E}}}
  \newcommand{\E}{\mathbb{E}}

  \newcommand{\Lag}{\mathrm{Lag}}
  \DeclareMathOperator{\Stab}{Stab}\DeclareMathOperator{\Spec}{Spec}
  
  \newcommand{\R}{\mathbb{R}} 
  \newcommand{\C}{\mathbb{C}} 
  \newcommand{\D}{\mathbb{D}} 
  \newcommand{\N}{\mathbb{N}} 
  \newcommand{\Z}{\mathbb{Z}} 
  \newcommand{\F}{\mathbb{F}} 
  \newcommand{\Lcal}{\mathcal{L}}
  \newcommand{\HW}{\mathrm{HW}}
  \newcommand{\pmset}[1]{\{-1,1\}^{#1}} 
  \newcommand{\bset}[1]{\{0,1\}^{#1}} 
  \DeclareMathOperator{\vspan}{Span} 

  \newcommand{\tvect}[2]{%
   \ensuremath{\Bigl(\negthinspace\begin{smallmatrix}#1\\#2\end{smallmatrix}\Bigr)}}

  \newcommand{\Id}{\ensuremath{\mathop{\rm I}\nolimits}}

  \newcommand{\PFR}{\textsf{PFR}}
  \newcommand{\PGI}{\textsf{PGI}}
  \newcommand{\Iso}{\mathrm{Iso}}
  \newcommand{\Sp}{\mathrm{Sp}}
  \newcommand{\ASp}{\mathrm{ASp}}
  \DeclareMathOperator{\GL}{GL}
  
  \newcommand{\one}{\mathbf{1}}

  \DeclareMathOperator{\supp}{supp}

  \newcommand{\eps}{\varepsilon}

  \DeclareMathOperator{\U}{\mathcal{U}}
  
  \newcommand{\poly}{\mbox{\rm poly}}

  \DeclareMathOperator{\tr}{tr}
  
  \DeclareMathOperator{\Diag}{Diag}
  \DeclareMathOperator{\spec}{Spec}
  \DeclareMathOperator{\Est}{Est}
  \DeclareMathOperator{\Weyl}{Weyl}
  \DeclareMathOperator{\FourEst}{FourEst}
  \DeclareMathOperator{\dist}{dist}

  \newcommand{\ra}{\rangle}
  \newcommand{\la}{\langle}
  \newcommand{\tp}{\mathsf{T}}

  \newcommand{\ket}[1]{|#1\rangle}
  \newcommand{\bra}[1]{\langle#1|}

  \newcommand{\A}{\mathcal A}
  \newcommand{\B}{\mathcal B}

  \newcommand{\beq}{\begin{equation}}
  \newcommand{\eeq}{\end{equation}}
  \newcommand{\beqn}{\begin{equation*}}
  \newcommand{\eeqn}{\end{equation*}}
  \newcommand{\beqr}{\begin{eqnarray}}
  \newcommand{\eeqr}{\end{eqnarray}}
  \newcommand{\beqrn}{\begin{eqnarray*}}
  \newcommand{\eeqrn}{\end{eqnarray*}}
  \newcommand{\bmline}{\begin{multline}}
  \newcommand{\emline}{\end{multline}}
  \newcommand{\bmlinen}{\begin{multline*}}
  \newcommand{\emlinen}{\end{multline*}}

  \theoremstyle{plain}
  \newtheorem{theorem}{Theorem}[section]
  \newtheorem{lemma}[theorem]{Lemma}
  \newtheorem{proposition}[theorem]{Proposition}
  \newtheorem{claim}[theorem]{Claim}
  \newtheorem{corollary}[theorem]{Corollary}
  
  \theoremstyle{definition}
  \newtheorem{definition}[theorem]{Definition}

  \theoremstyle{remark}
  \newtheorem{remark}[theorem]{Remark}
  
  \renewenvironment{proof}[1][]{
    	\begin{trivlist}
     	\item[\hspace{\labelsep}{\em\noindent Proof#1:\/}]}
     	{{\hfill$\Box$}
    	\end{trivlist}
  }

\begin{document}

\title{An algorithmic Polynomial Freiman--Ruzsa theorem}

\begin{abstract}
    We provide algorithmic versions of the Polynomial Freiman--Ruzsa theorem of Gowers, Green, Manners, and Tao (\emph{Ann.\ of Math.,}~2025).
    In particular, we give a polynomial-time algorithm that, given a set $A \subseteq \mathbb{F}_2^n$ with doubling constant~$K$, returns a subspace $V \subseteq \mathbb{F}_2^n$ of size $|V| \leq |A|$ such that~$A$ can be covered by $2K^C$ translates of $V$, for a universal constant $C>1$.
    We also provide efficient algorithms for several ``equivalent'' formulations of the Polynomial Freiman--Ruzsa theorem, such as the polynomial Gowers inverse theorem, the classification of approximate Freiman homomorphisms, and quadratic structure-vs-randomness decompositions.

    Our algorithmic framework is based on a new and optimal version of the Quadratic Goldreich--Levin algorithm, which we obtain using ideas from quantum learning theory.
    This framework fundamentally relies on a connection between quadratic Fourier analysis and symplectic geometry, first speculated by Green and Tao (\emph{Proc.\ of Edinb.\ Math.\ Soc.,}~2008) and which we make explicit in this paper.
\end{abstract}

\author[D.\ Castro-Silva]{Davi Castro-Silva}
\address{Department of Computer Science and Technology, University of Cambridge, UK}
\email{dd654@cam.ac.uk} 
\author[J.\ Bri\"{e}t]{Jop Bri\"et} 
\address{Centrum Wiskunde \& Informatica (CWI) and Leiden University, The Netherlands}
\email{j.briet@cwi.nl}
\author[S.\ Arunachalam]{Srinivasan Arunachalam}
\address{IBM Research, Silicon Valley Lab, USA}
\email{srinivasan.arunachalam@ibm.com}
\author[A.\ Dutt]{\newline Arkopal Dutt} 
\address{IBM Research, Cambridge, USA}
\email{arkopal@ibm.com}
\author[T.\ Gur]{Tom Gur} 
\address{Department of Computer Science and Technology, University of Cambridge, UK}
\email{tom.gur@cl.cam.ac.uk}

\maketitle

\noteswarning

\section{Introduction}

The Freiman--Ruzsa theorem \cite{freiman1987structure,ruzsa1999analog} is a cornerstone of additive combinatorics,
with numerous applications to theoretical computer science~\cite{lovett2015exposition}.
Loosely speaking, the theorem shows that sets exhibiting approximate combinatorial subgroup behavior must be algebraically structured.
To make this precise, recall that an additive set $A$ has \emph{doubling constant}~$K$ if $|A+A|\le K|A|$, where  $A+A = \{a+a' \;:\; a,a' \in A\}$.
In the extremal case $K=1$, the set~$A$ must be a subgroup or a coset of a subgroup.
The doubling constant therefore gives a combinatorial measure of the approximate subgroup behavior of sets.

Here, we focus on subsets of $\mathbb{F}_2^n$.
In this setting, the Freiman--Ruzsa theorem states that any set $A \subseteq \F_2^n$ with doubling constant~$K$ can be covered by $F(K)$ translates of a subspace $V \leq \F_2^n$ of size $|V| \leq |A|$, where $F: \R_+ \to \R_+$ is a universal function.
The original proof of this result, due to Ruzsa \cite{ruzsa1999analog}, shows that one may take $F(K) \leq 2K^2 2^{K^4}$.
In the same paper, Ruzsa puts forward a conjecture of Marton asserting that the dependence on~$K$ can be improved to a polynomial.
This has become widely known as the \emph{Polynomial Freiman--Ruzsa} ($\PFR$) conjecture.

The $\PFR$ conjecture has sparked much research in additive combinatorics, as it became clear that this question lies at the heart of several results relating algebraic and combinatorial notions of structure; see \cite{green2004finite, green2005notes}.
The first significant improvement, due to Sanders \cite{Sanders2012}, gave a quasipolynomial bound:
$F(K) \leq \exp\big((\log K)^{4+o(1)}\big)$.
Since then, the status of the $\PFR$ conjecture remained open for over a decade.
In a recent breakthrough, the $\PFR$ conjecture was proved by Gowers, Green, Manners, and Tao~\cite{GowersGMT2025}.

\begin{theorem}[$\PFR$]
\label{thm:marton_conjecture}
    Let~$n\geq 1$ be an integer and let $A \subseteq \F_2^{n}$ be a set satisfying $|A + A| \leq K |A|$. 
    Then, there exists a subspace~$V \leq \F_2^n$ of size $|V|\leq |A|$ such that~$A$ can be covered by $\poly(K)$ translates of~$V$.
\end{theorem}

In the theorem above, and throughout the paper, we use $\poly(\cdot)$ to denote an arbitrary (positive) polynomial $P: \R_+ \to \R_+$ that does not depend on any parameters (such as the dimension~$n$).

\subsection{Algorithmic PFR}
The $\PFR$ theorem and closely-related variants play an important role in several areas of theoretical computer science, including linearity testing of maps $f \colon \F_2^n \to \F_2^m$ \cite{samorodnitsky2007low}, constructions of two-source extractors from affine extractors \cite{zewi2011affine}, communication complexity lower bounds \cite{ben2014additive},
super-polynomial lower bounds on locally decodable codes \cite{bhowmick2013new}, constructions of non-malleable codes \cite{aggarwal2014non}, sparsification algorithms for 1-in-3-SAT \cite{bedert2025strong}, quantum and classical worst-case to average-case reductions \cite{asadi2022worst,asadi2024quantum}, quantum algorithms for testing stabilizer states \cite{ad2024tolerant, bao2025tolerant,mehraban2024improved}, learning bounded stabilizer-extent quantum states \cite{SCforstabilizers}, and testing Clifford unitaries \cite{hinsche2025clifford}.

However, certain applications in theoretical computer science rely on an \emph{efficient algorithmic} statement, where an explicit description of the subspace can be learned efficiently, as opposed to an existential combinatorial statement.
For instance, while the Freiman--Ruzsa theorem plays a crucial role in the Quadratic Goldreich--Levin algorithm of Tulsiani and Wolf~\cite{TulsianiW2014}, the $\PFR$ theorem does not in itself imply any improvements to this algorithm because its proof does not readily translate to an efficient procedure.
Indeed, the naive brute-force algorithm that extracts the subspace runs in time exponential in the dimension~$n$.
This motivates a natural question that arises after the resolution of the $\PFR$ conjecture: 
Can the subspace~$V$ from \cref{thm:marton_conjecture} be learned efficiently?

Our main contribution answers this question affirmatively by providing an explicit algorithm that obtains a basis for the covering subspace in $\poly(n)$-time.

\begin{theorem}[Algorithmic \PFR]\label{thm:APFR}
    For every~$K \geq 1$, there exists a randomized algorithm such that the following holds.
    If~$A\subseteq \F_2^n$ is a set satisfying $|A+A| \leq K|A|$, then, with probability at least~$2/3$, the algorithm outputs a basis of a subspace~$V\leq \F_2^n$ of size $|V|\leq |A|$ such that $A$ can be covered by $\poly(K)$ translates of~$V$.
    Moreover, the algorithm uses $O(\log|A|)$ random samples from~$A$, makes $\widetilde{O}(\log^2|A|)$ queries to~$A$, and runs in time $\widetilde{O}(n^4)$.
\end{theorem}

Above, a \emph{query} to a set $A \subseteq \F_2^n$ is an evaluation of the characteristic function $\one_A(x)$ for a chosen $x \in \F_2^n$, and a \emph{random sample} from $A$ is a uniformly chosen element $a\in A$.
We use the standard asymptotic notation $f(x) = O(g(x))$ to denote that $f(x) \leq Cg(x)$ for some constant $C>0$ and all sufficiently large~$x$, and use $f(x) = \widetilde{O}(g(x))$ to mean that $f(x) \leq Cg(x) (\log g(x))^C$ for some constant $C>0$ and all sufficiently large~$x$.
Note that one must allow access to random samples from~$A$ in order to have a sub-exponential time algorithm:
the density of~$A$ inside~$\F_2^n$ can potentially be exponentially small, and one would then require $2^{\Omega(n)}$-many queries to~$A$ to hit a single element of that set.\footnote{This situation happens, for instance, in the proof of the inverse theorem for the Gowers $U^3$ norm \cite{samorodnitsky2007low}.
In that setting, the Freiman--Ruzsa (or $\PFR$) theorem is used for a ``graph'' set $\big\{(x, \phi(x)):\: x\in A\big\} \subset \F_2^{2n}$, which has density at most $2^{-n}$ inside its ambient space $\F_2^{2n}$.}
On the other hand, our algorithm only uses random samples in order to learn the linear span of~$A$, and thus access to samples from~$A$ can be replaced by access to a basis of its linear span (which might be more adequate for some applications).

As typically viewed in additive combinatorics, the doubling constant $K$ is independent of the dimension~$n$ (as bounded doubling implies structure), and in turn our asymptotic notation suppresses factors of $K$ for better readability.
Our proof gives more precise bounds:
the algorithm takes $O(\log|A|+K)$ random samples from~$A$, makes $2^{2K+O(\log^2K)} \log^2 |A| \log\log|A|$ queries to~$A$, and runs in time $K^{O(\log K)} n^4 \log n + 2^{2K+O(\log^2K)} n^3 \log n$, where we assume $K\geq 2$.

\subsection{An algorithmic polynomial Gowers inverse theorem} 
Let us briefly discuss how \cref{thm:APFR} is proved here.
A natural approach  would be to algorithmize each step in the proof of the combinatorial $\PFR$ theorem in \cite{GowersGMT2025}, which would in principle  answer the question.
However, this proof relies heavily on entropy-minimization techniques, and it is unclear whether such machinery can be transformed into efficient algorithms.
Instead, we utilize a connection to higher-order Fourier analysis, a field where the Gowers uniformity norms play a prominent role.

Given a function $f:\F_2^n\to\C$ and $a\in \F_2^n$, let $\Delta_af(x) := f(x+a)\overline{f(x)}$.
The \emph{Gowers $U^3$ norm} of~$f$ is then given by
\begin{equation*}
    \|f\|_{U^3} = \big(\Exp_{x,a,b,c \in \F_2^n}\Delta_a\Delta_b\Delta_cf(x)\big)^{\frac{1}{8}},
\end{equation*}
where we use the usual averaging notation $\E_{x\in X} f(x) := |X|^{-1} \sum_{x\in X} f(x)$.
It immediately follows from the triangle inequality that bounded functions have bounded uniformity norms:
\begin{equation*}
    \|f\|_{U^3} \leq \|f\|_{\infty}.
\end{equation*}
The extremizers of this inequality are given by (scalar multiples of) \emph{non-classical quadratic phase functions}:
functions $\psi: \F_2^n \to \C$ that satisfy
\begin{equation}\label{def:ncqphase}
    \Delta_a \Delta_b \Delta_c \psi(x) = 1 \quad \text{for all $x, a, b, c\in \F_2^n$.}
\end{equation}

For any quadratic polynomial $p: \F_2^n \to \F_2$, the function $(-1)^{p}$ is an example of a non-classical quadratic phase function,\footnote{These functions are sometimes known as \emph{classical} quadratic phase functions.}
but these are not the only examples.
If we denote by $|\cdot|: \F_2 \to \{0,1\}\subseteq \Z$ the natural identification map, then for any $c\in \F_2^n$ the function $\psi(x) = i^{|c_1x_1|+\dots+|c_nx_n|}$ will also be a non-classical quadratic phase function.
While these ``strictly non-classical'' functions do not commonly play an important role in quadratic Fourier analysis, they will be important to us later on.

The $U^3$ norm quantifies approximate quadratic structure in a function.
The so-called ``direct theorem,''  which follows from repeated applications of the Cauchy-Schwarz inequality, shows that the uniformity norms bound correlation with quadratic phases:
$|\E_{x\in \F_2^n} f(x) \overline{\psi(x)}| \leq \|f\|_{U^{3}}$ holds for any non-classical quadratic phase function~$\psi$.
Inverse theorems for the uniformity norms show that, for bounded functions, a large $U^3$ norm implies correlation with a quadratic phase \cite{samorodnitsky2007low, GreenTao2008}.
Of particular interest here are quantitative aspects. 
One of the principal motivations for proving \cref{thm:marton_conjecture} ($\PFR$) was to obtain a \emph{polynomial inverse theorem} for the~$U^3$ norm ($\PGI$, for ``polynomial Gowers inverse'').
In what follows, we say that a complex function~$f$ is \emph{1-bounded} if~$\|f\|_\infty \leq 1$.

\begin{theorem}[$\PGI$]\label{thm:pgi}
    If $f:\F_2^n\to\C$ is a 1-bounded function with $\|f\|_{U^3} \geq \gamma$, then there exists a quadratic polynomial $p:\F_2^n\to\F_2$ such that
    \begin{equation*}
        \Big|\Exp_{x\in \F_2^n}f(x)(-1)^{p(x)}\Big| \geq \poly(\gamma).
    \end{equation*}
\end{theorem}

It was shown by Green and Tao~\cite{green2010equivalence}, and independently by Lovett~\cite{lovett2012equivalence}, that $\PFR$ is equivalent to $\PGI$.
As such, the resolution of the $\PFR$ conjecture by Gowers, Green, Manners and Tao also provided a proof of Theorem~\ref{thm:pgi}.
Here, we use this equivalence \emph{in the other direction} as a bridge to reduce the proof of \cref{thm:APFR} to obtaining an algorithmic version of $\PGI$.
Since the proof of equivalence between $\PGI$ and $\PFR$ is combinatorial---as opposed to the information-theoretic proof of $\PFR$---it is easier to translate it into an algorithmic framework that provides such a bridge.

We note that algorithmic versions of the $U^3$ inverse theorem have been previously developed in \cite{TulsianiW2014, ben2014sampling}.
However, these algorithms are not guaranteed to produce sufficiently strong quadratic correlators to obtain a Freiman--Ruzsa algorithm of polynomial strength.
One of our main contributions in this work, then, is  to provide the first efficient algorithm for the $\PGI$ theorem.
In the following, a \emph{query} to a function $f:\F_2^n\to\C$ means an evaluation of~$f$ at some given point~$x\in \F_2^n$.

\begin{theorem}[Algorithmic $\PGI$] \label{thm:algoPGI}
    For every $\gamma>0$,  there exists a randomized algorithm such that the following holds. 
    If $f: \F_2^n \to \C$ is a 1-bounded function satisfying $\|f\|_{U^3} \geq \gamma$, then, with probability at least~$2/3$, the algorithm outputs a quadratic polynomial $q: \F_2^n \to \F_2$ satisfying
    \begin{equation*}
    \big|\Exp_{x\in \F_2^n} f(x) (-1)^{q(x)}\big| \geq \poly(\gamma).
    \end{equation*}
    This algorithm makes $\widetilde{O}(n^2)$ queries to~$f$ and runs in time $\widetilde{O}(n^3)$.
\end{theorem}

Here again, we have hidden the dependence of the implied constants on the parameter~$\gamma$ for better readability. An inspection of the algorithm shows that it makes $(1/\gamma)^{O(\log(1/\gamma))} n^2 \log n$ queries to~$f$, and runs in time $(1/\gamma)^{O(\log(1/\gamma))} n^3 \log n$.

\subsection{Symplectic geometry and quadratic Fourier analysis}

The proof of \cref{thm:algoPGI} (algorithmic $\PGI$) relies on a connection between symplectic geometry and quadratic Fourier analysis, which was first observed by Green and Tao~\cite{GreenTao2008}.
This observation appears in an outline of the inverse theorem for the Gowers~$U^3$ norm in a model case, namely over a finite abelian group~$G$ of odd order, for a non-classical quadratic phase function $f = e^{2\pi i\phi(x)}$ given by a map $\phi:G\to \R/\Z$.\footnote{
Though we focus on a particular group of even order here, the connection with symplectic geometry remains (see \cref{sec:symplectic}).}
In this case, the discrete derivatives of~$\phi$ turn out to satisfy an identity of the form
\begin{equation*}
    \phi(x+h) - \phi(x) =  ( Mh + c)\cdot x,
\end{equation*}
where $M:G \to \widehat{G}$ is a linear map (homomorphism) and $c\in \widehat{G}$.
This shows that the discrete derivatives of~$\phi$ resemble affine linear functions.
In this setting, the inverse problem involves finding a global description of~$f$ from this data. In other words, one must somehow integrate this identity.
This integration is possible  due to the fact that the map~$M$ can be shown to obey a self-adjointness condition of the form
\begin{equation}\label{eq:selfadjoin}
    M(y'-y)\cdot x -  Mx\cdot (y'-y)=0
    \quad
    \text{for all $x,y,y'\in G$.}
\end{equation}
Motivated by this, Green and Tao remark:

\begin{quote}
    ``There appear to be some intriguing parallels with symplectic geometry here.
    Roughly speaking, the vanishing~\eqref{eq:selfadjoin} is an assertion that the graph $\{(h,Mh) : h\in G\}$ is a ``Lagrangian manifold'' on the ``phase space'' $G\times \widehat{G}$. [\dots] Thus we see hints of some kind of ``combinatorial symplectic geometry'' emerging, though we do not see how to develop these possible connections further.''
\end{quote}

In Section~\ref{sec:symplectic}, we expand on this observation by providing a number of results showing that quadratic Fourier analysis and symplectic geometry are indeed tightly intertwined.
Our starting point is the fact that Hilbert spaces form the natural analytic space for the~$U^3$ norm~\cite{eisner2012large}.
In this setting, the multiplicative derivatives of a function $f\in L^2(\F_p^n)$ give rise to a natural probability distribution on the phase space $V = \F_p^n \times \F_p^n$, which is closely connected to the Heisenberg group over~$V$.
Quadratic structure in~$f$ is then reflected in this distribution as a bias towards isotropic sets in~$V$ (with respect to the standard symplectic form).
An extremal instance of this phenomenon is that the extremizers of the $U^3$ norm relative to the~$L^2$ norm can be characterized in terms of maximal isotropic subspaces of~$V$ (i.e., Lagrangian manifolds).
This implies, for instance, that the unitary isometry group of the~$U^3$ norm modulo the Heisenberg group is isomorphic to the symplectic group~$\Sp(V)$;
see \cref{sec:symplectic} for details.
The central component of our $\PGI$ algorithm operates on the phase space~$V$, and is most naturally expressed in this context.

\subsection{Algorithms for approximate algebraic structure}
Our work provides a general framework that yields a number of  algorithms for learning algebraic structure in sets and functions.
These algorithms naturally fall into two categories.

The first category falls within the scope of set addition.
This includes \cref{thm:APFR} (algorithmic $\PFR$), as well as algorithms for learning approximations of functions between boolean hypercubes by homomorphisms (see \cref{thm:algoPFR2'} and \cref{thm:algoPFR3'}).

The second category is related to quadratic Fourier analysis. This includes \cref{thm:algoPGI} (algorithmic $\PGI$), as well as the following quadratic analogue of the Goldreich--Levin algorithm~\cite{GoldreichLevin1989} and its corollaries.

\begin{theorem}[Quadratic Goldreich--Levin algorithm] \label{thm:quadratic_GL}
For every $\eps,\delta>0$, there exists a randomized algorithm such that the following holds.
Given query access to a function $f: \F_2^n \rightarrow \C$, with probability at least~$1 - \delta$, the algorithm outputs a quadratic polynomial $p: \F_2^n \rightarrow \F_2$ satisfying
$$\big|\Exp_{x \in \F_2^n} f(x) (-1)^{p(x)}\big| \geq \max_{q \text{ quadratic }} \big|\Exp_{x \in \F_2^n} f(x) (-1)^{q(x)}\big| - \eps.$$
This algorithm makes $n^2 \log n \log(1/\delta) (1/\varepsilon)^{O(\log(1/\varepsilon))}$ queries to $f$ and runs in time $n^3 \log n \log(1/\delta) (1/\varepsilon)^{O(\log(1/\varepsilon))}$.
\end{theorem}

Earlier works on Quadratic Goldreich--Levin theorems were based on algorithmic proofs of the inverse theorem for the~$U^3$ norm.
Given a function~$f$ whose maximal correlation with a quadratic phase~$(-1)^q$ is at least~$\tau>0$, these algorithms produce a quadratic phase that has correlation either $\exp(-\poly(1/\tau))$ \cite{TulsianiW2014} or $\exp(-\poly\log(1/\tau))$ \cite{ben2014sampling}.
\cref{thm:quadratic_GL} shows that this loss in correlation can be avoided almost entirely.
(Note that it would be impossible to guarantee the exact optimal correlator using only a polynomial number of queries to~$f$.)

\cref{thm:quadratic_GL} in fact plays a central role in proving our main results.
Theorem~\ref{thm:algoPGI} (algorithmic $\PGI$)  follows immediately by combining \cref{thm:pgi} ($\PGI$) and \cref{thm:quadratic_GL}.
As further applications, we obtain an optimal self-correction algorithm for quadratic Reed-Muller codes over~$\F_2$ (\cref{cor:ReedMuller}) and an algorithm for quadratic structure-versus-randomness decompositions (\cref{cor:decomposition}).
For the proof of \cref{thm:APFR} (algorithmic $\PFR$), we  also use \cref{thm:quadratic_GL}, rather than the closely-related algorithmic $\PGI$~theorem.

Our proof of \cref{thm:quadratic_GL} crucially relies on a connection to quantum information theory.
Namely, it can be viewed as a ``dequantization'' of a result of Chen, Gong, Ye, and Zhang~\cite{chen2024stabilizer}, who gave an efficient quantum protocol for learning the stabilizer state closest to a given quantum state. We capitalize on the close connection between stabilizer states and quadratic phase functions to obtain a classical analogue of their~result.

\subsection{Quantum algorithms.}

Since our classical algorithm for $\PFR$ is obtained by dequantizing a \emph{quantum} algorithm for learning stabilizer states, it is natural to ask whether any advantage can be retained by working directly in the quantum setting. 
In \cref{sec:quantum}, we present a quantum algorithm for this same task whose query and time complexities\footnote{In the context of quantum algorithms, by time we mean the total number of single and two-qubit quantum gates used in the quantum algorithm.} are both improved by a factor of~$n$ compared to their classical counterparts.
Moreover, the quantum result admits a significantly simpler proof, as we can invoke the stabilizer learning algorithm of~\cite{chen2024stabilizer} as a black box.

\begin{restatable}{theorem}{quantumPFR}
\label{thm:quantumPFR} (Quantum Algorithmic $\PFR$)
    Let $A \subseteq \mathbb{F}_2^n$ satisfy $|A + A| \leq K|A|$ for a doubling constant $K \geq 1$.
    There is an $O(n^3)$-time quantum algorithm that uses $O(\log|A|)$ random samples and $O(\log |A|)$ quantum queries to~$A$ which, with probability at least $2/3$, returns a subspace $V \leq \mathbb{F}_2^n$ of size $|V| \leq |A|$ such that $A$ can be covered by $\poly(K)$ translates of~$V$.
\end{restatable}

We remark that the proof of this result is largely modular.
Accordingly, readers primarily interested in the quantum setting may proceed directly to the final section.

\subsection{Structure of the paper} Section~\ref{sec:symplectic} covers connections between symplectic geometry and quadratic Fourier analysis.
In Section~\ref{sec:finding_Lagrangians}, we give the core algorithmic primitive: finding high-weight Lagrangian subspaces.
In Section~\ref{sec:quadGL}, we use this primitive to prove the optimal Quadratic Goldreich--Levin theorem and deduce the algorithmic $\PGI$ theorem and other corollaries.
Then, in Section~\ref{sec:algo_PFR}, we derive our classical algorithmic $\PFR$ theorems.
Finally, in Section~\ref{sec:quantum} we prove the quantum algorithmic $\PFR$ theorem.

\subsection{Notation}
\label{sec:notation}

Let~$\F$ be a field.
Given vectors $a, b\in \F^n$, define their inner product by
$$a\cdot b = a_1b_1 + \dots+ a_nb_n$$
and their entry-wise product by
$$a\circ b = (a_1b_1,\, a_2 b_2,\, \dots,\, a_nb_n).$$
The linear span of a set of points $S\subseteq \F^n$ is denoted $\vspan(S)$.

We are most interested in the case where the field is~$\F_2$.
We write $|\cdot|: \F_2 \to \bset{}\subseteq\Z$ to denote the natural identification map given by $|0| = 0$ and $|1| = 1$.
For a vector $a\in \F_2^n$, let $|a| = |a_1| + \dots + |a_n|$.
Note that~$|a|$ is the usual Hamming weight of the vector~$a$.

For a finite set~$X$, we use the common averaging notation $\E_{x\in X} [f(x)] := |X|^{-1} \sum_{x\in X} f(x)$. The support of a function~$f$ is denoted by $\supp(f)$.
We say that $f: X \to \C$ is \emph{1-bounded} if $|f(x)|\leq 1$ for all $x\in X$, and denote $\D = \{z\in\C:\: |z|\leq 1\}$. 
We write~$\U(X)$ for the unitary group on~$\C^X$, and $\U(1) := \{z\in \C:\: |z|=1\}$ for the unitary group on~$\C$.
We denote by $\propto$ proportionality up to a constant in~$\C$.
For functions $f, g: X \to \C$, denote $\langle f, g\rangle = \E_{x\in X} f(x) \overline{g(x)}$ and $\|f\|_2 = \langle f, f\rangle^{1/2}$.
For a linear operator $A: \C^X \to \C^X$, we write~$A^*$ for its Hermitian conjugate:
$$\langle f, Ag\rangle = \langle A^*f, g\rangle.$$
The Hilbert--Schmidt inner product on $\C^{X\times X}$ is defined by
\begin{equation*}
    \langle A, B\rangle_{HS} = \frac{1}{|X|} \tr(A^*B).
\end{equation*}

\section{Symplectic geometry and quadratic Fourier analysis}
\label{sec:symplectic}

In this section we make explicit a connection between symplectic geometry and quadratic Fourier analysis that was first speculated by Green and Tao~\cite{GreenTao2008}.
Our goal is to develop the aspects of quadratic Fourier analysis needed for this paper (with the exception of the PFR theorem) directly from elementary arguments in symplectic geometry. 

We will primarily work in the setting of functions over~$\F_2^n$, which is the main regime of interest in this paper. 
However, the connection persists---and in fact becomes simpler---in odd characteristic spaces~$\F_p^n$. 
For this reason, at the end of each subsection we state the corresponding results in odd characteristic.
Their proofs follow by straightforward simplifications of the characteristic-two arguments and will therefore be omitted.

The connection we wish to establish becomes most transparent once we change the analytic framework in which the $U^3$ norm is considered. 
In additive combinatorics one typically studies bounded functions, such as indicator functions of sets, and inverse theorems for the uniformity norms are therefore formulated relative to the $L^\infty$ norm (as in \cref{thm:pgi}). 
However, as observed by Eisner and Tao~\cite{eisner2012large}, the Hilbert space~$L^2$ provides a more natural ambient space for the $U^3$ norm. Indeed, $L^2$ is the largest Lebesgue space in which the $U^3$ norm remains bounded independently of the ambient dimension:
\begin{equation}
    \sup\big\{\|f\|_{U^3} : \|f\|_{k}=1\big\}
    =
    \begin{cases}
        1 & \text{if $k\geq 2$,}\\
        \infty & \text{if $k<2$,}
    \end{cases}
\end{equation}
where the supremum is taken over all $n\geq 1$ and all functions $f: \F_2^n \to \C$.
Thus, the $U^3$ norm remains controlled when passing from uniformly bounded functions to~$L^2$. 
Moreover, $k=2$ is the unique exponent for which the norms $\|f\|_k$ and $\|f\|_{U^3}$ remain comparable under arbitrary dilations of the Haar measure on~$\F_2^n$.

As we will see, once the $U^3$ norm is viewed inside the Hilbert space~$L^2$, the connection between quadratic Fourier analysis and symplectic geometry emerges naturally.

\subsection{Basic notions in finite symplectic geometry}

We collect a few basic facts about linear-algebraic aspects of symplectic geometry (see for instance~\cite[Chapter~1]{Eslami_symp}).
Let~$\F$ be a finite field.
A symplectic vector space over~$\F$ is given by an $\F$-vector space~$V$ equipped with a non-degenerate, alternating bilinear map $\omega: V \times V \to \F$.
If~$V$ is finite-dimensional, then its dimension must be even, and we will typically denote it by~$2n$.
By choosing an appropriate basis for~$V$---called a \emph{symplectic basis}---we can assume~$V$ to be $\F^{2n}$ equipped with the \emph{standard symplectic inner product}:
\begin{equation*}
    [(a,b),\, (c,d)] := a\cdot d - b\cdot c \quad \text{for $(a,b),\, (c,d) \in \F^{2n}$.}
\end{equation*}
We will henceforth assume that such a symplectic basis $(e_1, \dots, e_{2n})$ has been chosen, and thus restrict our attention to the standard symplectic vector space $(\F^{2n},\, [\cdot, \cdot])$.

Note that every element $v\in \F^{2n}$ is self-orthogonal with respect to the symplectic inner product.
A subspace $U\leq \F^{2n}$ is \emph{isotropic} if it is likewise self-orthogonal, that is, if
$$[u,v] =0 \quad \text{for all $u,v\in U$.}$$
An important role is played by those isotropic subspaces that are maximal with respect to set inclusion;
they are called \emph{Lagrangian subspaces}, or \emph{Lagrangians} for short.
We denote the set of all Lagrangian subspaces of~$\F^{2n}$ by $\Lag(\F^{2n})$.

Note that Lagrangians must equal their orthogonal complement under the symplectic inner product, and therefore have dimension~$n$.
Since they are maximal isotropic subspaces, any isotropic subspace can be extended (non-uniquely) to a Lagrangian.

Lagrangian subspaces admit the following useful characterization:
a subspace $L\leq \F^{2n}$ is Lagrangian if and only if it can be written in the form
$$L = \big\{(h,\, Mh+w) : h\in V,\, w\in V^{\perp}\big\}$$
for some subspace $V\leq \F^n$ and some symmetric matrix $M\in \F^{n\times n}$.
The proof of this fact is an exercise in linear algebra, and will be omitted.
We will also use the following basic fact about complements:
for any Lagrangian $L \leq \F^{2n}$, there exists a (non-unique) complementary Lagrangian~$L'$ such that $\F^{2n} = L \oplus L'$.

Linear transformations between symplectic vector spaces that preserve the symplectic inner product are called \emph{symplectic maps}, or \emph{symplectomorphisms}.
Of particular importance to us will be invertible symplectic maps in $\GL(\F^{2n})$, which represent the automorphisms of the symplectic vector space $(\F^{2n},\, [\cdot,\cdot])$.
These maps form a group called the \emph{symplectic group}, denoted~$\Sp(\F^{2n})$.
One can show that the symplectic group acts transitively on~$\Lag(\F^{2n})$.

\subsection{The Heisenberg group}

It is possible to associate a Heisenberg group to any given symplectic vector space $(V, \omega)$.
If the characteristic of the underlying field~$\F$ is not two, this can be done in a canonical way by taking $H(V) = V \times \F$ (regarded as sets) equipped with the group operation
\begin{equation*}
    (u,s) \bullet (v,t) = \big(u+v,\, s+t+\tfrac{1}{2}\omega(u,v)\big).
\end{equation*}
This defines a central extension
$$0 \rightarrow \F \rightarrow H(V) \rightarrow V \rightarrow 0,$$
and one can easily check that the symplectic form~$\omega$ determines the commutator relations:
$$(u,s) \bullet (v,t) \bullet (u,s)^{-1} \bullet (v,t)^{-1} = (0,\, \omega(u,v)).$$

If the underlying field is~$\F_2$, however, dividing by two is disallowed and there is no canonical (basis-free) way to define a Heisenberg group.
Moreover, to preserve the main representation-theoretic properties of the Heisenberg group, one must consider a central extension of~$V$ by~$\Z_4$ rather than~$\F_2$ \cite{GurevichH2012}.
This is ultimately due to the existence of (strictly) non-classical quadratic phase functions, which take values on the fourth-roots of unity rather than $\pmset{}$;
see \cite[Section~0.2]{GurevichH2012}.

The definition of a Heisenberg group in characteristic two thus depends on the choice of a symplectic basis for~$V$.
Let us assume we are working on the vector space~$\F_2^{2n}$ with the standard symplectic inner product~$[\cdot,\cdot]$, corresponding to the symplectic basis $(e_1, \dots, e_{2n})$.
As we wish to preserve the connection between commutator relations and the symplectic inner product, it is simpler to define the associated Heisenberg group $H(\F_2^{2n})$ in terms of a \emph{group presentation}, that is, a set of generators together with the set of defining relations they satisfy (see \cite[Chapter~7.10]{artin_algebra}).

\begin{definition}[The Heisenberg group over~$\F_2$]
    Define~$H(\F_2^{2n})$ by
    \begin{align}\label{eq:Heisrelations}
        \big\langle z, w(e_1), w(e_2), \dots, w(e_{2n}) \mid\: &z^4=1,\, w(e_i)^2 = 1,\, w(e_i)z = zw(e_i),\\
        &w(e_i)w(e_j) = z^{2[e_i,e_j]}w(e_j)w(e_i) \quad \text{for $i,j\in [2n]$}\big\rangle.\nonumber
    \end{align}
\end{definition}

Note that the elements of the Heisenberg group can identified with the points in~$\F_2^{2n}$ up to powers of~$z$.
Indeed, for $x = (a,b)\in \F_2^{2n}$, let $\kappa(x) := |a\circ b|$ (recall the definition of~$|\cdot|$ from \cref{sec:notation}) and define
\begin{equation}\label{eq:wxdef}
    w(x) = z^{\kappa(x)}w(e_1)^{x_1}w(e_2)^{x_2}\cdots w(e_{2n})^{x_{2n}}.
\end{equation}
One easily checks that these elements have order~$2$ (this is the reason for adding the term~$z^{\kappa(x)}$), and that they satisfy the commutation relations
\begin{align}\label{eq:commutation}
    w(x)w(y) &= z^{2[x,y]}w(y)w(x).
\end{align}
The commutation relations imply that every element of~$H(\F_2^{2n})$ has a unique representation of the form $z^tw(x)$ for $t\in \Z_4$ and $x\in \F_2^{2n}$.
It follows that this group has order~$2^{2n+2}$, its center is $\langle z\rangle$, and (from equation~\eqref{eq:commutation}) it is 2-step nilpotent.\footnote{This means that the commutators $ghg^{-1}h^{-1}$ belong to the center for all $g, h \in H(\F_2^{2n})$.}

The Heisenberg group is a central extension of~$\F_2^{2n}$ by~$\Z_4$, in which the multiplication is given in terms of a 2-cocycle $\beta: \F_2^{2n} \times\F_2^{2n} \to \Z_4$,
\begin{equation} \label{eq:multiplication}
    w(x)w(y) = z^{\beta(x, y)} w(x+y).
\end{equation}
While we will not need an explicit formula for~$\beta$, we note it here for completeness.
For $x = (a, b) \in \F_2^{2n}$, define the projection maps $\pi_1, \pi_2: \F_2^{2n} \to \F_2^n$ by $\pi_1(x)=a$ and $\pi_2(x)=b$.
Note that $\kappa(x) = |\pi_1(x) \circ \pi_2(x)|$ in this notation.
The cocycle~$\beta$ can then be expressed as
\begin{equation} \label{eq:beta}
    \beta(x, y) = \kappa(x) + \kappa(y) - \kappa(x+y) + 2 \big|\pi_2(x)\circ \pi_1(y)\big| \mod{4},
\end{equation}
as can be checked from the definition of the elements~$w(x)$ and the relations~\eqref{eq:Heisrelations} defining the group.

\subsection{The Weyl operators}
We now introduce a unitary representation of the Heisenberg group $H(\F_2^{2n})$ known as the Weil representation.\footnote{Named after Andr\'{e} Weil.}
We will assume knowledge of the basic representation theory of finite groups, given e.g. in \cite[Chapter~10]{CSTbook2018}.

The Weil representation is given in terms of the \emph{Weyl operators},\footnote{Named after Hermann Weyl.}
which are an important notion in the theory of quantum computation and quantum error correction.
These operators can be defined in terms of two natural unitary representations of~$\F_2^n$ as operators on $L^2(\F_2^n)$:
the \emph{translations}~$\tau_a$ given by
$$(\tau_a f)(x) = f(x+a) \quad \text{for $a, x\in \F_2^n$},$$
and the \emph{characters}~$\chi_b$, whose action is given by
$$(\chi_b f)(x) = (-1)^{b\cdot x} f(x) \quad \text{for $b, x\in \F_2^n$.}$$

\begin{definition}[Weyl operators]\label{def:Weyl}
    For a pair $(a,b)\in \F_2^n\times\F_2^n$, define the linear operator
    $$W(a,b) := i^{|a\circ b|} \tau_a \chi_b.$$
    The group generated by all Weyl operators $W(u)$, $u\in \F_2^{2n}$, is called the \emph{Heisenberg--Weyl group}, and it is denoted $\HW(\F_2^{2n})$.
\end{definition}

\begin{remark}
    In terms of Pauli matrices in quantum theory, we have $W(0,0) = I$, $W(0,1) = Z$, $W(1,0) = X$ and $W(1,1) = Y$.
    Note that this is slightly different from the notation commonly used in the quantum literature, where the roles of the first component $a\in \F_2^n$ and the second component $b\in \F_2^n$ are typically reversed.
\end{remark}

The action of the Weyl operators on functions $f:\F_2^n\to \C$ is given by
\begin{equation*}
    (W(a,b)f)(x) = (-i)^{|a\circ b|}(-1)^{b\cdot x}f(x+a).
\end{equation*}
These operators are clearly unitary, and one can readily check that they square to identity and satisfy the commutation relations
\begin{equation} \label{eq:commutation_Weyl}
    W(u)W(v) = (-1)^{[u, v]} W(v) W(u) \quad \text{for all $u, v\in \F_2^{2n}$.}
\end{equation}
It follows from the defining relations~\eqref{eq:Heisrelations} of the Heisenberg group that the Weyl operators give a unitary representation $\rho: H(\F_2^{2n})\to \U(\F_2^n)$ by
\begin{equation} \label{eq:Weil}
    \quad \rho(z^t w(u)) = i^t W(u) \quad \text{for all $u\in \F_2^{2n}$, $t\in \Z_4$.}
\end{equation}
This is called the \emph{Weil representation} of the Heisenberg group, and it provides an isomorphism between the Heisenberg group~$H(\F_2^{2n})$ and the Heisenberg--Weyl group~$\HW(\F_2^{2n})$.

\begin{remark}
    As already noted by Heinrich \cite[Chapter~2]{heinrich2021}, there is a common misconception regarding the Weyl operators in the quantum literature.
    It is often assumed that
    $W(u) W(v) = i^{|\pi_1(u)\circ \pi_2(v)| - |\pi_2(u)\circ \pi_1(v)|} W(u+v)$, where $\pi_i: (u_1,u_2) \mapsto u_i$, but this formula \emph{does not} hold in general;
    this can be seen already in the case $n = 1$ by setting $u=(0,1)$ and $v=(1,0)$.
    The multiplication rule of the Weyl operators is the same as that of the Heisenberg group we defined:
    \begin{equation} \label{eq:Weyl_multiplication}
        W(u) W(v) = i^{\beta(u,v)} W(u+v),
    \end{equation}
    where~$\beta$ is the 2-cocycle given by equation~\eqref{eq:beta}.
\end{remark}

A useful property of the Weyl operators is that they form an orthonormal basis of $\C^{\F_2^n \times \F_2^n}$ under the normalized Hilbert-Schmidt inner product
\begin{equation*}
    \langle A, B\rangle_{HS} := \frac{1}{2^n} \tr(A^*B).
\end{equation*}
Indeed, it is easy to check that $\tr(W(u)) = 2^n \one[u=0]$ for all $u\in \F_2^{2n}$, and thus
$$\big\langle W(u),\, W(v)\big\rangle_{HS} = \frac{1}{2^n} \tr(W(u)W(v)) = \one[u+v=0];$$
since there are~$2^{2n}$ Weyl operators, they form an orthonormal basis.
As a consequence, for any function $f: \F_2^n \to \C$, we have 
\begin{align*}
    \big\|f\otimes \overline{f}\big\|_{HS}^2 = \sum_{u\in \F_2^{2n}} \big|\big\langle W(u),\, f\otimes \overline{f}\big\rangle_{HS}\big|^2.
\end{align*}
By the cyclic property of the trace, we conclude that $\|f\otimes \overline{f}\|_{HS}^2 = 2^n \|f\|_2^4$ and
\begin{equation}\label{eq:WeylL2}
\|f\|_2^4 = \frac{1}{2^n} \sum_{u\in \F_2^{2n}} |\langle f,\, W(u) f\rangle|^2.    
\end{equation}

\subsubsection{In odd characteristics}
One can similarly define the Weyl operators and the Weil representation of the Heisenberg group $H(\F_p^{2n})$ for odd primes~$p$.
Recall that the Heisenberg group over~$\F_p^{2n}$ for odd~$p$ has elements $\F_p^{2n} \times \F_p$ and group operation
\begin{equation*} \label{eq:Heisenberg_odd}
    (u,s) \bullet (v,t) = \big(u+v,\, s+t+\tfrac{1}{2}[u,v]\big).
\end{equation*}

Let $\omega_p = e^{2\pi i/p}$ and let $f: \F_p^n \to \C$ be a function.
For $a, b \in \F_p^n$, denote by~$\tau_a$ the translation operator $(\tau_af)(x) := f(x+a)$, and denote by~$\overline{\chi_b}$ the conjugated character operator $(\overline{\chi_b} f)(x) := \omega_p^{-b\cdot x} f(x)$.
The Weyl operators are then defined by
$$W(a,b) := \omega_p^{a\cdot b/2} \tau_a \overline{\chi_b},$$
where the division by~$2$ in the exponent is done over~$\F_p$.
The group generated by these operators is denoted $\HW(\F_p^{2n})$, and called the \emph{Heisenberg--Weyl group}.

One easily checks that
$$W(u)W(v) = \omega_p^{-[u,v]/2} W(u+v) = \omega_p^{-[u,v]} W(v)W(u)$$
for all $u, v\in \F_p^{2n}$, and thus the map $\rho_p: H(\F_p^{2n}) \to \U(\F_p^n)$ given by $\rho_p(v,t) = \omega_p^{-t} W(v)$ defines a unitary representation.
This is the \emph{Weil representation} in odd characteristic~$p$, and provides an isomorphism from $H(\F_p^{2n})$ to $\HW(\F_p^{2n})$.

As in characteristic two, one can show that the Weyl operators form an orthonormal basis of $\C^{\F_p^n \times \F_p^n}$ under the normalized Hilbert-Schmidt inner product, and that
\begin{equation*}
\|f\|_2^4 = \frac{1}{p^n} \sum_{u\in \F_p^{2n}} |\langle f,\, W(u) f\rangle|^2.    
\end{equation*}

\subsection{The $U^3$ norm via Weyl operators}

The Weyl operators (and thus the Heisenberg group) naturally appear when studying the $U^3$ norm.
Indeed, note that
\begin{equation}\label{eq:FourierWeyl}
\widehat{\Delta_a f}(b) = \langle \chi_b,\, (\tau_a f) \overline{f}\rangle = \langle f,\, (\tau_a f) \overline{\chi_b}\rangle = i^{|a\circ b|} \langle f,\, W(a,b)f\rangle.
\end{equation}
From the simple (and well-known) identity
\begin{equation} \label{eq:U3Fourier}
    \|f\|_{U^3}^8 = \E_{a\in \F_2^n} \sum_{b\in \F_2^n} \big|\widehat{\Delta_a f}(b)\big|^4,
\end{equation}
we conclude that
\begin{equation} \label{eq:U3Weyl}
\|f\|_{U^3}^8 = \frac{1}{2^n} \sum_{u\in \F_2^{2n}} |\langle f,\, W(u) f\rangle|^4.
\end{equation}
The~$U^3$ norm of a function~$f$ can thus be defined solely in terms of its self correlation when acted upon by the Weyl operators.
This will be a more convenient expression for our purposes.

Note that the connection between the~$U^3$ and~$L^2$ settings is made clearer when the~$U^3$ norm is expressed in this form, given the presence of the inner product and the unitaries~$W(u)$, and it helps explain why~$L^2$ is the ``right'' analytic space for quadratic Fourier analysis.
The inequality $\|f\|_{U^3} \leq \|f\|_2$ easily follows from equations~\eqref{eq:U3Weyl} and~\eqref{eq:WeylL2} using the Cauchy-Schwarz inequality:
\begin{align*}
    \|f\|_{U^3}^8 &\leq \frac{1}{2^n} \Big(\max_{u\in \F_2^{2n}} |\langle f,\, W(u) f\rangle|^2\Big) \sum_{u\in \F_2^{2n}} |\langle f,\, W(u) f\rangle|^2 \\
    &= \Big(\max_{u\in \F_2^{2n}} |\langle f,\, W(u) f\rangle|^2\Big) \|f\|_2^4 \\
    &\leq \|f\|_2^8.
\end{align*}

\begin{remark}
    By the cyclic property of the trace, equation~\eqref{eq:U3Weyl} can be rewritten as
    $$\|f\|_{U^3}^8 = \frac{1}{2^n} \sum_{u\in \F_2^{2n}} \big|\big\langle W(u),\, f\otimes \overline{f} \big\rangle_{HS} \big|^4.$$
    Thus, $\|f\|_{U^3}^2$ equals (up to normalization) the $\ell^4$ norm of $f\otimes \overline{f}$ written in the Weyl basis;
    this is reminiscent of the well-known fact that $\|f\|_{U^2}$ equals the $\ell^4$ norm of $f$ written in the Fourier basis.
\end{remark}

From identity~\eqref{eq:U3Weyl} above, we will extract two results connecting the $U^3$ norm with symplectic geometry:
(i) the extremizers of the $U^3$ norm are naturally associated with Lagrangian subspaces;
(ii) the isometries of the $U^3$ norm are naturally associated with symplectic maps.
We will then see how the inverse theorem for the $U^3$ norm relates to the ``characteristic weight'' of Lagrangian subspaces, and point out some instances where the notions discussed here have implicitly appeared in earlier works by Gowers and by Green and Tao.

\subsubsection{In odd characteristics}
Everything given in this subsection holds similarly over~$\F_p^n$, with only trivial modifications.
For instance, equation~\eqref{eq:FourierWeyl} now becomes
$$\widehat{\Delta_a f}(b) = \omega_p^{a\cdot b/2} \langle f,\, W(a,b)f\rangle,$$
and the~$U^3$ norm can be expressed as
$$\|f\|_{U^3}^8 = \frac{1}{p^n} \sum_{u\in \F_p^{2n}} |\langle f,\, W(u) f\rangle|^4.$$

\subsection{Extremizers of the $U^3$ norm} \label{sec:extremizers}

Recall that the extremizers of the~$U^3$ norm relative to the~$L^\infty$ norm are given by non-classical quadratic phase functions.
Relative to~$L^2$, the extremizers of the~$U^3$ norm form a larger set of functions \cite[Theorem~1.4]{eisner2012large}
(see \cref{lem:stab_explicit} for an explicit description of them).
The connection between the $U^3$ norm and the Heisenberg--Weyl group explained in the previous subsection
enables us to identify these extremizers with those functions known in quantum information theory as \emph{stabilizer states}~\cite{ad2024tolerant}.
For this reason, we will refer to them as such.

\begin{definition}[Stabilizer states]
A function $\phi:\F_2^n\to \C$ is a stabilizer state if it satisfies $\|\phi\|_{2} = \|\phi\|_{U^3} = 1$.
We denote the set of stabilizer states by~$\Stab(\F_2^n)$.
\end{definition}

Below, we will treat the notion of a stabilizer state projectively in that we will tacitly identify stabilizer states that differ by a global phase factor~$e^{i\theta}$.

Inverse theorems for the $U^3$ norm under $L^2$ normalization were obtained in the context of quantum property testing \cite{ad2024tolerant, bao2025tolerant, mehraban2024improved}.
Roughly speaking, they show that a function $f: \F_2^n \to \C$ with $\|f\|_2 \leq 1$ has high $U^3$ norm if and only if it correlates well with a stabilizer state.
This motivates a better study of stabilizer states in the context of quadratic Fourier analysis, which will further reinforce its ties with symplectic geometry.

The next result establishes a basic connection between stabilizer states and Lagrangian subspaces:

\begin{proposition}\label{prop:stab_Lagrangian}
    A function $\phi: \F_2^n \to \C$ is a stabilizer state if and only if there exists a Lagrangian subspace $L \leq \F_2^{2n}$ such that
    \begin{equation} \label{eg:Lphi}
        \big|\widehat{\Delta_a\phi}(b)\big| =
        \begin{cases}
			1 & \text{if $(a, b) \in L$,}\\
            0 & \text{if $(a, b) \notin L$.}
		 \end{cases}
    \end{equation}
\end{proposition}

\begin{proof}
The forward direction follows easily from Parseval's identity and identity~\eqref{eq:U3Fourier}:
\begin{align*}
    \|\phi\|_2^4 = \E_{a\in \F_2^n} \|\Delta_a\phi\|_2^2 = \E_{a\in \F_2^n} \sum_{b\in \F_2^n} \big|\widehat{\Delta_a\phi}(b)\big|^2 = \frac{|L|}{2^n} = 1, \\
    \|\phi\|_{U^3}^8 = \E_{a\in \F_2^n} \sum_{b\in \F_2^n} \big|\widehat{\Delta_a\phi}(b)\big|^4 = \frac{|L|}{2^n} = 1.
\end{align*}

For the reverse direction, suppose that~$\phi$ is a stabilizer state and define the set 
\begin{equation*}
    S = \big\{(a,b)\in \F_2^{2n} : |\widehat{\Delta_a\phi}(b)| = 1\big\}.
\end{equation*}
It follows from equations~\eqref{eq:WeylL2}, \eqref{eq:FourierWeyl} and~\eqref{eq:U3Weyl} that
$$\frac{1}{2^n} \sum_{a, b \in \F_2^n} |\widehat{\Delta_a\phi}(b)|^2 = 1 = \frac{1}{2^n} \sum_{a, b \in \F_2^n} |\widehat{\Delta_a\phi}(b)|^4,$$
from which we conclude that $|S| = 2^n$.
From~\eqref{eq:FourierWeyl}, it follows that for each $(a,b)\in S$, there is a phase $\sigma_{a,b}\in \U(1)$ such that $W(a,b)f = \sigma_{a,b}f$.
In turn, this gives that the Weyl operators $W(a,b)$ with $(a,b)\in S$ pairwise commute.
Equivalently, the set~$S$ is isotropic.
Since $|S|=2^n$, we conclude that~$S$ is a Lagrangian subspace, as desired.
\end{proof}

This last result shows that each stabilizer state is associated with a unique Lagrangian subspace.
We denote the Lagrangian associated with a given stabilizer state~$\phi$ by $\Lcal(\phi)$, so that equation~\eqref{eg:Lphi} can be rewritten as
$$\big|\widehat{\Delta_a\phi}(b)\big| = \one_{\Lcal(\phi)}(a,b) \quad \text{for all $a, b \in \F_2^n$ and all $\phi \in \Stab(\F_2^n)$.}$$
We will next show that every Lagrangian subspace gives rise to stabilizer states, and that those stabilizer states associated with each given Lagrangian form an orthonormal basis of $L^2(\F_2^n)$.

\begin{proposition} \label{prop:Lag_stab_basis}
    Let $L\in \Lag(\F_2^{2n})$ be a Lagrangian subspace and let $u_1,\dots,u_n\in L$ be a basis for~$L$.
    For any choice of signs $\sigma_1,\dots,\sigma_n\in \pmset{}$, there exists a unique (up to phases) stabilizer state~$\phi\in \Stab(\F_2^n)$ satisfying $\mathcal L(\phi) = L$ and
    $$W(u_i)\phi = \sigma_i\phi \quad \text{for all $i\in [n]$.}$$
    Moreover, the set
    $$\Stab_L := \big\{\phi\in \Stab(\F_2^n):\: \Lcal(\phi)=L\big\}$$
    forms (scalar multiples of) an orthonormal basis of $L^2(\F_2^n)$. 
\end{proposition}

\begin{proof}
Let $\Weyl(L) := \{W(u):\: u\in L\}$ be the set of Weyl operators associated to elements in the Lagrangian~$L$, and note that the operators inside this set pairwise commute.
We will first show that $\Weyl(L)$ admits a unique (up to phases) orthonormal basis of joint eigenvectors, and then show that this basis corresponds to the set $\Stab_L$.

Since the operators in $\Weyl(L)$ are unitary (hence normal) and pairwise commute, existence of a common orthonormal basis of eigenvectors is guaranteed by the spectral theorem.
To prove uniqueness of this basis, let $\{u_1, \dots, u_n\}$ be a basis of the Lagrangian~$L$.
Note that the set $\Weyl(L)$ is, up to phases, generated by the operators $\{W(u_i):\: i\in [n]\}$ under multiplication.
The common eigenvectors of this generating set will then also be common eigenvectors of the larger set $\Weyl(L)$.

The operators $W(u_i)$ are unitary and Hermitian, and so their eigenvalues are $\pmset{}$ and the associated eigenspace projectors are
$$\Pi_{u_i}^{-} = \frac{I-W(u_i)}{2} \quad \text{and} \quad \Pi_{u_i}^{+} = \frac{I+W(u_i)}{2}.$$
For any $\sigma\in \pmset{n}$, the projector onto the common eigenspace of $\{W(u_i):\: i\in [n]\}$ corresponding to eigenvalue~$\sigma_i$ for each~$W(u_i)$ is
$$\Pi^{\sigma} := \prod_{i=1}^n \frac{I + \sigma_i W(u_i)}{2}.$$
(The order of the product does not matter since the terms commute.)
The dimension of this common eigenspace is
\begin{align*}
    \tr(\Pi^{\sigma})
    &= \frac{1}{2^n} \tr\bigg(\sum_{a\in \F_2^n} \prod_{i=1}^n \sigma_i^{a_i} W(u_i)^{a_i} \bigg) \\
    &= \frac{1}{2^n} \sum_{a\in \F_2^n} \bigg(\prod_{i=1}^n \sigma_i^{a_i}\bigg) \tr\bigg( \prod_{i=1}^n W(u_i)^{a_i} \bigg) \\
    &= \frac{1}{2^n} \sum_{a\in \F_2^n} \bigg(\prod_{i=1}^n \sigma_i^{a_i}\bigg) \cdot 2^n \one\{a=0\} \\
    &= 1,
\end{align*}
where we used the fact that $\tr(W(u)) = 2^n \one\{u=0\}$.
Since these eigenspaces for different choices of $\sigma\in \pmset{n}$ are orthogonal, it follows that the joint eigenspaces of $\Weyl(L)$ decompose $L^2(\F_2^n)$ into~$2^n$ pairwise-orthogonal one-dimensional subspaces.
In other words, $\Weyl(L)$ admits a unique orthonormal basis of common eigenvectors (up to phases).

We now relate this basis to the stabilizer states in $\Stab_L$.
By \cref{prop:stab_Lagrangian}, the set $\Stab_L$ corresponds to those functions~$\phi$ that satisfy
$$|\langle\phi,\, W(a,b)\phi\rangle| = \big|\widehat{\Delta_a\phi}(b)\big| = \one_{L}(a,b),$$
where we used equation~\eqref{eq:FourierWeyl} for the first equality.
On the other hand, a unit-norm function~$\phi$ is a joint eigenvector of $\Weyl(L)$ if and only if it satisfies
$$|\langle\phi,\, W(u)\phi\rangle| = |\langle\phi, \phi\rangle| = 1 \quad \text{for all $u\in L$;}$$
since $|L| = 2^n$ and $\|\phi\|_2 = 1$ by assumption, equation~\eqref{eq:WeylL2} implies that
$$|\langle\phi,\, W(v)\phi\rangle| = 0 \quad \text{for all $v\notin L$.}$$
Comparing these conditions completes the proof.
\end{proof}

This result shows that each Lagrangian is naturally associated with an orthonormal basis composed of stabilizer states.
We will refer to such bases as \emph{single-Lagrangian bases}.
Note, however, that not every orthonormal stabilizer basis is of this type:
for instance, the stabilizer states $2\cdot \one_{(0,0)}$, $2\cdot \one_{(0,1)}$, $\sqrt{2} (\one_{(1,0)}+ \one_{(1,1)})$, $\sqrt{2} (\one_{(1,0)}- \one_{(1,1)})$ form an orthonormal basis of $L^2(\F_2^2)$ with two distinct Lagrangians.

An interesting consequence of the last two results is that we can identify the set $\Stab(\F_2^n)/\U(1)$ of stabilizer states (up to phases) with the set
$$\mathrm{LC}(\F_2^{2n}) := \big\{(L, \chi):\: L\in \Lag(\F_2^{2n}),\, \chi\in \widehat{L}\big\}$$
of Lagrangian-character pairs, as we now show.
Let $\B(L) = \{u_1, \dots, u_n\}$ be a basis for a given Lagrangian subspace~$L$.
By \cref{prop:Lag_stab_basis}, for each character $\chi \in \widehat{L}$, there exists a unique (up to phases) stabilizer state $\phi_\chi \in \Stab(\F_2^n)$ that satisfies
$$W(u_i) \phi_\chi = \chi(u_i) \phi_\chi \quad \text{for all $i\in [n]$;}$$
moreover, those are all stabilizer states whose associated Lagrangian is~$L$.
These~$n$ relations specify all other eigenvalues associated with~$\phi_\chi$:
for any $u\in L$, write $u = a_1u_1 + \dots + a_nu_n$ and
\begin{equation} \label{eq:def_gammaB}
    W(u) = i^{\gamma_{\B(L)}(u)} \prod_{j=1}^n W(u_i)^{a_i},
\end{equation}
where $\gamma_{\B(L)}: L \to \Z_4$ is a (basis-dependent) function specified by the multiplication rule~\eqref{eq:Weyl_multiplication} of the Weyl operators.
Then
$$W(u) \phi_\chi = i^{\gamma_{\B(L)}(u)} \bigg(\prod_{j=1}^n \chi(u_i)^{a_i}\bigg) \phi_\chi = i^{\gamma_{\B(L)}(u)} \chi(u) \phi_\chi$$
holds for all $u\in L$, from which we conclude that
$$\langle\phi_\chi,\, W(v)\phi_\chi\rangle = \one_L(v) i^{\gamma_{\B(L)}(v)} \chi(v) \quad \text{for all $v\in \F_2^{2n}$.}$$
Since the Weyl operators form an orthonormal basis, it follows that
\begin{align*}
    \phi_\chi \otimes \overline{\phi_\chi}
    &= \sum_{u\in \F_2^{2n}} \big\langle W(u),\, \phi_\chi \otimes \overline{\phi_\chi}\big\rangle_{HS} W(u) \\
    &= \sum_{u\in \F_2^{2n}} \langle \phi_\chi,\, W(u)^* \phi_\chi \rangle W(u) \\
    &= \sum_{u\in L} i^{\gamma_{\B(L)}(u)} \chi(u) W(u).
\end{align*}
This decomposition is unique since the Weyl operators are linearly independent.
The promised identification can now be made precise:

\begin{definition}[Identification $\simeq_\B$] \label{def:identification}
    Fix a basis~$\B(L)$ for each Lagrangian $L\in \Lag(\F_2^{2n})$.
    For a character $\chi\in \widehat{L}$, we write $\phi \simeq_\B (L, \chi)$ to denote that
    \begin{equation}\label{eq:stab_id}
    \phi \otimes \overline{\phi} = \sum_{u\in L} i^{\gamma_{\B(L)}(u)} \chi(u) W(u),
    \end{equation}
    where $\gamma_{\B(L)}: L \to \Z_4$ is the function defined by~\eqref{eq:def_gammaB}.
\end{definition}

By the discussion above, once the bases $\B(L)$ are specified, each stabilizer state~$\phi$ can be written in the form~\eqref{eq:stab_id} for a unique Lagrangian~$L$ and character $\chi \in \widehat{L}$.
Moreover, \cref{prop:stab_Lagrangian} shows that every function~$\phi$ that can be written in the form~\eqref{eq:stab_id} is a stabilizer state.
The relation~$\simeq_\B$ thus gives a bijection between $\Stab(\F_2^n)/\U(1)$ and the set
$\mathrm{LC}(\F_2^{2n}) = \big\{(L, \chi):\: L\in \Lag(\F_2^{2n}),\, \chi\in \widehat{L}\big\}$
of Lagrangian-character pairs.

The action of the Weyl operators on stabilizer states is simple to describe using this identification:
if $\phi \simeq_\B (L,\chi)$, then for any $v\in \F_2^{2n}$ we have
\begin{align*}
    \big(W(v)\phi\big)\otimes \overline{\big(W(v)\phi\big)}
    &=
    \sum_{u\in L} i^{\gamma_{\B(L)}(u)} \chi(u) W(v)W(u)W(v)^*\\
    &=
    \sum_{u\in L} i^{\gamma_{\B(L)}(u)} \chi(u)(-1)^{[v,u]} W(u),
\end{align*}
where we used the commutation relation~\eqref{eq:commutation_Weyl} for the second equality.
It follows that
\begin{equation} \label{eq:Weyl_action}
    \phi \simeq_\B (L,\, \chi) \implies W(v)\phi \simeq_\B \big(L,\, (-1)^{[v, \,\cdot\,]}\chi \big).
\end{equation}
As the functions $(-1)^{[v, \,\cdot\,]}\chi$ (defined over~$L$) correspond precisely to the characters in~$\widehat{L}$, we conclude that the single-Lagrangian basis $\Stab_L$ associated with~$L$ is identified with the set $\{(L, \chi):\: \chi\in \widehat{L}\}$, and that the Heisenberg--Weyl group acts transitively on each such basis.

Finally, the identification~$\simeq_\B$ also allows us to compute the inner products of any two given stabilizer states, as long as the bases corresponding to their Lagrangians are compatible:

\begin{proposition} \label{prop:inner_product_intersection}
    Let $\phi \simeq_\B (L, \chi)$ and $\phi' \simeq_\B (L', \chi')$ be two stabilizer states, with associated Lagrangian bases $\B(L)$ and $\B(L')$.
    Suppose that $\B(L) \cap \B(L')$ forms a basis of the intersection subspace $L\cap L'$.
    Then
    \begin{equation} \label{eq:intersection}
    |\langle\phi, \phi'\rangle| = \sqrt{\frac{|L\cap L'|}{2^n}} \one\{\chi_{|L\cap L'} = \chi'_{|L\cap L'}\}.
\end{equation}
\end{proposition}

\begin{proof}
If $\phi \simeq_\B (L, \chi)$ and $\phi' \simeq_\B (L', \chi')$, then we have
\begin{align*}
    |\langle\phi, \phi'\rangle|^2
    &= \frac{1}{2^n} \big\langle \phi\otimes\overline{\phi},\, \phi'\otimes\overline{\phi'} \big\rangle_{HS} \\
    &= \frac{1}{2^n} \sum_{u\in L} \sum_{v\in L'} i^{-\gamma_{\B(L)}(u)+ \gamma_{\B(L')}(v)} \overline{\chi(u)} \chi'(v) \big\langle W(u),\, W(v)\big\rangle_{HS} \\
    &= \frac{1}{2^n} \sum_{u\in L\cap L'} i^{-\gamma_{\B(L)}(u)+ \gamma_{\B(L')}(u)} \overline{\chi(u)} \chi'(u).
\end{align*}
The assumption that $\B(L) \cap \B(L')$ forms a basis of $L\cap L'$ implies that $\gamma_{\B(L)}(u) = \gamma_{\B(L')}(u)$ on $L\cap L'$.
It then follows that
$$|\langle\phi, \phi'\rangle|^2 = \frac{1}{2^n} \sum_{u \in L\cap L'} \overline{\chi(u)} \chi'(u) = \frac{|L\cap L'|}{2^n} \one\{\chi_{|L\cap L'} = \chi'_{|L\cap L'}\}$$
by the orthogonality of characters, as wished.
\end{proof}

This last result allows for a convenient characterization of single-Lagrangian bases, $\Stab_L$ for $L\in\Lag(\F_2^n)$, that relies only on their correlations with stabilizer states.
It follows from this result that two stabilizer states have correlation that is either zero or a half-integer power of~2.
Say that an orthonormal basis~$\mathcal F$ of~$L^2(\F_2^n)$ composed of stabilizer states is \emph{regular} if, for any stabilizer state $\phi \in\Stab(\F_2^n)$, there is a $k\in \N$ such that
\begin{equation*}
    \{|\langle \phi,\phi'\rangle|^2 : \phi'\in \mathcal F\} = \{0,2^{-k}\}.
\end{equation*}

\begin{lemma} \label{lem:signature}
    An orthonormal stabilizer basis~$\mathcal F\subseteq\Stab(\F_2^n)$ of~$L^2(\F_2^n)$ is a single-Lagrangian basis if and only if~$\mathcal F$ is regular.
\end{lemma}

\begin{proof}
It follows from \cref{prop:inner_product_intersection} that any single-Lagrangian basis is regular.
It thus suffices to show that any stabilizer basis that is not a single-Lagrangian basis is not regular.
To this end, let $\mathcal F \subseteq \Stab(\F_2^n)$ be a stabilizer basis and suppose that $\phi,\phi'\in \mathcal F$ have distinct Lagrangians $L = \mathcal L(\phi)$ and $L' = \mathcal L(\phi')$.
Let $\mathcal B(L')$ be a basis for~$L'$ and $\chi\in \widehat L'$ be such that $\phi'\simeq_{\mathcal B} (L',\mathcal \chi')$.
Denote by $L'' \in \Lag(\F_2^{2n})$ a complementary Lagrangian for~$L$, so that $\F_2^{2n} = L\oplus L''$.
Then $L\cap L'' = \{0\}$ and $L'\cap L'' \neq \{0\}$.
Let $\mathcal B(L'')$ be a basis for~$L''$ that agrees with~$\mathcal B(L')$ on the intersection $L'\cap L''$.
Let $\chi''\in \widehat L''$ be such that $\chi'_{L'\cap L''} = \chi''_{L'\cap L''}$, and let $\psi\in \Stab(\F_2^n)$ be a stabilizer state such that $\psi \simeq_{\mathcal B}(L'',\chi'')$.
By \cref{prop:inner_product_intersection}, we have that $|\langle\phi,\psi\rangle|^2 = 2^{-n}$ and $|\langle\phi',\psi\rangle|^2 > 2^{-n}$, showing that~$\mathcal F$ is not regular.
\end{proof}

\subsubsection{In odd characteristics}

We similarly define stabilizer states over~$\F_p^n$ for an odd prime~$p$ as the unit-$L^2$-norm extremizers of the $U^3$ norm:
$$\phi \in \Stab(\F_p^n) \quad \text{if} \quad \|\phi\|_{U^3} = \|\phi\|_2 = 1.$$
As in the characteristic-2 setting, one can show that $\phi \in \Stab(\F_p^n)$ if and only if there exists a Lagrangian $L\in \Lag(\F_p^{2n})$ such that
$$\big|\widehat{\Delta_a\phi}(b)\big| = \one_L(a,b) \quad \text{for all $a, b\in \F_p^n$.}$$
This is the \emph{Lagrangian associated with~$\phi$}, and is denoted~$\Lcal(\phi)$.

The theory of these extremizers in odd characteristics becomes simpler because the Weyl operators on a Lagrangian subspace form a group isomorphic to~$\F_p^n$:
$$W(u) W(v) = W(u+v) \quad \text{whenever $[u,v]=0$.}$$
This allows for a canonical (basis-free) identification~$\simeq$ between $\Stab(\F_p^n)/\U(1)$ and the set of Lagrangian-character pairs
$\mathrm{LC}(\F_p^{2n}) := \big\{(L, \chi):\: L\in \Lag(\F_p^{2n}),\, \chi\in \widehat{L}\big\}$:
write $\phi \simeq (L,\chi)$ to denote that
$$\phi \otimes \overline{\phi} = \sum_{u\in L} \chi(u) W(u).$$
Note that this is equivalent to requiring that
$$\widehat{\Delta_a \phi}(b) = \omega_p^{a\cdot b/2} \one_L(a,b) \overline{\chi(a,b)} \quad \text{for all $a, b\in \F_p^n$,}$$
and it gives a bijection between $\Stab(\F_p^n)/\U(1)$ and $\mathrm{LC}(\F_p^{2n})$.

Using this identification, it is easy to compute the inner product between two stabilizer states:
if $\phi\simeq (L, \chi)$ and $\phi' \simeq (L', \chi')$, then
$$|\langle\phi, \phi'\rangle| = \sqrt{\frac{|L\cap L'|}{p^n}} \one\{\chi_{|L\cap L'} = \chi'_{|L\cap L'}\}.$$
The action of the Weyl operators is also simple to specify:
$$\phi\simeq (L, \chi) \implies W(v)\phi \simeq \big(L,\, \omega_p^{-[v,\,\cdot\,]} \chi\big).$$

\subsection{Isometries of the $U^3$ norm}

In this subsection, we show how the symmetries of the normed vector space
$$U^3(\F_2^n) = \big(\{f: \F_2^n \to \C\},\, \|\cdot\|_{U^3}\big)$$
are related to those of the symplectic vector space $(\F_2^{2n},\, [\cdot, \cdot])$.
While this result will not be needed in our algorithms, we include it here because it provides a particularly clear connection between quadratic Fourier analysis and symplectic geometry.

To make this idea precise, let $\Iso_{U^3}(\F_2^n)$ denote the set of \emph{unitary isometries} of~$U^3(\F_2^n)$, meaning the unitary operators on~$L^2(\F_2^n)$ that leave the~$U^3$ norm invariant.
This can be regarded as the automorphism group of~$U^3(\F_2^n)$ when embedded into~$L^2(\F_2^n)$.
Our goal is to establish a connection between this automorphism group $\Iso_{U^3}(\F_2^n)$ and the symplectic group $\Sp(\F_2^{2n})$, which represents the automorphisms of the standard symplectic vector space.

It is clear that $\U(1) \leq \Iso_{U^3}(\F_2^n)$, and it immediately follows from our expression~\eqref{eq:U3Weyl} for the~$U^3$ norm that the Heisenberg--Weyl group $\HW(\F_2^{2n})$ is a subgroup of $\Iso_{U^3}(\F_2^n)$.
One can show that $\U(1)\times \HW(\F_2^{2n})$ is normal inside $\Iso_{U^3}(\F_2^n)$, and it can be regarded as a ``trivial part'' of the~$U^3$ isometries;
its action on the stabilizer states is given by equation~\eqref{eq:Weyl_action}.
For this reason, we will consider the quotient of the isometry group of the $U^3$ norm by this normal subgroup.
The main result of this subsection shows that this quotient is isomorphic to the symplectic group $\Sp(\F_2^{2n})$.

In order to prove this isomorphism, we will first need to introduce a number of preliminary results.
Throughout this section, we will use the relation symbol~$\propto$ to denote proportionality. The next lemma shows that there exists a ``semi-representation'' of the symplectic group $\Sp(\F_2^{2n})$ on the unitary group $\U(\F_2^{2n})$, whose action on the Weyl operators by conjugation mimics the symplectic group up to phases.

\begin{lemma}[Semi-representation] \label{lem:semi_rep}
    For every symplectic map $S\in \Sp(\F_2^{2n})$ there exist a unitary $\sigma(S) \in \U(\F_2^n)$ satisfying
    $$\sigma(S) W(x) \sigma(S)^* \propto W(Sx) \quad \text{for all $x\in \F_2^{2n}$.}$$
\end{lemma}

\begin{proof}
We define a map $\alpha_S: H(\F_2^{2n})\to H(\F_2^{2n})$ as follows.
On the generators, set
\begin{equation*}
    \alpha_S(z) = z
    \quad\text{and}\quad
    \alpha_S\big(w(e_i)\big) = w(Se_i).
\end{equation*}
Since the elements $w(Se_i)$ have order two, it follows from equation~\eqref{eq:commutation} that~$\alpha_S$ preserves the relations~\eqref{eq:Heisrelations} defining $H(\F_2^{2n})$.
By the fundamental theorem of group presentations~\cite[Section~7.10]{artin_algebra}, we can then extend~$\alpha_S$ uniquely to an automorphism of~$H(\F_2^{2n})$, which is given by
$$\alpha_S \bigg(z^t \prod_{i=1}^{2n} w(e_i)^{x_i} \bigg) = z^t \prod_{i=1}^{2n} w(Se_i)^{x_i}.$$
(The products above are assumed to be in increasing order of $i\in [2n]$.)

From the multiplication rule~\eqref{eq:multiplication} and the formula above, we conclude there exists a map $\tau_S: \F_2^{2n} \to \Z_4$ satisfying
$$\alpha_S(z^t w(x)) = z^{t+\tau_S(x)} w(Sx) \quad \text{for all $t\in \Z_4$, $x\in \F_2^{2n}$.}$$
Denote the Weil representation by~$\rho$.
Since~$\alpha_S$ is an automorphism, the map
$$\rho_S := \rho\circ \alpha_S:\: z^t w(x) \mapsto i^{t+\tau_S(x)} W(Sx)$$
gives another unitary representation of~$H(\F_2^{2n})$.
The characters of this representation are given by
\begin{equation*}
    \chi_{\rho_S} \big(z^t w(x)\big) = \tr\big( i^{t+\tau_S(x)} W(Sx)\big) = i^{t+\tau_S(x)} 2^n \one[Sx = 0] = i^{t} 2^n \one[x = 0].
\end{equation*}
These equal the characters~$\chi_\rho$ of the Weil representation.
It then follows that these two representations are unitarily equivalent (see e.g. \cite[Chapter~10]{CSTbook2018}):
there exists a unitary $\sigma(S)\in \U(\F_2^n)$ such that
\begin{equation*}
    \sigma(S) \rho(h) \sigma(S)^* = \rho_S(h) \quad \text{for all $h\in H(\F_2^{2n})$.}
\end{equation*}
Applying this equation to the elements $h = w(x)$ gives the lemma.
\end{proof}

The lemma below gives a characterization of unitaries that act diagonally on the Weyl basis by conjugation.

\begin{lemma}[Diagonal action] \label{lem:diag_Weyl}
    Suppose that $U \in \U(\F_2^n)$ satisfies the property
    $$U W(x) U^* \propto W(x) \quad \text{for all $x\in \F_2^{2n}$.}$$
    Then, there exist $\alpha\in \U(1)$ and $v\in \F_2^{2n}$ such that $U = \alpha W(v)$.
\end{lemma}

\begin{proof}
Denote the proportionality map by~$\tau$, so that $U W(x) U^* = \tau(x) W(x)$ for all~$x$, and note that $\tau(0) = 1$.
For all $x, y\in \F_2^{2n}$ we have
$$W(x+y) = i^{-\beta(x,y)} W(x) W(y) = i^{-\beta(x,y)} W(x) U^*U W(y),$$
and thus
\begin{align*}
    \tau(x+y) W(x+y) &= U W(x+y) U^* \\
    &=i^{-\beta(x,y)} UW(x)U^* UW(y)U^* \\
    &= i^{-\beta(x,y)} \tau(x)W(x) \tau(y)W(y) \\
    &= \tau(x)\tau(y) W(x+y).
\end{align*}
We conclude that $\tau(x+y) = \tau(x)\tau(y)$ for all $x, y\in \F_2^{2n}$, and thus~$\tau$ is a character of~$\F_2^{2n}$.
We can then write $\tau(x) = (-1)^{[v, x]}$ for some $v\in \F_2^{2n}$ and all~$x$.

Now consider the unitary map $V := UW(v)$.
Since $W(v)W(x)W(v)^* = (-1)^{[v,x]} W(x)$ by the commutation relations, we conclude that
$$VW(x)V^* = (-1)^{[v,x]} UW(x)U^* = W(x) \quad \text{for all $x\in \F_2^{2n}$.}$$
It follows that~$V$ commutes with all Weyl operators~$W(x)$.
As the Weyl operators form a basis of $\U(\F_2^n)$, we conclude that $V=\alpha I$ for some $\alpha\in \U(1)$, and thus $U = \alpha W(v)$.
\end{proof}

Finally, we will need a special case of Chow's theorem from incidence geometry~\cite{Chow}.
This result shows that every automorphism of the symplectic dual polar graph is induced by a symplectic map.

\begin{theorem}[Chow's theorem]\label{thm:chow}
    Suppose $\nu:\Lag(\F_2^{2n})\to \Lag(\F_2^{2n})$ is a map with the following property:
    for every pair $L,L'\in \Lag(\F_2^{2n})$ such that $\dim(L\cap L') = n-1$, it holds that $\dim\big(\nu(L)\cap \nu(L')\big) = n-1$.
    Then, there exists a symplectic map $S\in \Sp(\F_2^{2n})$ such that, for every $L\in \Lag(\F_2^{2n})$, we have $\nu(L) = SL$.
\end{theorem}

We are now finally ready to characterize the symmetries of the normed space $U^3(\F_2^n)$ in terms of the symplectic group $\Sp(\F_2^{2n})$:

\begin{theorem}[Symmetries of~$U^3$] \label{thm:isoU3}
    Let $\sigma: \Sp(\F_2^{2n}) \to \U(\F_2^n)$ be a semi-representation in the sense of \cref{lem:semi_rep}.
    Then, $M\in \Iso_{U^3}(\F_2^n)$ if and only if it can be written in the form $M = \alpha \sigma(S)W(v)$ for some $\alpha\in \U(1)$, $S\in \Sp(\F_2^{2n})$ and $v\in \F_2^{2n}$.
    Moreover, we have the group isomorphism
    $$\Iso_{U^3}(\F_2^n)/(\U(1)\times \HW(\F_2^{2n})) \cong \Sp(\F_2^{2n}).$$
\end{theorem}

\begin{proof}
From our expression of the~$U^3$ norm in terms of Weyl operators (equation~\eqref{eq:U3Weyl}), one immediately sees that any operator of the form $M = \alpha \sigma(S)W(v)$ is a unitary isometry of~$U^3(\F_2^n)$.
For the converse, let $M\in \Iso_{U^3}(\F_2^n)$ be an arbitrary element, and note that~$M$ must map stabilizer states to stabilizer states.

We first show that this isometry induces a map $\nu_M: \Lag(\F_2^{2n}) \to \Lag(\F_2^{2n})$ on the Lagrangian subspaces.
\cref{prop:Lag_stab_basis} shows that, to each Lagrangian~$L$, we can associate a single-Lagrangian basis $\Stab_L \subset \Stab(\F_2^n)$.
By unitarity,~$M$ maps this basis to another orthonormal basis composed of stabilizer states.
Moreover, since~$\Stab_L$ is regular (in the sense of \cref{lem:signature}), so is~$M\Stab_L$.
Hence, by \cref{lem:signature}, there exists a Lagrangian~$L'$ such that
\begin{equation*}
    M\Stab_L = \Stab_{L'}.
\end{equation*}
We thus obtain a map~$\nu_M$ given by $\nu_M(L) = L'$, and note that this map satisfies
$$\nu_M(\Lcal(\phi)) = \Lcal(M\phi) \quad \text{for all $\phi\in \Stab(\F_2^n)$.}$$

We now show that this map~$\nu_M$ preserves intersection sizes of the Lagrangians.
Let $L, L'$ be two Lagrangians and choose bases $\B(L)$ and $\B(L')$ for them in such a way that $\B(L) \cap \B(L')$ forms a basis of their intersection $L\cap L'$.
Consider the stabilizer states $\phi \simeq_\B (L, \one_L)$ and $\phi' \simeq_\B (L', \one_{L'})$, which by \cref{prop:inner_product_intersection} satisfy
$$|\langle\phi, \phi'\rangle|^2 = \frac{|L\cap L'|}{2^n}.$$
Since~$M$ is unitary, we have $|\langle M\phi, M\phi'\rangle|^2 = 2^{-n}|L\cap L'|$ as well, which (by \cref{prop:inner_product_intersection} again) implies
$$|\Lcal(M\phi) \cap \Lcal(M\phi')| = |L\cap L'|.$$
We conclude that $|\nu_M(L) \cap \nu_M(L')| = |L\cap L'|$ for any Lagrangians $L, L'$, as wished.

It then follows from Chow's theorem (\cref{thm:chow}) that there exists a (unique) symplectic map $S\in \Sp(\F_2^{2n})$ such that
$$\nu_M(L) = SL \quad \text{for all $L\in \Lag(\F_2^{2n})$.}$$
Now consider the unitary map $V := \sigma(S)^* M$.
Our goal is to show that $V = \alpha W(v)$ for some phase $\alpha\in \U(1)$ and some Weyl operator~$W(v)$, which will imply the first part of the theorem. For each Lagrangian~$L$, define the vector space (and algebra)
$$\A(L) := \vspan(\{W(u):\: u\in L\}).$$
Fixing any basis $\B(L)$ for~$L$, one easily shows that the set
$$\bigg\{\sum_{u\in L} i^{\gamma_{\B(L)}(u)} \chi(u) W(u):\: \chi \in \widehat{L}\bigg\}$$
forms a basis for~$\A(L)$.
Since for a stabilizer state $\phi \simeq_\B (L, \chi)$ we have
$$\phi \otimes \overline{\phi} = \sum_{u\in L} i^{\gamma_{\B(L)}(u)} \chi(u) W(u),$$
and since~$V$ induces a permutation of the stabilizer states associated with any given Lagrangian~$L$, we conclude there is some $\phi' \simeq_\B (L, \chi')$ such that
\begin{align*}
    V \bigg(\sum_{u\in L} i^{\gamma_{\B(L)}(u)} \chi(u) W(u)\bigg) V^*
    &= (V\phi) \otimes \overline{(V\phi)} \\
    &= \phi' \otimes \overline{\phi'} \\
    &= \sum_{u\in L} i^{\gamma_{\B(L)}(u)} \chi'(u) W(u) \in \A(L),
\end{align*}
and thus $V \A(L) V^* \subseteq \A(L)$ for all $L\in \Lag(\F_2^{2n})$.

We next show that the conjugation map $A \mapsto VAV^*$ acts diagonally on the Weyl basis.
For any $u\in \F_2^{2n}$ and any Lagrangian~$L$ containing~$u$, we have $W(u) \in \A(L)$ by definition.
We then conclude from the last paragraph that
$$VW(u)V^* \in \bigcap_{L:\: u\in L} V\A(L)V^* \subseteq \bigcap_{L:\: u\in L} \A(L).$$
As the intersection $\bigcap_{L:\: u\in L} L$ of all Lagrangians containing~$u$ equals $\{0, u\}$, we conclude from linear independence of the Weyl operators that
$\bigcap_{L:\: u\in L} \A(L) = \vspan(\{I, W(u)\})$.
It follows that we can write $VW(u)V^* = \alpha(u) W(u) + \beta(u) I$.
Since for $u\neq 0$ we have
$$\beta(u) = \langle I,\, VW(u)V^*\rangle_{HS} = \frac{\tr(VW(u)V^*)}{2^n} = \frac{\tr(W(u))}{2^n} = 0,$$
we conclude that
$$VW(u)V^* \propto W(u) \quad \text{for all $u\in \F_2^{2n}$,}$$
as claimed.
It then follows from Lemma~\ref{lem:diag_Weyl} that there exist $v\in \F_2^{2n}$ and $\alpha\in \U(1)$ such that $V = \alpha W(v)$, and thus $M = \alpha \sigma(S)W(v)$.
This concludes the proof of the first part of the theorem.

For the second part we note that, for all $S, T\in \Sp(\F_2^{2n})$, the unitary $\sigma(S)\sigma(T) \sigma(ST)^*$ acts diagonally on the Weyl basis by conjugation:
$$\sigma(S)\sigma(T) \sigma(ST)^* W(x) \sigma(ST) \sigma(T)^*\sigma(S)^* \propto W(x) \quad \text{for all $x\in \F_2^{2n}$.}$$
It then follows from \cref{lem:diag_Weyl} that
$$\sigma(S) \sigma(T) \propto \sigma(ST) W(h_{S,T})$$
for some $h_{S,T} \in \F_2^{2n}$.
The multiplication of two elements $M = \alpha \sigma(S)W(v)$ and $M' = \alpha' \sigma(T)W(u)$ from $\Iso_{U^3}(\F_2^n)$ thus satisfies
\begin{align*}
    MM' &\propto \sigma(S)W(v) \sigma(T)W(u) \\
    &= \sigma(S)\sigma(T) \big(\sigma(T)^* W(v)\sigma(T)\big) W(u) \\
    &\propto \sigma(ST) W(h_{S,T}) W(T^{-1}v) W(u) \\
    &\propto \sigma(ST) W(T^{-1}v+u+h_{S,T}).
\end{align*}
The claimed isomorphism follows.
\end{proof}

As a simple corollary, we obtain the following characterization of the unitary isometries of~$U^3$ in terms of the Heisenberg--Weyl group:

\begin{corollary}[Normalizer] \label{cor:normalizer}
    $\Iso_{U^3}(\F_2^n)$ is the normalizer group of the Heisenberg--Weyl group in $\U(\F_2^n)$:
    $$\Iso_{U^3}(\F_2^n) = \big\{U\in \U(\F_2^n):\: U\HW(\F_2^{2n})U^{-1} = \HW(\F_2^{2n})\big\}.$$
\end{corollary}

\begin{proof}
From the definition of the $U^3$ norm in terms of Weyl operators (equation~\eqref{eq:U3Weyl}), we see that every element in the normalizer group belongs to $\Iso_{U^3}(\F_2^n)$.
For the converse, note that every element of the form $\sigma(S)W(v)$ conjugates Weyl operators to scalar multiples of Weyl operators.\footnote{The proof of \cref{lem:diag_Weyl} shows that these scalar multipliers are in $\pmset{}$.}
The claim now follows from \cref{thm:isoU3}.
\end{proof}

The normalizer of the Heisenberg--Weyl group in $\U(\F_2^n)$ is known in the quantum literature as the \emph{Clifford group}, and it is an important concept in quantum computation and quantum information theory \cite{Gottesman1997, Gross2021}.
Our result then provides a proof of the structural characterization of the Clifford group, a folklore result whose proof (in characteristic two) seems to have appeared in print only in Heinrich's thesis \cite[Chapter~4]{heinrich2021}.

Finally, we remark on the action of an element $M = \alpha \sigma(S)W(v) \in \Iso_{U^3}(\F_2^n)$ on the stabilizer states, from which one can extend to the full space $L^2(\F_2^n)$ by linearity.
The ``linear part'' $W(v)$ only acts by permuting the character associated to the stabilizer state, without changing its Lagrangian:
if $\phi \simeq_\B (L, \chi)$, then $W(v)\phi \simeq_\B (L,\, (-1)^{[v,\cdot]} \chi)$.
The ``symplectic part'' $\sigma(S)$ changes the associated Lagrangian according to~$S$:
$$\Lcal(\sigma(S)\phi) = S\Lcal(\phi).$$
However, this action also changes the character associated with~$\phi$, in a way that depends on the specific semi-representation~$\sigma$ chosen.

\subsubsection{In odd characteristics}

In the case of~$\F_p^n$ when~$p$ is an odd prime, the situation is again significantly simpler.
In this setting, there exists a projective unitary representation $\sigma: \Sp(\F_p^n) \to \U(\F_p^n)$ satisfying
$$\sigma(S) W(x) \sigma(S)^* = W(Sx) \quad \text{for all $x\in \F_p^{2n}$.}$$
A similar argument to that in the proof of \cref{thm:isoU3} shows that $M\in \Iso_{U^3}(\F_p^n)$ if and only if it can be written in the form $M = \alpha \sigma(S)W(v)$ for some $\alpha\in \U(1)$, $S\in \Sp(\F_p^{2n})$ and $v\in \F_p^{2n}$.

Moreover, the multiplication rule also becomes simpler in this setting:
since~$\sigma$ is a projective representation, we have that
$\sigma(ST) \propto \sigma(S)\sigma(T)$ for all maps $S,T$.
If $M = \alpha \sigma(S)W(v)$ and $M' = \alpha' \sigma(T)W(u)$, we conclude that
\begin{align*}
    MM' &\propto \sigma(S)W(v) \sigma(T)W(u) \\
    &= \sigma(S)\sigma(T) \big(\sigma(T)^* W(v)\sigma(T)\big) W(u) \\
    &\propto \sigma(ST) W(T^{-1}v+u).
\end{align*}
This corresponds precisely to the multiplication rule of \emph{affine symplectic maps}, meaning maps of the form $x \mapsto S(x+v)$ for $v\in \F_p^{2n}$ and $S\in \Sp(\F_p^{2n})$.
Denoting the affine symplectic group by $\ASp(\F_p^{2n})$, we conclude that
$$\Iso_{U^3}(\F_p^n)/\U(1) \cong \ASp(\F_p^{2n}) \cong \F_p^{2n} \rtimes \Sp(\F_p^{2n}).$$
We note that this isomorphism \emph{does not} hold in characteristic two, as $\Iso_{U^3}(\F_2^n)/\U(1)$ cannot be written in the form of a semidirect product between $\F_2^{2n}$ and $\Sp(\F_2^{2n})$;
this fact was shown (in the context of the Clifford group) by Heinrich \cite[Chapter~4]{heinrich2021}.

The action of the unitary isometries on the stabilizer states can be fully specified (up to phases) using the canonical identification~$\simeq$:
if $M = \alpha \sigma(S) W(v) \in \Iso_{U^3}(\F_p^n)$, then
$$\phi\simeq (L, \chi) \implies M\phi \simeq \big(SL,\, \omega_p^{-[Sv,\, \cdot\,]} \chi\circ S^{-1}\big).$$
Note that the action of~$M$ on the phases depends on the specific projective representation $\sigma: \Sp(\F_p^n) \to \U(\F_p^n)$ chosen.
Once this is specified, the action of $\Iso_{U^3}(\F_p^n)$ can be extended from $\Stab(\F_p^n)$ to the full space $L^2(\F_p^n)$ by linearity.

\subsection{Lagrangian weights and the inverse theorem}

Since the work of Gowers \cite{Gowers1998}, it has been understood that the quadratic structure of a function~$f$ is encoded in the large Fourier coefficients of its multiplicative derivatives.
In the context of the inverse theorem for the~$U^3$ norm, this motivates the following probability distribution over~$\F_2^{2n}$, which is called the \emph{characteristic distribution} in the quantum literature.

\begin{definition}[Characteristic distribution] \label{def:char_dist}
    For a nonzero function $f:\F_2^n\to \C$, define its characteristic distribution~$P_f$ over~$\F_2^{2n}$ by
    \begin{equation*}
        P_f(a,b) = \frac{1}{2^n \|f\|_2^4}|\widehat{\Delta_af}(b)|^2 \quad \text{for all $(a, b) \in \F_2^n \times \F_2^n$.}
    \end{equation*}
\end{definition}

The quadratic structure of~$f$ is reflected in the characteristic distribution as a bias towards isotropic sets.
Below, we give a number of basic results that make this precise.

The relation~\eqref{eq:FourierWeyl} expresses Fourier coefficients of multiplicative derivatives in terms of the Weyl operators.
This perspective gives rise to an ``uncertainty principle'' which places strong upper bounds on the characteristic weight of sets of pairwise symplectically non-orthogonal vectors.
Closely related to this is the fact that sets of pairwise anti-commuting Weyl operators give explicit isometric embeddings of Euclidean spaces into~$C^*$ algebras (a fundamental property of CAR algebras).

\begin{lemma}[Uncertainty principle]\label{lem:uncertainty1}
Let $x_1 = (a_1,b_1),\dots,x_k = (a_k,b_k)\in \F_2^{2n}$ be such that $[x_i,x_j]=1$ for all $i\ne j$.
Then, for any function $f:\F_2^n\to \C$, we have that
\begin{equation*}
    \sum_{i=1}^k\big|\widehat{\Delta_{a_i}f}(b_i)\big|^2\leq \|f\|_2^4.
\end{equation*}
\end{lemma}

\begin{proof}
    For each $i\in [k]$, let $\alpha_i = \langle f,W(x_i)f\rangle$. Since the Weyl operators are Hermitian, we have that $\alpha_i\in \R$.
    It follows from~\eqref{eq:FourierWeyl} that
    $$
    \big|\widehat{\Delta_{a_i}f}(b_i)\big|^2 = \alpha_i^2.$$
    Defining $M = \alpha_1W(x_1) + \cdots +\alpha_kW(x_k)$ and $r = (\alpha_1^2 + \cdots + \alpha_k^2)^{1/2}$, we get that
    \begin{align*}
        r^2 &= \langle f, Mf\rangle\leq \|f\|_2^2\|M\|.
    \end{align*}
    By the commutation relations of the Weyl operators, we have that
    $MM^* = r^2\Id$.
    From this, we get that the operator norm of~$M$ equals
    $$
    \|M\| = \sqrt{\|MM^*\|} = r.$$
    Hence, $r^2\leq \|f\|_2^2\, r$, which gives the result.
\end{proof}

If~$\phi$ is a stabilizer state, it follows from \cref{prop:stab_Lagrangian} that~$P_\phi$ equals the uniform probability distribution over the Lagrangian~$\Lcal(\phi)$.
In this case, we have $P_\phi(\Lcal(\phi)) = 1$.
More generally, the characteristic weight $P_f(L)$ of a Lagrangian subspace~$L$ is closely connected with the correlation between the underlying function~$f$ and the stabilizer states associated with~$L$.
This is made precise by the following result:

\begin{proposition}
    If $f: \F_2^n \to \C$ is a nonzero function and $L\in \Lag(\F_2^{2n})$ is a Lagrangian, then
    \begin{equation} \label{eq:Pf_L4}
        P_f(L) = \sum_{\phi:\: \Lcal(\phi)=L} \bigg(\frac{|\langle f, \phi\rangle|}{\|f\|_2}\bigg)^4,
    \end{equation}
    where the sum is over distinct representatives of the set $\Stab(\F_2^n)/\U(1)$ whose associated Lagrangian is~$L$.
\end{proposition}

\begin{proof}
Fix an arbitrary basis~$\B(L)$ for~$L$.
Using our identification~$\simeq_\B$ (\cref{def:identification}), we can write the set we are summing over by $\{\phi_\chi:\: \chi\in \widehat{L}\}$, where each~$\phi_\chi$ denotes a stabilizer state satisfying $\phi_\chi \simeq_\B (L,\chi)$.

For convenience, let us denote by $\tau: \F_2^{2n} \to \Z_4$ the function given by
$$\tau(a,b) = |a\circ b| + \gamma_{\B(L)}(a,b) \mod 4,$$
so that (by equation~\eqref{eq:FourierWeyl}) we can write
$$\widehat{\Delta_a \phi_\chi}(b) = i^{\tau(a,b)} \chi(a,b) \one_L(a,b).$$
We then have
\begin{align*}
    |\langle \phi_\chi, f\rangle|^2
    &= \E_{a\in \F_2^n} \langle \Delta_a\phi_\chi, \Delta_af\rangle \\
    &= \E_{a\in \F_2^n} \sum_{b\in \F_2^n} \overline{\widehat{\Delta_a\phi_\chi}(b)} \widehat{\Delta_af(b)} \\
    &= \E_{(a,b) \in L} i^{-\tau(a,b)} \chi(a,b) \widehat{\Delta_af(b)},
\end{align*}
from which we conclude
$$|\langle \phi_\chi, f\rangle|^4 = \E_{(a,b), (c,d) \in L} i^{-\tau(a,b) +\tau(c,d)} \chi(a+c, b+d) \widehat{\Delta_af(b)} \overline{\widehat{\Delta_cf(d)}}.$$
Summing over all characters $\chi\in \widehat{L}$ and using the orthogonality of characters, we obtain
\begin{align*}
    \sum_{\chi\in \widehat{L}} |\langle \phi_\chi, f\rangle|^4
    &= \E_{(a,b), (c,d) \in L} i^{-\tau(a,b) +\tau(c,d)} 2^n \one\big[(a+c, b+d)=(0,0)\big] \widehat{\Delta_af(b)} \overline{\widehat{\Delta_cf(d)}} \\
    &= \E_{(a,b) \in L} \big|\widehat{\Delta_af(b)}\big|^2.
\end{align*}
This final expression is precisely $\|f\|_2^4 P_f(L)$, finishing the proof.
\end{proof}

This last proposition shows that the characteristic distribution~$P_f$ is biased towards the Lagrangian associated with a stabilizer state that correlates well with~$f$.
This makes the characteristic distribution relevant for the $U^3$-inverse theorem, as is made clearer in the following simple result:

\begin{lemma} \label{lem:Lcorr}
Let~$\phi:\F_2^n\to\C$ be a stabilizer state and let~$L = \Lcal(\phi)$ be its Lagrangian subspace.
Then, for any nonzero function $f:\F_2^n\to \C$, we have that
\begin{equation*}
    \bigg(\frac{|\langle f, \phi\rangle|}{\|f\|_2}\bigg)^4
    \leq
    P_f(L)
    \leq 
    \bigg(\frac{\|f\|_{U^3}}{\|f\|_2}\bigg)^4.
\end{equation*}
\end{lemma}

\begin{proof}
The first inequality follows immediately from equation~\eqref{eq:Pf_L4}.
For the second inequality, apply Cauchy-Schwarz to conclude that
$$\|f\|_2^4 P_f(L) = \frac{1}{2^n} \sum_{(a,b) \in L} \big|\widehat{\Delta_af}(b)\big|^2 \leq \bigg(\frac{1}{2^n} \sum_{(a,b) \in L} \big|\widehat{\Delta_af}(b)\big|^4 \bigg)^{1/2}.$$
This last expression is clearly bounded by
$$\bigg(\frac{1}{2^n} \sum_{a,b \in \F_2^n} \big|\widehat{\Delta_af}(b)\big|^4 \bigg)^{1/2} = \|f\|_{U^3}^4,$$
where we used identity~\eqref{eq:U3Fourier}.
The result follows.
\end{proof}

The maximal characteristic weight of a Lagrangian subspace is thus sandwiched between the~$U^3$ norm of~$f$ and its maximal correlation with a stabilizer state, making it a good proxy when investigating the inverse theorem.
To complement this, we also have an ``integration lemma,'' which allows one to pass from a high-weight Lagrangian to a correlating stabilizer state:\footnote{We note that this fact was already known in the quantum information literature~\cite{Gross2021}.}

\begin{lemma}[Integration] \label{lem:integration}
    For any nonzero function $f:\F_2^n\to \C$ and any Lagrangian $L\in \Lag(\F_2^{2n})$, we have that
    $$P_f(L) \leq \max_{\phi:\: \Lcal(\phi)=L} \bigg(\frac{|\langle f, \phi\rangle|}{\|f\|_2}\bigg)^2.$$
\end{lemma}

\begin{proof}
From \cref{prop:Lag_stab_basis}, we know that the stabilizer states whose Lagrangian is~$L$ form (scalar multiples of) an orthonormal basis.
We then conclude from equation~\eqref{eq:Pf_L4}~that
\begin{align*}
    \|f\|_2^4 P_f(L)
    &= \sum_{\phi:\: \Lcal(\phi)=L} |\langle f, \phi\rangle|^4 \\
    &\leq \Big(\max_{\phi:\: \Lcal(\phi)=L} |\langle f, \phi\rangle|^2\Big) \cdot \sum_{\phi:\: \Lcal(\phi)=L} |\langle f, \phi\rangle|^2 \\
    &= \Big(\max_{\phi:\: \Lcal(\phi)=L} |\langle f, \phi\rangle|^2\Big) \cdot \|f\|_2^2,
\end{align*}
and the result follows.
\end{proof}

These results inform our algorithmic strategy for the $U^3$-inverse theorem.
Given a bounded function $f: \F_2^n \to \C$ with high $U^3$ norm, we first find a ``high-weight'' Lagrangian~$L$ by sampling from a probability distribution closely related to~$P_f$;
such a high-weight Lagrangian must exist due to the (existential) $U^3$-inverse theorem and \cref{lem:Lcorr}.
We then find a stabilizer state~$\phi$ whose associated Lagrangian is~$L$ and whose correlation $|\langle f, \phi\rangle|$ is high;
this is possible due to \cref{lem:integration}.
Finally, we ``round'' the obtained stabilizer state~$\phi$ to a quadratic phase function~$(-1)^{q(\cdot)}$ without losing much in terms of $f$-correlation, which is possible due to the boundedness of~$f$.

\subsubsection{In odd characteristics}
With the exception of the uncertainty principle, everything else generalizes trivially to the odd-characteristic setting.
In this case, the \emph{characteristic distribution}~$P_f$ of a function $f: \F_p^n \to \C$ is defined over~$\F_p^{2n}$ by
$$P_f(a,b) = \frac{1}{p^n \|f\|_2^4} |\widehat{\Delta_af}(b)|^2.$$
As previously remarked, this distribution is natural to consider given the well-known connection between the quadratic structure of~$f$ and the large Fourier coefficients $|\widehat{\Delta_af}(b)|$ of its discrete derivatives \cite{Gowers1998}.

The ``characteristic weight'' of a Lagrangian subspace $L\in \Lag(\F_p^{2n})$ is closely connected with the correlation between~$f$ and the stabilizer states associated with~$L$:
$$P_f(L) = \sum_{\phi:\: \Lcal(\phi)=L} \bigg(\frac{|\langle f, \phi\rangle|}{\|f\|_2}\bigg)^4,$$
where the sum is over distinct representatives of the set $\Stab(\F_p^n)/\U(1)$.
From this equation, one easily sees how the characteristic weight of Lagrangians is related to the $U^3$-inverse theorem:
we have
\begin{equation*}
    \max_{\phi:\: \Lcal(\phi)=L} \bigg(\frac{|\langle f, \phi\rangle|}{\|f\|_2}\bigg)^4
    \leq
    P_f(L)
    \leq 
    \bigg(\frac{\|f\|_{U^3}}{\|f\|_2}\bigg)^4.
\end{equation*}
The (polynomial) $U^3$-inverse theorem under $L^2$ normalization posits that
\begin{equation} \label{eq:inverse_Fp}
    \max_{\phi \in \Stab(\F_p^n)} \frac{|\langle f, \phi\rangle|}{\|f\|_2} \geq \poly\bigg(\frac{\|f\|_{U^3}}{\|f\|_2} \bigg),
\end{equation}
and thus the maximum characteristic weight of a Lagrangian is also polynomially related to $\|f\|_{U^3}/\|f\|_2$.

The characteristic weight of Lagrangian subspaces is, however, much better behaved than the left-hand side of equation~\eqref{eq:inverse_Fp}, and it is (implicitly) used in the known proofs of the $U^3$-inverse theorem to arrive at the desired correlation bound.
To close the loop, we have the ``integration inequality''
\begin{equation} \label{eq:integration}
    P_f(L) \leq \max_{\phi:\: \Lcal(\phi)=L} \bigg(\frac{|\langle f, \phi\rangle|}{\|f\|_2}\bigg)^2,
\end{equation}
which allows us to pass from high Lagrangian weight to high stabilizer correlation.

Finally, we remark on a weaker version of the uncertainty principle (\cref{lem:uncertainty1}) which holds in all characteristics.
As shown by Gross, Nezami and Walter \cite[Lemma~3.10]{Gross2021} in the setting of quantum information theory, if $[(a,b),\, (c,d)] \neq 0$ then
$$|\widehat{\Delta_af}(b)|^2 + |\widehat{\Delta_cf}(d)|^2 \leq \bigg(2 - \frac{1}{4p^2}\bigg) \|f\|_2^4.$$
This is smaller than the maximum $2\|f\|_2^4$ that can be attained when $[(a,b),\, (c,d)] = 0$, for instance in the case where~$f$ is a stabilizer state and $(a,b)$, $(c,d) \in \Lcal(f)$.

\subsection{Connection to previous work}

The original proof of the $U^3$-inverse theorem over~$\F_p^n$ by Green and Tao \cite{GreenTao2008} can be cleanly expressed through the perspective developed in this section, as we now show.
In that setting, one starts with a bounded function $f: \F_p^n \to \C$ having high $U^3$ norm and wishes to show that~$f$ correlates with a quadratic phase function $\omega_p^{q(\cdot)}$.
This proceeds by studying the set of pairs $(a, b) \in \F_p^{2n}$ on which the characteristic distribution $P_f(a,b) \propto |\widehat{\Delta_af}(b)|^2$ is large, and showing that this set satisfies some strong linearity properties;
this is the main part of the argument, and closely follows Gowers's original approach \cite{Gowers1998}.

The linear property one arrives at in the end of this argument is that there exists a linear subspace $V \leq \F_p^{2n}$ of size roughly $p^n$ whose characteristic weight $P_f(V)$ is large.
In the approach followed by Gowers and by Green and Tao, this subspace is a ``graph'' $V = \{(x, Mx):\, x\in W\}$ for some subspace $W\leq \F_p^n$ of bounded codimension, a property that can be enforced due to boundedness of~$f$.\footnote{In the $L^2$ setting this is no longer true, as can be seen by considering the function $f = p^{n/2} \one_{\{0\}}$, whose characteristic distribution~$P_f$ is supported on $\{0^n\}\times \F_p^n$.}
The main step missing from Gowers's argument---later provided by Green and Tao---was essentially to show that the subspace~$V$ thus obtained is \emph{isotropic}, which translates to the matrix~$M$ in its definition being symmetric (on~$W$).
This ultimately allows one to ``integrate'' the linear behavior of the discrete derivative to arrive at a quadratic behavior for the original function~$f$, which is encapsulated by inequality~\eqref{eq:integration} above.
Green and Tao realized the importance of isotropy in this context, which is what led them to conjecture a link to symplectic geometry.

It is interesting to note that their original $U^3$-inverse theorem \cite[Theorem~2.3]{GreenTao2008} first provides correlation of~$f$ with \emph{stabilizer states}, from which they later conclude correlation with a quadratic phase function~$\omega_p^q$.
In fact, their proof shows the existence of a subspace $V = \{(x, Mx):\, x\in W\}$, where $W\leq \F_p^n$ is a subspace of bounded codimension and $M\in \F_p^{n\times n}$ is symmetric (self-adjoint) on~$W$, for which $P_f(V)$ is large.
This subspace~$V$ is contained inside the Lagrangian
$$L = \big\{(x,\, Mx+b):\: x\in W,\, b\in W^{\perp}\big\},$$
whose $P_f$-weight is then large as well.
The stabilizer states associated with this Lagrangian are supported on the cosets of~$W$, and share the same ``quadratic part'' given by the matrix~$M$ (see equation~\eqref{eq:explicit_odd} below).
What their inverse theorem shows is that, on average over cosets $y+W$, the function~$f$ correlates well with a stabilizer state whose Lagrangian is~$L$ and whose support is $y+W$.
From there, one can obtain ``global'' quadratic correlation via a simple averaging argument.

\subsection{Explicit formulas}

We now derive explicit descriptions for the class of stabilizer states and for ``neighbor'' stabilizer states to be defined below.
We note that (most of) the first result can be obtained by combining a theorem of Eisner and Tao \cite[Theorem~1.4]{eisner2012large} with the classification of non-classical quadratic phase functions by Tao and Ziegler \cite[Lemma~1.7]{TaoZiegler2012};
in the quantum setting, a result of this type was first obtained by Dehaene and De Moor~\cite{dehaene2003clifford}.
We provide a self-contained proof more in line with the techniques developed in this section.

\begin{lemma}[Description of stabilizer states] \label{lem:stab_explicit}
    Every stabilizer state $\phi \in \Stab(\F_2^n)$ can be written in the form
    \begin{equation} \label{eq:stab_explicit}
        \phi(x) = \alpha \sqrt{2^{n-\dim(V)}} \one_V(x+r) (-1)^{(x+r)^\tp A(x+r) + s\cdot(x+r)} i^{|d\circ (x+r)|},
    \end{equation}
    where $\alpha\in \U(1)$, $V\leq \F_2^n$ is a subspace, $A\in \F_2^{n\times n}$ is a matrix and $r, s , d \in \F_2^n$.
    Conversely, every function of the form~\eqref{eq:stab_explicit} is a stabilizer state, and its associated Lagrangian is
    \begin{equation} \label{eq:Lagrangian_explicit}
        \Lcal(\phi) = \big\{(h,\, Mh+w):\: h\in V,\, w\in V^\perp\big\} \quad \text{where} \quad M = A+A^\tp+\Diag(d).
    \end{equation}
\end{lemma}

\begin{proof}
Consider the simpler case where $r=s=0$, given by
$$\phi_0(x) := \alpha \sqrt{2^{n-\dim(V)}} \one_V(x) (-1)^{x^\tp Ax} i^{|d\circ x|}.$$
One can check that its multiplicative derivative in direction $a\in \F_2^n$ is given by
\begin{equation*}
    \Delta_a \phi_0(x)
    = 2^{n-\dim(V)} \one_V(x) \one_V(a) (-1)^{a^\tp (A+A^\tp)x + a^\tp Aa} i^{|d\circ a|} (-1)^{(d\circ a)\cdot x}.
\end{equation*}
Denote $M = A+A^\tp + \Diag(d)$, so that $a^\tp (A+A^\tp)x + (d\circ a)\cdot x = a^\tp Mx$.
Then
\begin{align*}
    \widehat{\Delta_a \phi_0}(b)
    &= 2^{n-\dim(V)} \one_V(a) i^{|d\circ a|} (-1)^{a^\tp Aa} \E_{x\in \F_2^n} \one_V(x) (-1)^{a^\tp Mx} (-1)^{b\cdot x} \\
    &= \one_V(a) i^{|d\circ a|} (-1)^{a^\tp Aa} \E_{x\in V} (-1)^{(Ma+b) \cdot x} \\
    &= i^{|d\circ a|} (-1)^{a^\tp Aa} \one_V(a) \one_{V^\perp}(Ma+b).
\end{align*}
Denoting
$L = \big\{(h,\, Mh+w):\: h\in V,\, w\in V^\perp\big\}$,
we see that~$L$ is a Lagrangian subspace and $|\widehat{\Delta_a \phi_0}(b)| = \one_L(a,b)$, so~$\phi_0$ is a stabilizer state with $\Lcal(\phi_0) = L$.

Finally, note that for all $r, s \in \F_2^n$ we have
$$W(r,s) \phi_0(x) = \alpha i^{|r\circ s|} \sqrt{2^{n-\dim(V)}} \one_V(x+r) (-1)^{(x+r)^\tp A(x+r) + s\cdot(x+r)} i^{|d\circ (x+r)|},$$
which is of the form~\eqref{eq:stab_explicit}.
We have seen in \cref{sec:extremizers} that all functions of the form $W(v)\phi_0$ for $v\in \F_2^{2n}$ are stabilizer states with the same Lagrangian~$L$, and that they form all such stabilizer states (up to phases).
Since every Lagrangian subspace can be written in the form~\eqref{eq:Lagrangian_explicit}, it follows that all stabilizer states can be written in the form~\eqref{eq:stab_explicit}, as wished.
\end{proof}

We remark that we can always assume, in equation~\eqref{eq:stab_explicit} above, that either $d=0$ or $d\notin V^\perp$.
Indeed, if $d\in V^{\perp}$, then the function $x \mapsto i^{|d\circ x|}$ over~$x\in V$ is a quadratic phase function taking values in $\pmset{}$, and thus can be absorbed into the ``classical part'' $(-1)^{x^\tp Ax + s\cdot x}$.
This technical remark will be useful for us in our algorithm.

Finally, we will need a description of ``neighbor'' stabilizer states, meaning two non-collinear stabilizer states~$\phi$ and~$\phi'$ whose inner product $|\langle\phi, \phi'\rangle| \neq 1$ is maximal.
By \cref{prop:inner_product_intersection}, we see that the maximum value of this inner product is $1/\sqrt{2}$.

\begin{lemma}[Description of neighbors] \label{lem:neighbors}
    Let $\phi, \phi' \in \Stab(\F_2^n)$ be stabilizer states such that $|\langle\phi, \phi'\rangle| = 1/\sqrt{2}$.
    Then, for any $v\in \Lcal(\phi')\setminus \Lcal(\phi)$, there exist $\sigma\in \pmset{}$ and $\alpha\in \U(1)$ such that
    $$\phi' = \alpha \bigg(\frac{I+ \sigma W(v)}{\sqrt{2}}\bigg) \phi.$$
\end{lemma}

\begin{proof}
Fix any $v\in \Lcal(\phi')\setminus \Lcal(\phi)$, and note that $\langle\phi,\, W(v)\phi\rangle = 0$.
Denote
$$\sigma = \langle\phi',\, W(v)\phi'\rangle \in \pmset{},$$ so that~$\phi'$ is in the $\sigma$-eigenspace of~$W(v)$, and let $\Pi_v^{\sigma} = (I+ \sigma W(v))/2$ be the projection onto this eigenspace.
We will first show that $\Pi_v^\sigma \phi$ is proportional to~$\phi'$.

Since~$\Pi_v^{\sigma}$ is self-adjoint, we obtain from the lemma's assumption that
$$|\langle\Pi_v^{\sigma} \phi,\, \phi'\rangle| = |\langle\phi,\, \Pi_v^{\sigma} \phi'\rangle| = |\langle\phi,\, \phi'\rangle| = \frac{1}{\sqrt{2}}.$$
Moreover,
$$\langle\Pi_v^{\sigma} \phi,\, \Pi_v^{\sigma} \phi\rangle = \frac{\big\langle\phi+ \sigma W(v)\phi,\, \phi+ \sigma W(v)\phi\big\rangle}{4} = \frac{2+ 2\sigma \langle\phi,\, W(v)\phi\rangle}{4} = \frac{1}{2},$$
and thus $\|\Pi_v^{\sigma} \phi\|_2 = 1/\sqrt{2}$.
We conclude that $|\langle\Pi_v^{\sigma} \phi,\, \phi'\rangle| = \|\Pi_v^{\sigma} \phi\|_2 \|\phi'\|_2$, and so by the equality case of the Cauchy-Schwarz inequality, $\Pi_v^{\sigma} \phi$ is proportional to~$\phi'$.

It follows that there exists some $\alpha\in \U(1)$ such that
$$\phi' = \alpha \frac{\Pi_v^{\sigma} \phi}{\|\Pi_v^{\sigma} \phi\|_2} = \alpha \frac{(I+\sigma W(v)) \phi}{\sqrt{2}},$$
as wished.
\end{proof}

\subsubsection{In odd characteristics}
Over~$\F_p^n$ for~$p$ odd, the stabilizer states can be written as
\begin{equation} \label{eq:explicit_odd}
    \phi(x) = \alpha \sqrt{p^{n-\dim(V)}} \one_V(x+r) \omega_p^{(x+r)^\tp M(x+r)/2 + s\cdot(x+r)},
\end{equation}
where $\alpha\in \U(1)$, $V\leq \F_p^n$ is a subspace, $M\in \F_p^{n\times n}$ is a symmetric matrix and $r, s \in \F_p^n$.
Moreover, every function of the form above is a stabilizer state, and its associated Lagrangian subspace is
$$\Lcal(\phi) = \big\{(h,\, Mh+w):\: h\in V,\, w\in V^\perp\big\}.$$

The maximum correlation between two linearly independent stabilizer states $\phi, \phi'$ is $|\langle\phi, \phi'\rangle| = 1/\sqrt{p}$.
If this maximum is attained, then for any $v\in \Lcal(\phi')\setminus \Lcal(\phi)$ there exist $\alpha\in \U(1)$ and a $p$-th root of unity~$\sigma$ such that
$$\phi' = \frac{\alpha}{\sqrt{p}} \sum_{j=0}^{p-1} \sigma^j W(v)^j \phi.$$

\section{Finding high-weight Lagrangians}
\label{sec:finding_Lagrangians}

This section establishes the central component of \cref{thm:quadratic_GL} (the Quadratic Goldreich--Levin theorem).
The following version of the original Goldreich--Levin algorithm, which is a special case of \cite[Theorem~4.3]{kim2023cubic}, serves both as subroutine and as a template for a new subroutine that we use in the quadratic setting.

\begin{theorem}[Goldreich--Levin algorithm] \label{thm:GL}
    Let $f:\F_2^n\to \C$ be a 1-bounded function, let~$\delta>0$ and $0<\tau\leq 1$. 
    There is a randomized algorithm that, with probability at least~$1-\delta$, returns a list $L\subseteq \F_2^n$ such that:
    \begin{itemize}
        \item If $|\widehat{f}(b)| \geq \tau$, then $b\in L$;
        \item For every $b\in L$, we have $|\widehat{f}(b)| \geq \tau/2$.
    \end{itemize}
    This algorithm makes $n\log n\, \poly(\log(1/\delta)/\tau)$ queries to~$f$ and runs in time \\
    $n^2 \log n\, \poly(\log(1/\delta)/\tau)$.
\end{theorem}

This ``list-decoding'' version of the Goldreich--Levin algorithm thus returns, with high probability, a complete list of  linear phase functions that have constant correlation with~$f$.
This is possible in $\poly(n)$-time because there are only a constant number of such linear phases, due to Parseval's identity.

Our Quadratic Goldreich--Levin algorithm will use a similar list-decoding procedure, where we obtain a list of stabilizer states which have high correlation with~$f$.
However, in the quadratic setting, we no longer have an analogue of Parseval's identity, and in fact there can be $\exp(n)$-many stabilizer states (or even quadratic phase functions) that have high correlation with~$f$.
Obtaining a complete list is therefore infeasible in polynomial time.
Instead, we limit our search to stabilizer states whose correlation with~$f$ is both non-negligible and nearly maximal among their ``neighbors.''
Surprisingly, there turns out to exist only a bounded number of such approximate local maximizers.

The following definition from~\cite{chen2024stabilizer} formalizes the notion of approximate local maximizers, and will be crucial for our arguments.
It is based on the fact that $|\langle\phi,\phi'\rangle|^2 \leq 1/{2}$ for any two linearly independent stabilizer states~$\phi, \phi'$ (this follows from \cref{prop:inner_product_intersection}).
Thus, two stabilizer states can be considered neighbors if they satisfy $|\langle\phi,\phi'\rangle|^2 = 1/2$.

\begin{definition}[Approximate local maximizer]
    Let~$f:\F_2^n \to \C$ be a function and $\gamma>0$ be a positive parameter.
    A stabilizer state $\phi\in \Stab(\F_2^n)$ is a \emph{$\gamma$-approximate local maximizer of correlation for~$f$} if it satisfies
    \begin{equation*}
        |\langle f, \phi\rangle|^2
        \geq
        \gamma \max_{\substack{\phi'\in \Stab(\F_2^n),\\ |\langle\phi,\phi'\rangle|^2 = 1/2}} |\langle f, \phi'\rangle|^2.
    \end{equation*}
\end{definition}

In this section, we develop the main component of an algorithm which identifies approximate local maximizers that correlate with~$f$.
More precisely, the main result of this section is an algorithm which, with non-negligible probability, recovers the Lagrangian subspace associated with a $\gamma$-approximate local maximizer of correlation~$\phi$ satisfying $|\langle f, \phi\rangle| \geq \tau$, where~$\phi$ is fixed but unknown.
(Recall the definition of~$\Lcal(\phi)$ from \cref{sec:extremizers}.)

\begin{theorem}[Lagrangian sampling] \label{thm:lagrangian}
For every $\gamma>1/2$ and $\tau\in (0,1)$, there exists a randomized algorithm {\sc LagrangianSampling} such that the following holds. 
Let $f:\F_2^n\to\C$ be a 1-bounded function, and let~$\phi$ be a stabilizer state that is a $\gamma$-approximate local maximizer for~$f$ and satisfies $|\langle f, \phi\rangle|\geq \tau$.
Then, {\sc LagrangianSampling} produces a basis for a subspace~$L\leq \F_2^{2n}$ such that
$$\Pr[L = \Lcal(\phi)] \geq \big((\gamma - \tfrac{1}{2}) \tau\big)^{O(\log(1/\tau))}.$$
Moreover, the algorithm makes $n^2 \log n\, \poly\big((\gamma - \tfrac{1}{2})^{-1} \tau^{-1}\big)$ queries to~$f$ and runs in time $n^3 \log n\, \poly\big((\gamma - \tfrac{1}{2})^{-1} \tau^{-1}\big)$.
\end{theorem}

The algorithm of \cref{thm:lagrangian} is based on the intuition that samples from the characteristic distribution~$P_f(a, b) \propto \big|\widehat{\Delta_af}(b)\big|^2$ are biased towards elements from~$\mathcal L(\phi)$, as shown in \cref{lem:Lcorr}.
For technical reasons, we will instead use a smoothened version of the characteristic distribution, given by its self-convolution:

\begin{definition}[Convoluted distribution]
    For a nonzero function $f:\F_2^n\to \C$, define its \emph{convoluted distribution}~$Q_f$ by
    \begin{equation*}
        Q_f = P_f*P_f,
    \end{equation*}
    where~$P_f$ is the characteristic distribution of~$f$ (\cref{def:char_dist}).
\end{definition}

Since a Lagrangian~$L$ is a subspace, it follows easily that
\begin{equation*}
Q_f(L) \geq P_f(L)^2.
\end{equation*}
Hence, if~$f$ correlates with a stabilizer state~$\phi$, then \cref{lem:Lcorr} shows that $\Lcal(\phi)$ has large mass according to both~$P_f$ and~$Q_f$.

Sampling from the convoluted distribution~$Q_f$ is known in quantum information theory as \emph{Bell difference sampling}~\cite{Gross2021}.
Indeed, \cref{thm:lagrangian} is essentially obtained from a ``dequantization'' of a Bell difference sampling-based quantum algorithm due to Chen, Gong, Ye and Zhang~\cite{chen2024stabilizer}.
The main difference between our algorithms is the analytic space they operate on:
their algorithm operates on a Hilbert space~$L^2$, where one assumes $\|f\|_2 = 1$ and admissible quantum operations allow for unitary transformations and sample access from~$Q_f$.
By contrast, our algorithm operates on~$L^\infty$, where one assumes $\|f\|_\infty \leq 1$ and is given query access to the function $x \mapsto f(x)$, while being able to perform simple arithmetic~operations.

\subsection{Sampling a good Lagrangian subspace} \label{sec:Lagrangian}
Towards proving \cref{thm:lagrangian}, we first work in an idealized setting where we assume that we have sample access to the convoluted distribution~$Q_f$.
Once this is achieved, we show how such samples can be approximately simulated using  query access to the given function~$f$.

Recall that our goal is to give an algorithm that, with high probability, returns the Lagrangian of a fixed (but arbitrary and unknown) approximate local maximizer~$\phi$ that has non-negligible correlation with~$f$.
A problem we encounter is that, since we do not know~$\phi$, we have no way to certify if a sample from~$Q_f$ belongs to~$\mathcal L(\phi)$.
A key idea of~\cite{chen2024stabilizer} is to instead aim for a set that does allow for easy membership verification.

\begin{definition}[Spectral set]
    For a function $f:\F_2^n\to \C$, define
    \begin{equation*}
        \Spec(f) = \big\{(a,b)\in \F_2^n\times\F_2^n : |\widehat{\Delta_af}(b)|^2 \geq 0.7 \|f\|_2^4\big\}.
    \end{equation*}
\end{definition}

Due to the uncertainty principle (\cref{lem:uncertainty1}), the spectral set is isotropic.
(The constant 0.7 is arbitrary, any other constant $0.5 < c < 1$ would do.)
Intuition for why it provides useful information is given by the following fact:
if~$f$ \emph{equals} the stabilizer state~$\phi$, then the spectral set equals~$\mathcal L(\phi)$ and~$Q_f$ is the uniform probability distribution over~$\mathcal L(\phi)$.
In this case, we can efficiently generate~$\mathcal L(\phi)$ by sampling~$\Theta(n)$ times from~$Q_f$ and taking the linear span of those samples.
If~$f$ is not itself a stabilizer state, then the spectral set might no longer equal~$\mathcal L(\phi)$, but it will still serve as a useful object to guide our algorithm.

The advantage of working with the spectral set is that we can easily estimate the value of $|\widehat{\Delta_af}(b)|^2$ for any $a, b\in \F_2^n$, and thus we can approximately check whether a given pair $(a, b)$ belongs to that set.
Our estimation procedure for the Fourier coefficients of a bounded function is given in the following simple lemma:

\begin{lemma}[Fourier estimation] \label{lem:fourest}
    Let $\eps, \delta>0$.
    There is a randomized algorithm $\FourEst_{\eps,\delta}$ that, given $b\in \F_2^n$ and query access to a $1$-bounded function $g:\F_2^n\to \C$, returns a random value $c\in \C$ such that
    $$\Pr\big[|c - \widehat{g}(b)|\leq \eps\big] \geq 1-\delta.$$
    This algorithm makes $O(\frac{1}{\eps^2}\log(1/\delta))$ queries to~$g$ and runs in time $O(\frac{1}{\eps^2} n\log(1/\delta))$.
\end{lemma}

\begin{proof}
Let $m \geq 2$ be an integer, let $x_1,\dots,x_m$ be independent uniformly distributed $\F_2^n$-valued random variables, and let $X_i = g(x_i)(-1)^{b\cdot x_i}$ for each $i\in[m]$.
Then $\E[X_i] = \widehat{g}(b)$ for each $i\in [m]$.
Letting $\overline X = m^{-1}(X_1 + \cdots + X_m)$, it follows from Hoeffding's inequality~that
\begin{equation*}
    \Pr\big[\big|\overline{X} - \E\overline{X}\big| > \eps\big] \leq 4\exp(-2\eps^2m).
\end{equation*}
Thus, by taking $m=O(\frac{1}{\eps^2}\log(1/\delta))$, the quantity $c = \overline{X}$ satisfies the requirement of the lemma with the desired probability.
\end{proof}

Using Fourier estimation, we can implement a post-selection procedure on samples from~$Q_f$ that yields an approximate sampler from~$Q_f$  conditioned on lying in $\Spec(f)$.
Taking inspiration from the 100\% case where~$f$ is a stabilizer state, we will then generate a random set~$F\subseteq \F_2^n\times\F_2^n$ of~$\Theta(n)$ such samples.
We show that, with good probability, $\vspan(F)$ will then cover all but a tiny fraction of the whole spectral set.
The following notion makes this idea precise.

\begin{definition}[Approximate spectral set]
    Let $f:\F_2^n\to \C$ be a nonzero function.
    A set $S\subseteq \F_2^{2n}$ is an $\eps$-approximate spectral set for~$f$ if 
    \begin{equation*}
        Q_f\big(\Spec(f)\setminus S\big) \leq \eps.
    \end{equation*}
\end{definition}

We proceed with a case analysis.
The easy case covers the situation where the span of every approximate spectral set contains~$\mathcal L(\phi)$, which we refer to as \emph{robust Lagrangian generation}.
In this case,~$\mathcal L(\phi)$ can be generated simply by taking the linear span of our randomly sampled set~$F$.
The complementary case is more challenging and builds on an ``energy-increment''   or ``boosting'' procedure introduced in~\cite{chen2024stabilizer}.
The next two subsections cover these two cases in detail.

\subsubsection{Robust Lagrangian generation}

The first case is characterized by the definition~below.

\begin{definition}[Robust generation]
    Let~$f:\F_2^n\to \C$ be a nonzero function, $L \leq \F_2^{2n}$ be a Lagrangian subspace and $0<\eps<1$.
    We say that~$f$ \emph{$\eps$-robustly generates~$L$} if $L \leq \vspan(F)$ for every $\eps$-approximate spectral set~$F$.
\end{definition}

If~$f$ $\eps$-robustly generates~$\Lcal(\phi)$, then it is easy to learn a basis of~$\Lcal(\phi)$ by sampling $O(n/\eps)$ pairs $(a, b) \sim Q_f$.
This is because the span of such a sample is an approximate spectral set with good probability.
This was essentially proven in~\cite{chen2024stabilizer}, but we provide a proof here for completeness.

\begin{lemma}
\label{lem:robustgen}
    Let $f:\F_2^n\to \C$ be a $1$-bounded function and let $\eps, \tau>0$.
    Suppose that $\|f\|_2\geq \tau$ and that~$f$ $\eps$-robustly generates a Lagrangian subspace~$L$.
    Then, there is a randomized algorithm that uses $O(n/\eps)$ samples from~$Q_f$, makes $O\big(\tfrac{1}{\eps \tau^8} n\log \tfrac{n}{\eps}\big)$ random queries to~$f$, and returns a basis for a random subspace~$L'\leq \F_2^{2n}$ such that
    $$\Pr[L'=L] \geq 2/3.$$
    This algorithm runs in time $O\big(\tfrac{1}{\eps} n^3 + \tfrac{1}{\eps \tau^8} n^2\log \tfrac{n}{\eps}\big)$.
\end{lemma}

\begin{proof}
Let $S\subseteq \F_2^{2n}$ be a random set of $m = O(n/\eps)$ independent $Q_f$-samples and let $T = S\cap\Spec(f)$.
We first show that, with probability at least~$0.9$, the set $\vspan(T)$ is an $\eps$-approximate spectral set.

Denote $p = Q_f(\Spec(f))$.
We must have that $p > \eps$, since otherwise the empty set would be an approximate spectral set, in contradiction with the assumption that ~$f$ $\eps$-robustly generates a Lagrangian subspace.
Note that the elements of~$T$ are distributed independently according to the conditional distribution $R_f = \one_{\spec(f)}\cdot Q_f/p$.
By the Chernoff bound, we have that $|T| \geq (pm)/2$ with probability at least~$0.95$;
with our choice of~$m$, we may assume that
$$\frac{pm}{2} \geq \frac{4n+10}{\eps/p}.$$
Conditioned on this number of $R_f$-sampled points, the set $\vspan(T)$ will cover a $(1-\eps/p)$-fraction of the $R_f$-mass with probability at least~$0.95$
(this fact is given by \cite[Lemma~4.21]{chen2024stabilizer}).
We conclude that, with probability at least~$0.9$, we have $R_f(\vspan(T)) \geq 1-\eps/p$.
In this case,
\begin{equation*}
    Q_f\big(\spec(f)\setminus \vspan(T)\big) =
    p\,R_f\big(\spec(f)\setminus \vspan(T)\big)
    \leq
    \eps,
\end{equation*}
showing that $\vspan(T)$ is an $\eps$-approximate spectral set.

By boundedness of~$f$, we can estimate the value of~$\|f\|_2^2$ up to a $\tau^4/100$ additive error using $O(1/\tau^8)$ random queries to~$f$;
we obtain a real number $0<r\leq 1$ such that
$$\Pr\big[ \big|r- \|f\|_2^2\big| \leq \tau^4/100\big] \geq 0.9.$$
Whenever this event holds, we have that $\big|r^2 - \|f\|_2^4\big| \leq \tau^4/50 \leq \|f\|_2^4/50$.

Now, for each $(a,b)\in S$, run the algorithm $\FourEst_{\eps_1,\delta_1}$ from Lemma~\ref{lem:fourest} on input $(\Delta_af, b)$ with parameters $\eps_1 = \tau^4/100$ and $\delta_1 = 1/(10m)$.
By the union bound, with probability at least~$0.9$, all values~$c_{a,b}$ estimated by this algorithm will be within $\tau^4/100$ of the true value $\widehat{\Delta_af}(b)$.
In this case, we obtain the estimate
$$\big||c_{a,b}|^2 - |\widehat{\Delta_af}(b)|^2\big| \leq \tau^4/50 \leq \|f\|_2^4/50 \quad \text{for all $(a,b) \in S$.}$$
Combining the two estimates above, we conclude that
$$\frac{|\widehat{\Delta_af}(b)|^2}{\|f\|_2^4} - 0.1 < \frac{|c_{a,b}|^2}{r^2} < \frac{|\widehat{\Delta_af}(b)|^2}{\|f\|_2^4} + 0.1$$
holds for all $(a,b)\in S$ with probability at least~$0.8$.

Let~$F\subseteq S$ be the set of pairs $(a,b)$ for which we have~$|c_{a,b}|^2/r^2 \geq 0.6$.
By the above, with probability at least~$0.8$, we have
$$\big\{(a,b)\in S:\: |\widehat{\Delta_af}(b)|^2 \geq 0.7 \|f\|_2^4\big\} \subseteq F \subseteq \big\{(a,b)\in S:\: |\widehat{\Delta_af}(b)|^2 > 0.5 \|f\|_2^4\big\}.$$
The leftmost set is precisely $S\cap\Spec(f) = T$, while the rightmost set is isotropic by the uncertainty principle (\cref{lem:uncertainty1}).
As $\vspan(T)$ is an $\eps$-approximate spectral set with probability at least~$0.9$, it then follows that the set~$F$ is an isotropic $\eps$-approximate spectral set for~$f$ with probability at least~$0.7$.
Since~$f$ $\eps$-robustly generates~$L$, we get that~$\vspan(F) = L$ in this case.
We then return a basis for~$F$, which can be found in $O(mn^2)$ time via Gaussian elimination.
\end{proof}

\subsubsection{Non-robust Lagrangian generation implies energy increment.}

If~$f$ \emph{does not} generate~$\mathcal L(\phi)$ robustly, then there is a simple way to obtain an ``energy increment'' given by an increase of the normalized correlation with~$\phi$.
This is obtained by replacing~$f$ with the projection of~$f$ onto the appropriate eigenspace of a Weyl operator $W(a,b)$ associated with~$\phi$, as explained below.

Given $a, b\in \F_2^n$ and $\sigma\in \pmset{}$, let $V_{a,b}^{\sigma}$ denote the $\sigma$-eigenspace of the Weyl operator $W(a,b)$
(which is defined in \cref{def:Weyl}).
This space can be  explicitly written as
\begin{equation*}
    V_{a,b}^{\sigma} = \big\{g: \F_2^n\to \C\mid g(x+a) = \sigma i^{|a\circ b|} (-1)^{b\cdot x} g(x)\:\:\text{for all $x\in \F_2^n$}\big\}.
\end{equation*}
It follows readily from the definition of the Lagrangian $\Lcal(\phi)$ that for each $(a,b)\in \mathcal L(\phi)$, there is  a unique $\sigma\in \pmset{}$ such that $\phi\in V_{a,b}^\sigma$.
Furthermore, the projection $\Pi_{a,b}^{\sigma}f$ of a function~$f$ to~$V_{a,b}^{\sigma}$ is given by 
\begin{equation*}
    \Pi_{a,b}^{\sigma}f = \frac{f + \sigma W(a,b)f}{2}.
\end{equation*}

\begin{lemma}[Energy increment] \label{lem:boost}
    Let~$f:\F_2^n\to \C$ be a function and suppose~$\phi\in V_{a,b}^{\sigma}$ is a stabilizer state.
    If $(a,b)\not\in \Spec(f)$, then the function $f' := \Pi_{a,b}^{\sigma}f$ satisfies
    \begin{equation*}
        \frac{|\langle f',\phi\rangle|^2 }{\|f'\|_{2}^2} \geq 1.08\frac{|\langle f, \phi\rangle|^2}{\|f\|_{2}^2}.
    \end{equation*}
    Moreover, if~$\phi$ is a $\gamma$-approximate local maximizer for~$f$, then it is also a $\gamma$-approximate local maximizer for~$f'$.
\end{lemma}

\begin{proof}
Since $\phi\in V_{a,b}^\sigma$, we have that $\Pi_{a,b}^\sigma\phi = \phi$, and so
\begin{equation*}
    \langle f',\phi\rangle = \langle \Pi_{a,b}^\sigma f,\phi\rangle = \langle f,\Pi_{a,b}^\sigma\phi\rangle = \langle f, \phi\rangle. 
\end{equation*}
We also have that
\begin{align*}
    \|f'\|_2^2
    &= \frac{1}{4} \big\langle f+\sigma W(a,b)f,\, f+\sigma W(a,b)f\big\rangle \\
    &\leq \frac{1}{2}\|f\|_2^2 + \frac{1}{2}|\widehat{\Delta_af}(b)| \\
    &\leq \frac{1}{2}\big(1 + \sqrt{0.7}\big) \|f\|_2^2 \\
    &\leq 0.92 \|f\|_2^2,
\end{align*}
where we used identity~\eqref{eq:FourierWeyl} for the first inequality.
This implies the first claim.

Now suppose that~$\phi$ is a $\gamma$-approximate local maximizer for~$f$.
By \cref{lem:neighbors}, any $\phi'\in\Stab(\F_2^n)$ satisfying $|\langle \phi,\phi'\rangle| = 1/\sqrt{2}$ has the form $\tfrac{1}{\sqrt{2}}(I + \sigma' W(c,d))\phi$ for some $\sigma' \in \pmset{}$ and $c,d\in \F_2^n$.
Let then $M = \tfrac{1}{\sqrt{2}}(I + \sigma' W(c,d))$ and $\phi' = M\phi$.

There are two cases to consider.
If $[(a,b),(c,d)] = 0$, then $\Pi_{a,b}^\sigma$ and $M$ commute and we get that
\begin{align*}
    \langle f',\phi'\rangle = \langle f, \Pi_{a,b}^\sigma M\phi\rangle
    =\langle f, \phi'\rangle.
\end{align*}
This gives
\begin{align*}
    |\langle f',\phi'\rangle|^2 &=|\langle f, \phi'\rangle|^2
    \leq \frac{1}{\gamma}|\langle f, \phi\rangle|^2
    = \frac{1}{\gamma}|\langle f',\phi\rangle|^2.
\end{align*}
If $[(a,b),(c,d)] = 1$, then $\Pi_{a,b}^\sigma M\phi = \frac{1}{\sqrt{2}}\phi$ and so $\langle f',\phi'\rangle = \frac{1}{\sqrt{2}}\langle f',\phi\rangle$.
Since~$\gamma\leq 1$, this implies that
\begin{equation*}
    |\langle f',\phi'\rangle|^2 \leq \frac{1}{\gamma}|\langle f',\phi\rangle|^2
\end{equation*}
in this case as well, finishing the proof.
\end{proof}

The idea now is to iteratively apply this energy increment step until a function has been found that robustly generates~$\mathcal L(\phi)$, at which point the algorithm from Lemma~\ref{lem:robustgen} can be used to find~$\mathcal L(\phi)$ with good probability.
The key observation is that, if~$f$ does not robustly generate~$\Lcal(\phi)$, then it is not hard to find a projection that increases the energy as in Lemma~\ref{lem:boost}.
If we start with a function satisfying $|\langle f, \phi\rangle|/\|f\|_2 \geq \tau$, then the energy can only be boosted at most $t= O(\log(1/\tau))$ times;
hence, if we choose~$t'$ uniformly at random from~$\{0,\dots,t\}$ and boost~$t'$ times, then with probability at least~$1/t$ we will have obtained a projection of~$f$ that robustly generates~$\mathcal L(\phi)$.

The following lemma justifies the key observation above:
if $f$ does not robustly generate~$\Lcal(\phi)$, then a sample from~$Q_f$ yields a pair $(a,b)\in \mathcal L(\phi)\setminus \spec(f)$ with non-negligible probability.
Flipping a coin to choose a sign~$\sigma$ then gives a triple $(a,b,\sigma)$ enabling an energy boost with non-negligible probability.

\begin{lemma} \label{lem:nonrobust}
    Let~$\gamma\in (\tfrac{1}{2}, 1)$, $\tau>0$ and  $\eps = (\gamma  -\frac{1}{2})^2\tau^8/8$.
    Let~$f:\F_2^n\to\C$ be a function and let~$\phi$ be a $\gamma$-approximate local maximizer of correlation for~$f$ such that $|\langle f, \phi\rangle| \geq \tau$.
    Suppose that~$f$ does not $\eps$-robustly generate~$\mathcal L(\phi)$.
    Then,
    \begin{equation*}
        Q_f\big(\mathcal L(\phi)\setminus \spec(f)\big) \geq \eps.
    \end{equation*}
\end{lemma}

The proof of Lemma~\ref{lem:nonrobust} will occupy the next few pages.
It relies on the observation that---in the non-robust setting---there must be an approximate spectral set~$F$ such that $\Lcal(\phi) \cap\vspan(F)$ is a strict subspace of~$\Lcal(\phi)$.
It then uses the crucial fact that, if~$\phi$ is an approximate local maximizer of correlation for~$f$, then the convoluted distribution~$Q_f$ is smoothly distributed over cosets of subspaces of~$\mathcal L(\phi)$.
(This is where using~$P_f$ would not work, as this is not necessarily the case for~$P_f$.)
The lemmas below make precise what this smooth distribution property is, first in a special case and then in full generality.

\begin{lemma} \label{lem:terrible}
    Let $f:\F_2^n\to \C$ be a function and suppose that $\phi = 2^{n/2}\one_{\{0\}}$ is a $\gamma$-approximate local maximizer for~$f$.
    Let $V \leq \F_2^n$ be a subspace of codimension~1.
    Then
    \begin{equation*}
        Q_f(\{0^n\}\times (\F_2^n\setminus V)) \geq \tfrac{1}{4} \big(\gamma - \tfrac{1}{2}\big)^2|\langle f, \phi\rangle|^8.
    \end{equation*}
\end{lemma}

\begin{proof}
Let~$u\in \F_2^n\setminus\{0\}$ be such that~$V = \{u\}^\perp$.
  We begin by showing that
  \begin{equation}\label{eq:gammaballs}
      y:= \frac{f(u)^2}{f(0)^2}\not\in \{-1,1\} + \big(\gamma - \tfrac{1}{2}\big)\D.
  \end{equation}
  Indeed, since~$\phi$ is a $\gamma$-approximate local maximizer of correlation for~$f$, it follows that
  \begin{equation*}
      2^{-n}|f(0)|^2 \geq \gamma\max_{a\in \Z_4}\big|\langle f, 2^{(n-1)/2}(\one_{\{0\}} + i^a\one_{\{u\}})\rangle\big|^2.
  \end{equation*}
  In turn, this implies that
  \begin{equation*}
      \frac{f(u)}{f(0)} \not\in \{1,i,-1,-i\} + \big(\gamma - \tfrac{1}{2}\big)\D,
  \end{equation*}
  which gives~\eqref{eq:gammaballs}.

Since~$V$ only has one nontrivial coset, we can write $\F_2^n\setminus V = V+w$ for some $w\not\in V$.
The quantity we wish to bound may be written as
 \begin{align*}
      \sum_{y\in V}Q_f(0,y+w) 
      &= \sum_{y \in V}\sum_{c,d\in \F_2^n}P_f(c,d)P_f(c,y+w+d)\\
      &=\frac{2}{4^n}\sum_{c,d\in \F_2^n}|\widehat{\Delta_cf}(d)|^2\sum_{y \in V}|\widehat{\Delta_cf}(y+w+d)|^2\\
      &=2\sum_{c\in \F_2^n}\Big(\sum_{y\not\in V}\frac{|\widehat{\Delta_cf}(y)|^2}{2^n}\Big)
      \Big(\sum_{z\in V}\frac{|\widehat{\Delta_cf}(z)|^2}{2^n}\Big),
  \end{align*}
  where the last identity is obtained by splitting the over~$d\in \F_2^n$ up into a sum over~$V$ and a sum over~$V+w$.
  Keeping only the terms $c\in \{0,u\}$, we get that this is bounded from below by
  \begin{equation}\label{eq:foursums}
      2\Big(\sum_{y\not\in V}\frac{|\widehat{\Delta_0f}(y)|^2}{2^n}\Big)
      \Big(\sum_{z\in V}\frac{|\widehat{\Delta_0f}(z)|^2}{2^n}\Big)
      +
      2\Big(\sum_{y\not\in V}\frac{|\widehat{\Delta_uf}(y)|^2}{2^n}\Big)
      \Big(\sum_{z\in V}\frac{|\widehat{\Delta_uf}(z)|^2}{2^n}\Big).
  \end{equation}
  Expanding the definition of the Fourier transforms of the multiplicative derivatives gives that the first two of the above four sums  are bounded as follows
  \begin{align*}
     \sum_{y\not\in V}\frac{|\widehat{\Delta_0f}(y)|^2}{2^n}
     &=
     \frac{1}{2^{n+2}}\Exp_{x\in \F_2^n}\big(|f(x)|^2 - |f(x+u)|^2\big)^2\geq \frac{1}{2^{2n+1}}\big(|f(0)|^2 - |f(u)|^2\big)^2, \\
     \sum_{z\in V}\frac{|\widehat{\Delta_0f}(z)|^2}{2^n}
     &=
     \frac{1}{2^{n+2}}\Exp_{x\in \F_2^n}\big(|f(x)|^2 + |f(x+u)|^2\big)^2\geq \frac{1}{2^{2n+1}}\big(|f(0)|^2 + |f(u)|^2\big)^2.
  \end{align*}
Similarly, the last two of the sums in~\eqref{eq:foursums} are bounded by
    \begin{align*}
     \sum_{y\not\in V}\frac{|\widehat{\Delta_uf}(y)|^2}{2^n}
     &=
     \frac{1}{2^{n+1}}\Exp_{x\in \F_2^n}\big(|f(x)|^2|f(x+u)|^2 - \overline{f(x)}^2f(x+u)^2\big)\\
     &\geq \frac{1}{2^{2n}}\big(|f(0)|^2|f(u)|^2 - \Re\big(\overline{f(0)}^2f(u)^2\big)\big), \\
     \sum_{z\in V}\frac{|\widehat{\Delta_uf}(z)|^2}{2^n}
     &=
     \frac{1}{2^{n+1}}\Exp_{x\in \F_2^n}\big(|f(x)|^2|f(x+u)|^2 + \overline{f(x)}^2f(x+u)^2\big)\\
     &\geq \frac{1}{2^{2n}}\big(|f(0)|^2|f(u)|^2 + \Re\big(\overline{f(0)}^2f(u)^2\big)\big).
  \end{align*}

Combining these bounds gives that~\eqref{eq:foursums} is bounded from below by
\begin{equation}\label{eq:circles}
    \frac{1}{2^{4n+1}}|f(0)|^8(1-|y|)^2(1+|y|)^2
    +
    \frac{1}{2^{4n-1}}|f(0)|^8\big(|y| - \Re(y)\big)^2\big(|y| + \Re(y)\big)^2.
\end{equation}
Note that~$|f(0)|^8/2^{4n} = |\langle f, \phi\rangle|^8$.
We bound~\eqref{eq:circles} from below by using that the forbidden region of~$y$ in the complex plane given by~\eqref{eq:gammaballs} contains two segments of a narrow annulus around the complex unit circle (see Figure~\ref{fig:circles}).

\begin{center}
    \begin{figure}[ht]
        \includegraphics[width=6cm]{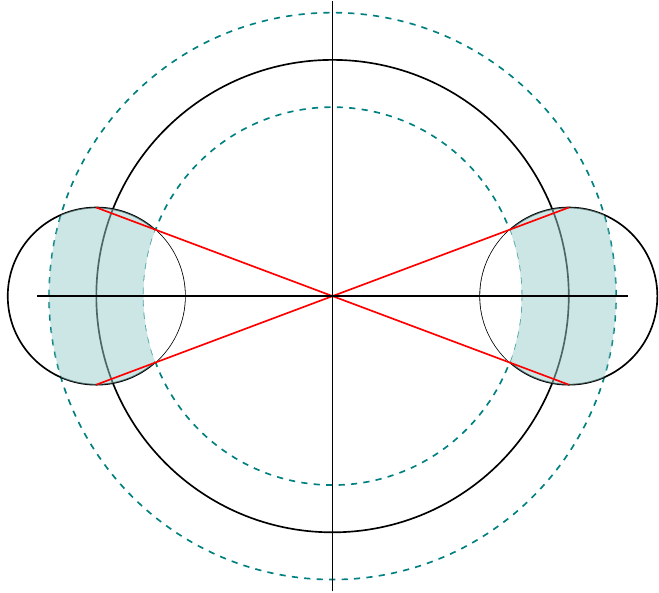}
    \caption{Forbidden regions for~$y$.}\label{fig:circles}
    \end{figure}
\end{center}

Choose the angles between the straight lines and the horizontal axis to be such that the distance from the origin to the small circles equals $r = \sqrt{1 - (\gamma-1/2)^2}$.

If~$y$ lies outside of the annulus, then the first term of~\eqref{eq:circles} is at least $\frac{1}{4}(\gamma - 1/2)^2|\langle f, \phi\rangle|^8$.
If~$y$ lies inside the annulus but outside of the small circles, then elementary trigonometry shows that the second term of~\eqref{eq:circles} is at least $\frac{1}{4}(\gamma - 1/2)^2|\langle f, \phi\rangle|^8$.
\end{proof}

\begin{lemma}[Smoothness over cosets] \label{lem:Qfsmooth}
    Let $f: \F_2^n\to \C$.
    For $\gamma \in (\tfrac{1}{2}, 1)$, let~$\phi$ be a $\gamma$-approximate local maximizer for~$f$ such that $|\langle f, \phi\rangle| \geq\tau$.
    Then, for every proper subspace $T\lneq \mathcal L(\phi)$, we have
    \begin{equation*}
        Q_f\big(\mathcal L(\phi)\setminus T\big) \geq \tfrac{1}{4} \big(\gamma - \tfrac{1}{2}\big)^2 \tau^8.
    \end{equation*}
\end{lemma}

\begin{proof}
As the symplectic group acts transitively on the Lagrangian subspaces, there exists a symplectic map $S\in \Sp(\F_2^{2n})$ such that $S\Lcal(\phi) = \{0^n\} \times \F_2^n$.
From \cref{lem:semi_rep}, there exists a unitary $\sigma(S) \in \U(\F_2^n)$ satisfying
$$\sigma(S) W(x) \sigma(S)^* \propto W(Sx) \quad \text{for all $x\in \F_2^{2n}$.}$$
As seen in \cref{sec:extremizers}, $\sigma(S) \phi$ is then a stabilizer state with associated Lagrangian $S\Lcal(\phi) = \{0^n\} \times \F_2^n$.
Finally, since the Weyl operators act transitively (up to phases) on stabilizer states sharing the same Lagrangian (see equation~\eqref{eq:Weyl_action}), and since $2^{n/2}\one_{\{0\}}$ is a stabilizer state with Lagrangian $\{0^n\} \times \F_2^n$, we conclude there exist $\alpha \in \U(1)$ and $v\in \F_2^{2n}$ such that $\alpha W(v) \sigma(S) \phi = 2^{n/2}\one_{\{0\}}$.

Denote $U = \alpha W(v) \sigma(S)$, so that~$U$ is a unitary map that takes stabilizer states to stabilizer states (see \cref{thm:isoU3}).
It follows that $U\phi = 2^{n/2}\one_{\{0\}}$ is a $\gamma$-approximate local maximizer of correlation for~$Uf$, and
\begin{equation*}
    Q_f(X) = P_f * P_f(X) = P_{Uf} * P_{Uf}(SX) = Q_{Uf}(SX)
\end{equation*}
for any set $X\subseteq \F_2^{2n}$.
Finally, note that for $T\lneq \Lcal(\phi)$ we have $ST \lneq S\Lcal(\phi) = \{0^n\} \times \F_2^n$, and thus there exists a subspace $V\lneq \F_2^n$ such that $ST = \{0^n\} \times V$.
We conclude that
\begin{equation*}
    Q_f\big(\mathcal L(\phi)\setminus T\big)
    =
    Q_{Uf}\big(\{0^n\}\times (\F_2^n\setminus V) \big),
\end{equation*}
and the result follows from Lemma~\ref{lem:terrible}.
\end{proof}

The proof of Lemma~\ref{lem:nonrobust} now follows immediately from the above lemma. If~$f$ does not $\eps$-robustly generate~$\mathcal L(\phi)$, then there is an $\eps$-approximate spectral set~$F$ such that $\mathcal L(\phi)\cap \vspan(F)$ is a proper subspace of~$\mathcal L(\phi)$.
It then follows from Lemma~\ref{lem:Qfsmooth} that
\begin{align*}
    Q_f\big(\mathcal L(\phi)\setminus \spec(f)\big) 
    &\geq Q_f(\mathcal L(\phi)\setminus \vspan(F)) - Q_f\big(\spec(f)\setminus \vspan(F)\big)\\
    &\geq \tfrac{1}{4} \big(\gamma - \tfrac{1}{2}\big)^2\tau^8 - \eps,
\end{align*}  
finishing the proof by our choice of~$\eps$.

\subsubsection{Sampling the desired Lagrangian}

Putting the above ideas together gives the following algorithm.

\begin{algorithm}[H]
\caption{\textsc{LagrangianSampling}}
\DontPrintSemicolon
\SetKwInOut{Input}{Input}
\Input{1-bounded function $f:\F_2^n\to\C$, $\gamma>1/2$ and $\tau>0$}
\medskip
Set $t = \lceil\log_{1.08}(1/\tau)\rceil$\;
Set $f_0 = f$\;
\For{$i\in [t]$}{
$(a_i,b_i) \gets Q_{f_{i-1}}$ \tcp*{sample from the convoluted distribution}
$\sigma_i \gets \mathrm{Uniform}\{+1,-1\}$\tcp*{sample a random sign}
Set $f_i = \Pi_{a_i,b_i}^{\sigma_i}f_{i-1}$\tcp*{generate a projection of $f_{i-1}$}
}
$s \gets \mathrm{Uniform}\{0,1,\dots,t\}$\tcp*{sample a random iterate index}
\Return{A basis obtained by the algorithm from Lemma~\ref{lem:robustgen} on input~$f_s$ with parameters $\eps = (\gamma -\frac{1}{2})^2 \tau^8/8$ and $\tau$}
\end{algorithm}

An analysis of this algorithm gives the following idealized analogue of \cref{thm:lagrangian}.

\begin{theorem}[Lagrangian sampling]
\label{thm:ideal_Lsampling}
Let~$f:\F_2^n\to\C$ be a 1-bounded function, and let~$\phi$ be a $\gamma$-approximate local maximizer of correlation for~$f$ that satisfies $|\langle f, \phi\rangle|\geq \tau$.
Denote $t = \lceil\log_{1.08}(1/\tau)\rceil$ and $\eps = (\gamma -\frac{1}{2})^2\tau^8/8$.
Then, $\textsc{LagrangianSampling}(f,\gamma,\tau)$ returns a basis for a subspace $L\leq \F_2^{2n}$ such that
$$\Pr[L = \Lcal(\phi)] \geq \frac{2}{3(t+1)} \Big(\frac{\eps}{2}\Big)^{t+1} = \big((\gamma - \tfrac{1}{2})\tau\big)^{O(\log(1/\tau))}.$$
This algorithm takes $O(n/\eps)$ samples from~$Q_{f_i}$ for some $i\in [t]$, makes $O\big(\tfrac{1}{\eps \tau^9} n \log\tfrac{n}{\eps}\big)$ queries to~$f$, and runs in $O\big(\tfrac{1}{\eps} n^3 + \tfrac{1}{\eps \tau^9} n^2 \log\tfrac{n}{\eps}\big)$ time.
\end{theorem}

\begin{proof}
Given functions $g,g':\F_2^n\to \C$, define the following conditions:
\begin{itemize}
    \item Base condition $\mathsf{BC}(g)$: $\|g\|_\infty\leq 1$, $\phi$ is a $\gamma$-approximate local maximizer of correlation for~$g$ and $|\langle g,\phi\rangle| \geq \tau$.
    \item Robust generation $\mathsf{RG}(g)$: $\mathsf{BC}(g)$ holds and $g$ $\eps$-robustly generates~$\mathcal L(\phi)$.
    \item Energy increment $\mathsf{EI}(g,g')$: $\frac{|\langle g',\phi\rangle|^2}{\|g'\|_2^2} \geq 1.08\frac{|\langle g,\phi\rangle|^2}{\|g\|_2^2}$ and $\mathsf{BC}(g), \mathsf{BC}(g')$ hold.
\end{itemize}

For each $i\in\{0,1,\dots,t-1\}$ consider the success event 
\begin{equation*}
    \mathrm{succ}_i = \Big(\bigwedge_{j=0}^{i}\mathsf{EI}(f_j,f_{j+1})\Big)
    \vee
    \bigvee_{j=0}^i\mathsf{RG}(f_j).
\end{equation*}
Because the energy is capped by~1, we have that $\mathrm{succ}_t = \bigvee_{j=0}^t\mathsf{RG}(f_j)$.
Thus, the final success event $\mathrm{succ}_t$ implies that one of the~$f_i$ $\eps$-robustly generates~$\mathcal L(\phi)$.

By Lemma~\ref{lem:boost} and Lemma~\ref{lem:nonrobust}, we have that 
\begin{equation*}
    \Pr\big[\mathrm{succ}_{i+1}\mid\mathrm{succ}_i\big]
    \geq\Pr\big[\mathsf{EI}(f_{i+1},f_{i+2})\vee \mathsf{RG}(f_{i+1})\mid \mathsf{BC}(f_{i+1})\big]\geq \frac{\eps}{2}.
\end{equation*}
It then follows that
\begin{align*}
    \Pr\big[\mathrm{succ}_t\big]
    &=
    \Pr[\mathrm{succ}_0]\prod_{i=0}^{t-1}\Pr\big[\mathrm{succ}_{i+1}\mid\mathrm{succ}_i\big] \geq \Big(\frac{\eps}{2}\Big)^{t+1}.
\end{align*}

Conditioned on the event~$\mathrm{succ}_t$, we have that with probability at least $1/(t+1)$ the function~$f_s$ $\eps$-robustly generates~$\mathcal L(\phi)$.
In that event, the algorithm returns~$\mathcal L(\phi)$ with probability at least~$2/3$.
This proves the probability bound.

The sample complexity of the algorithm follows from that of \cref{lem:robustgen}.
The time and query complexities also follow, since
\begin{align*}
    f_j(x) &= \frac{f_{j-1}(x) + \sigma_j (W(a_j, b_j)f_{j-1})(x)}{2} \\
    &= \frac{f_{j-1}(x) + \sigma_j (-i)^{|a_j\circ b_j|}(-1)^{b_j\cdot x} f_{j-1}(x+a_j)}{2}
\end{align*}
and thus a query to~$f_j$ can be made using~$2^j$ queries to $f_0=f$ and $O(2^j n)$ time.
\end{proof}

\subsection{Approximate sampling from the convoluted distribution} \label{sec:Convoluted}

Next we give an algorithmic procedure that allows us to approximately sample from the convoluted distribution~$Q_f$ when given query access to~$f$.

As a first step, note that sampling from $Q_f$ would be easy if we could sample from the simpler distribution~$P_f$.
However, doing so presents some difficulties:
by Parseval's identity we have
$$\sum_{b\in \F_2^n} P_f(a, b) = \sum_{b\in \F_2^n} \frac{|\widehat{\Delta_a f}(b)|^2}{2^n \|f\|_2^4} = \frac{\|\Delta_a f\|_2^2}{2^n \|f\|_2^4},$$
which can significantly vary with $a\in \F_2^n$.
As such, even if we can (approximately) sample from the marginal distribution $P_f(a, \cdot)/(\sum_{b} P_f(a, b))$ for a given $a\in \F_2^n$, there seems to be no easy way to sample~$a$ from a distribution proportional to $\|\Delta_a f\|_2^2$ while using few queries to~$f$.

Our solution is to ignore this difficulty and instead sample $a\in \F_2^n$ uniformly at random, followed by sampling~$b$ with probability close to $|\widehat{\Delta_a f}(b)|^2$.
We thereby obtain a sample $(a, b)$ from some probability distribution~$\nu_f$ that approximates the \emph{non-probability measure} $\|f\|_2^4\cdot  P_f$ in a \emph{fairly weak sense}.
Upon convolving~$\nu_f$ with itself, this distribution gets smoothened out and we obtain the following result:

\begin{lemma}[Convoluted sampling] \label{lem:sampling}
    Let $f: \F_2^n \to \C$ be a $1$-bounded function and $\xi>0$.
    There is a randomized sampling procedure that makes $n\log n\, \poly(1/\xi)$ queries to~$f$ and, with probability at least $1-1/n^2$, samples from a probability distribution~$\mu_f$ that satisfies
    \begin{equation*}
    \big|\mu_f(F) - \|f\|_2^8 Q_f(F)\big| \leq \frac{\xi |F|}{2^n} \quad \text{for all sets $F\subseteq \F_2^{2n}$.}
    \end{equation*}
    This sampling procedure takes $n^2 \log n\, \poly(1/\xi)$ time.
\end{lemma}

Note that, unless $\|f\|_2 = 1$, the expression $\|f\|_2^8 \cdot Q_f$ is \emph{not} a probability measure.
It would then be impossible for our samplable distribution~$\mu_f$ to approximate this measure in a more obvious way such as total variation distance.
However, since all the events that are important for our algorithm correspond to isotropic sets (and thus have size at most~$2^n$), the approximation given in Lemma~\ref{lem:sampling} is essentially just as good as total variation distance for our purposes.

\begin{proof}[ of Lemma~\ref{lem:sampling}]
Without loss of generality, we may assume that $\xi \leq 1/2$ and that~$1/\xi$ is an integer, so we do not need to deal with floor functions.
Let $\eta>0$ be a small number to be specified later on.
Given $a\in \F_2^n$, we can use the Goldreich--Levin algorithm (Theorem~\ref{thm:GL}) on $\Delta_a f$ to find a set $B_a \subseteq \F_2^n$ of size at most $64/\xi^2$ which, with probability at least $1-\eta$, satisfies
\begin{equation*}
    \big\{b\in \F_2^n:\: |\widehat{\Delta_a f}(b)| \geq \xi/4\big\} \subseteq B_a \subseteq \big\{b\in \F_2^n:\: |\widehat{\Delta_a f}(b)| \geq \xi/8\big\}.
\end{equation*}
This takes $n \log n\, \poly(\xi^{-1} \log(\eta^{-1}))$ queries to~$f$ and $n^2 \log n\, \poly(\xi^{-1} \log(\eta^{-1}))$ time.

Next, using  Lemma~\ref{lem:fourest}, one can make $\poly(\xi^{-1} \log(\eta^{-1}))$ queries to~$f$ to obtain nonnegative numbers $\{\lambda_a(b): b\in B_a\}$ such that, with probability at least $1-\eta$, we have
$$\big||\widehat{\Delta_a f}(b)|^2 - \lambda_a(b)\big| \leq \xi^4/64 \quad \text{for all $b\in B_a$}.$$
Then, with probability at least $1-\eta$, we have
\begin{equation*}
    \sum_{b\in B_a} \lambda_a(b) \leq \sum_{b\in B_a} \big(|\widehat{\Delta_a f}(b)|^2 + \xi^4/64\big) \leq \|\Delta_a f\|_2^2 + \frac{\xi^4 |B_a|}{64} \leq 1 + \xi^2.
\end{equation*}
If $\sum_{b\in B_a} \lambda_a(b) > 1+\xi^2$ (which happens with probability at most~$\eta$), replace each $\lambda_a(b)$ by zero.

Now we increase~$B_a$ arbitrarily to a superset $B_a' \subseteq \F_2^n$ of size $|B_a| + 4/\xi$, and define the function $\nu_a: \F_2^n \to [0, 1]$ by
$$\nu_a(b) = \frac{\lambda_a(b)}{1+\xi^2}\, \text{ if $b\in B_a$}, \quad \nu_a(b) = \frac{\xi}{4}\Big(1- \frac{1}{1+\xi^2} \sum_{b\in B_a} \lambda_a(b)\Big)\, \text{ if $b\in B_a' \setminus B_a$,}$$
and $\nu_a(b) = 0$ if $b\notin B_a'$.
It is clear that~$\nu_a$ is a probability measure with $|\supp(\nu_a)| \leq |B_a'| \leq 68/\xi^2$ and, with probability at least $1-2\eta$, it satisfies
$$\big|\nu_a(b) - |\widehat{\Delta_a f}(b)|^2\big| \leq \frac{\xi}{4} \quad \text{for all $b\in \F_2^n$.}$$

Define the probability distribution~$\nu_f$ on~$\F_2^{2n}$ by $\nu_f(a, b) = \nu_a(b)/2^n$.
This distribution is easy to sample from:
sample $a\in \F_2^n$ uniformly at random, then compute~$\nu_a$ on $\supp(\nu_a)$, then sample $b\in \supp(\nu_a)$ according to~$\nu_a$.
Each such sample requires $n \log n\, \poly(\xi^{-1} \log(\eta^{-1}))$ queries to~$f$ and $n^2 \log n\, \poly(\xi^{-1} \log(\eta^{-1}))$ time.

Define the random set
\begin{equation*}
    A = \big\{a\in \F_2^n:\: \big|\nu_a(b) - |\widehat{\Delta_a f}(b)|^2\big| > \xi/4\, \text{ for some $b\in \F_2^n$}\big\}.  
\end{equation*}
Since $\Pr[a\in A] \leq 2\eta$ independently for all $a\in \F_2^n$, we conclude from Chernoff's bound that $\Pr\big[|A| \geq 4\eta \cdot 2^n\big] \leq 1 - 1/n^2$.
Moreover, by boundedness of~$f$ and~$\nu_a$, we have
$$\big|\nu_a(b) - |\widehat{\Delta_a f}(b)|^2\big| \leq \frac{\xi}{4} + \one_A(a) \quad \text{for all $a, b\in \F_2^n$.}$$

Now let~$F\subseteq \F_2^{2n}$ be an arbitrary set.
Writing $\tilde{P}_f(a, b) := \|f\|_2^4 P_f(a, b) = 2^{-n} |\widehat{\Delta_a f}(b)|^2$, we~have
\begin{align*}
    \big|\tilde{P}_f * (\tilde{P}_f - \nu_f) (F)\big|
    &= \bigg| \sum_{c, d\in \F_2^n} \tilde{P}_f(c,d) \sum_{(a,b) \in F} \big(\tilde{P}_f(a+c, b+d) - \nu_f(a+c, b+d)\big) \bigg| \\
    &\leq\sum_{c, d\in \F_2^n} \tilde{P}_f(c,d) \sum_{(a,b) \in F} \frac{\big||\widehat{\Delta_{a+c}f}(b+d)|^2 - \nu_{a+c}(b+d)\big|}{2^n} \\
    &\leq\sum_{c, d\in \F_2^n} \tilde{P}_f(c,d) \sum_{(a,b) \in F} \frac{\xi/4 + \one_A(a+c)}{2^n} \\
    &\leq \frac{\xi}{4} \frac{|F|}{2^n} + \frac{1}{2^n} \sum_{(a,b) \in F} \sum_{c, d\in \F_2^n} \tilde{P}_f(c,d) \one_A(a+c) \\
    &= \frac{\xi}{4} \frac{|F|}{2^n} + \frac{1}{2^n} \sum_{(a,b) \in F} \sum_{c\in a+A} \sum_{d\in \F_2^n} \tilde{P}_f(c,d).
\end{align*}
Noting that
$$\sum_{d\in \F_2^n} \tilde{P}_f(c,d) = \frac{1}{2^n} \sum_{d\in \F_2^n} |\widehat{\Delta_c f}(d)|^2 = \frac{1}{2^n} \|\Delta_c f\|_2^2 \leq \frac{1}{2^n},$$
we conclude that
$$\big|\tilde{P}_f * (\tilde{P}_f - \nu_f) (F)\big| \leq \frac{\xi}{4} \frac{|F|}{2^n} + \frac{|F|}{2^n} \frac{|A|}{2^n}.$$
Similarly, we obtain
$$\big|\nu_f * (\tilde{P}_f - \nu_f)(F)\big| \leq \frac{\xi}{4} \frac{|F|}{2^n} + \frac{|F|}{2^n} \frac{|A|}{2^n},$$
and thus
$$\big|\nu_f * \nu_f(F) - \tilde{P}_f * \tilde{P}_f(F)\big| \leq \frac{\xi}{2} \frac{|F|}{2^n} + 2\frac{|F|}{2^n} \frac{|A|}{2^n}.$$

Taking $\eta = \xi/16$ and denoting $\mu_f = \nu_f * \nu_f$, we conclude that, with probability at least $1-1/n^2$, we have
\begin{equation}
    \big|\mu_f(F) - \|f\|_2^8 Q_f(F)\big| \leq \frac{\xi |F|}{2^n} \quad \text{for all $F\subseteq \F_2^{2n}$.}
\end{equation}
Note that we can sample from~$\mu_f$ by sampling independent pairs $(a,b)$, $(c,d)$ according to~$\nu_f$ and returning $(a+c,\, b+d)$.
The result follows.
\end{proof}

\subsection{Lagrangian sampling based on query access}

Finally, here we show how to combine the idealized setting from \cref{thm:ideal_Lsampling} with approximate $Q_f$-samples to obtain \cref{thm:lagrangian}.

Let $\eps = \frac{1}{8} \big(\gamma - \tfrac{1}{2}\big)^2\tau^8$ and $\xi = \frac{1}{2} \eps \tau^{8}$.
Let~$\mu_f$ be the random probability distribution from \cref{lem:sampling} with this parameter~$\xi$, and suppose that it satisfies the conclusion of the lemma whenever we sample from this distribution.
As the total number of samples we take is $n\, \poly\big((\gamma - \tfrac{1}{2})^{-1}\tau^{-1}\big)$, this holds with probability $1-o(1)$.

\emph{Robust generation.} 
We approximately implement the algorithm from Lemma~\ref{lem:robustgen} by substituting samples from~$Q_f$ by samples from~$\mu_f$. 
The number of samples we use now depends on the value $p = \mu_f\big(\Spec(f)\big)$.
By the relationship between~$\mu_f$ and~$Q_f$ and the fact that $\|f\|_2 \geq \tau$, an analysis similar to the proof of Lemma~\ref{lem:robustgen} shows that with a factor of $\poly(1/\tau)$ more samples from~$\mu_f$ we obtain a basis for a subspace $L\leq \F_2^{2n}$ such that $L = \mathcal L(\phi)$ with probability~$2/3$.
As each sample from~$\mu_f$ costs $n\log n\, \poly(1/\xi)$ queries to~$f$ and $n^2 \log n\, \poly(1/\xi)$ time, the query and time complexities of this algorithm are $n^2 \log n\,\poly(1/\xi)$ and $n^3\log n\,\poly(1/\xi)$, respectively.

\emph{Non-robust generation.}
If~$f$ does not $\eps$-robustly generate~$\mathcal L(\phi)$, we have from Lemma~\ref{lem:nonrobust} that
\begin{align*}
    \mu_f\big(\mathcal L(\phi)\setminus\Spec(f)\big) &\geq \tau^8 Q_f\big(\mathcal L(\phi)\setminus\Spec(f)\big) - \xi \geq \xi.
\end{align*}
As in the previous case, we approximately implement $\textsc{LagrangianSampling}(f_s, \xi, \tau)$ by substituting samples from~$Q_{f_s}$ with samples from~$\mu_{f_s}$.
We then obtain a basis for a subspace $L\leq \F_2^{2n}$ that satisfies~$L = \mathcal L(\phi)$ with probability at least $\big((\gamma - \tfrac{1}{2})\tau\big)^{O(\log(1/\tau))}$.

Note that we can query each projected function~$f_i$ using~$2^i$ queries to~$f$ and $O(2^in)$ time.
A sample from $\mu_{f_i}$ therefore costs~$n\log n\,\poly\big((\gamma - \tfrac{1}{2})^{-1} \tau^{-1}\big)$ queries to~$f$ and takes $n^2 \log n\, \poly\big((\gamma - \tfrac{1}{2})^{-1} \tau^{-1}\big)$ time.
The generation of~$f_1,\dots,f_t$ has the same order of complexity.
It follows that the total algorithm uses $n^2 \log n\, \poly\big((\gamma - \tfrac{1}{2})^{-1} \tau^{-1}\big)$ queries to~$f$ and runs in time $n^3 \log n\, \poly\big((\gamma - \tfrac{1}{2})^{-1} \tau^{-1}\big)$, finishing the proof of \cref{thm:lagrangian}.

\section{The Quadratic Goldreich--Levin theorem and its corollaries}
\label{sec:quadGL}

In this section, we use our Lagrangian sampling algorithm to obtain our optimal Quadratic Goldreich--Levin theorem (\cref{thm:quadratic_GL}) and its corollaries:
the $\PGI$ algorithm (\cref{thm:algoPGI}), the optimal self-corrector for Reed-Muller codes (\cref{cor:ReedMuller}), and the quadratic decomposition algorithm (\cref{cor:decomposition}).

We begin by giving a ``list-decoding'' algorithm for stabilizer states as mentioned in \cref{sec:finding_Lagrangians}, which will be crucial for proving our main results.
This is an algorithm that, given query access to a bounded function~$f$, with high probability returns all stabilizer states that are approximate local maximizers for~$f$ and have non-negligible correlation with~$f$. 
Note that this is only possible due to the notion of approximate local maximality, as there can be $\exp(n)$-many stabilizer states that have non-negligible correlation with~$f$.
This result can be regarded as a dequantization of the quantum procedure given by~\cite[Corollary~6.2]{chen2024stabilizer}.

\begin{theorem}[List-decoding stabilizer states] \label{thm:list}
    Let $f:\F_2^n \to \C$ be a 1-bounded function and let $\tau,\, \delta>0$ and $1/2<\gamma\leq 1$.
    There is a randomized algorithm that, when given query access to~$f$, returns a list of size $\log(\delta^{-1}) \big((\gamma-1/2) \tau\big)^{-O(\log(1/\tau))}$ which, with probability at least $1-\delta$, contains all $\gamma$-approximate local maximizers that have correlation at least~$\tau$ with~$f$.
    This algorithm makes $n^2 \log n\, \log(\delta^{-1}) \big((\gamma - \tfrac{1}{2}) \tau\big)^{-O(\log(1/\tau))}$ queries to~$f$ and has runtime $n^3 \log n\, \log(\delta^{-1}) \big((\gamma - \tfrac{1}{2}) \tau\big)^{-O(\log(1/\tau))}$.
\end{theorem}

The main ingredient for the proof of this result is \cref{thm:lagrangian}, which provides an efficient algorithm that---with non-negligible probability---returns the Lagrangian subspace $\Lcal(\phi)$ associated to some fixed (but unknown) $\gamma$-approximate local maximizer of correlation~$\phi$ satisfying $|\langle f, \phi\rangle| \geq \tau$.

We next show how to learn~$\phi$ from~$\Lcal(\phi)$ (with non-negligible probability).
Note that there are~$2^n$ stabilizer states associated to the Lagrangian subspace~$\Lcal(\phi)$, and several of them can satisfy the requirements for our unknown stabilizer state~$\phi$.
Our learning algorithm will then return a random such stabilizer state~$\psi$ whose probability of being picked depends only on its correlation with~$f$.

\begin{lemma}[Stabilizer sampling]\label{lem:linear}
    Let $f: \F_2^n \to \C$ be a $1$-bounded function and let~$\phi$ be a stabilizer state with $|\langle f, \phi\rangle| \geq \tau$.
    There is a randomized algorithm which, when given a basis $\{v_1, \dots, v_n\}$ for~$\Lcal(\phi)$, returns a random stabilizer state~$\psi$ such that
    $$\Pr[\psi = \phi] \geq \tau^6/8.$$
    This algorithm makes $n \log n\, \poly(1/\tau)$ queries to~$f$ and runs in time $O(n^3) + n^2\log n\, \poly(1/\tau)$.
\end{lemma}

\begin{proof}
Since~$\Lcal(\phi)$ is a Lagrangian subspace, we can write
\begin{equation} \label{eq:Lagrangian}
    \Lcal(\phi) = \big\{(h,\, Mh+w):\: h\in V,\, w\in V^{\perp}\big\}
\end{equation}
for a subspace $V\leq \F_2^n$ and some symmetric matrix $M\in \F_2^{n\times n}$.
Moreover, from \cref{lem:stab_explicit} we conclude that (up to phases)
\begin{equation} \label{eq:phi}
    \phi(x) = 2^{(n-\dim(V))/2} \one_{u+V}(x) (-1)^{x^{\mathsf T} Qx + c\cdot x} i^{|d\circ x|},
\end{equation}
where~$Q$ is the upper-triangular part of the matrix~$M$, $d$ is the diagonal of~$M$, and $c, u\in \F_2^n$ are vectors.

From the given basis $\{v_1, \dots, v_n\}$ of~$\Lcal(\phi)$ we can obtain, in $O(n^3)$ time, a basis for the subspace~$V$ and a matrix~$M$ such that identity~\eqref{eq:Lagrangian} holds.
To completely determine~$\phi$ as in equation~\eqref{eq:phi}, it only remains to find the correct coset~$u+V$ on which it is supported and its linear part~$(-1)^{c\cdot x}$.

Since~$f$ is bounded, the codimension of~$V$ is also bounded:
$$\tau \leq |\langle f, \phi\rangle| \leq 2^{(n-\dim(V))/2} \E_{x\in \F_2^n} |f(x)| \one_{u+V}(x) \leq 2^{-(n-\dim(V))/2},$$
which implies that $n-\dim(V) \leq 2\log(1/\tau)$.
It follows that there are at most~$1/\tau^2$ cosets of~$V$ on which~$\phi$ can be supported.
Choosing a uniformly random vector~$w\in \F_2^n$, with probability at least~$\tau^2$ we obtain the correct coset $w+V = u+V$.

Now suppose we have found the correct coset $w+V$, and consider the function~$g$ given~by
$$g(x) = \one_{w+V}(x) f(x)(-1)^{x^{\mathsf T} Qx} i^{-|d\circ x|}.$$
Letting~$c\in \F_2^n$ be the (unknown) vector given in equation~\eqref{eq:phi} above, we have that
$$|\widehat{g}(c)| = \big|\E_{x\in \F_2^n} f(x) \one_{w+V}(x) (-1)^{x^{\mathsf T} Qx} i^{-|d\circ x|} (-1)^{c\cdot x}\big| = 2^{-(n-\dim(V))/2} |\langle f, \phi\rangle| \geq \tau^2.$$
Applying the Goldreich--Levin algorithm (Theorem~\ref{thm:GL}) to the function~$g$ with $\delta = 1/2$ and~$\tau$ substituted by~$\tau^2$, we obtain a list~$B\subseteq \F_2^n$ of size at most~$4/\tau^4$ which, with probability at least~$1/2$, satisfies
$$\big\{b\in \F_2^n:\: |\widehat{g}(b)| \geq \tau^2\big\} \subseteq B \subseteq \big\{b\in \F_2^n:\: |\widehat{g}(b)| \geq \tau^2/2\big\}.$$
Taking an element~$b\in B$ uniformly at random, we then get~$b=c$ with probability at least~$\tau^4/8$. In conclusion, the (random) stabilizer state
$$\psi(x) := 2^{(n-\dim(V))/2} \one_{w+V}(x) (-1)^{x^{\mathsf T} Qx + b\cdot x} i^{|d\circ x|}$$
thus obtained will be equal to~$\phi$ with probability at least~$\tau^6/8$.
\end{proof}

Combining \cref{thm:lagrangian} and \cref{lem:linear}, we obtain an algorithm that does the following:
for any fixed $\gamma$-approximate local maximizer of correlation~$\phi$ for~$f$ satisfying $|\langle f, \phi\rangle| \geq \tau$, the algorithm returns~$\phi$ with probability at least $p = \big((\gamma - \tfrac{1}{2}) \tau\big)^{O(\log(1/\tau))}$.
Since this holds for any such $\gamma$-approximate local maximizer~$\phi$, there must be at most~$1/p$ of them.
Repeating the algorithm $O\big(\tfrac{1}{p} \log\tfrac{1}{p} \log\tfrac{1}{\delta}\big)$ times then gives a list that, with probability at least $1-\delta$, contains all of them.

This algorithm makes $n^2 \log n\, \log(1/\delta) \big((\gamma - \tfrac{1}{2})^{-1} \tau^{-1}\big)^{O(\log(1/\tau))}$ queries to~$f$ and runs in time $n^3 \log n\, \log(1/\delta) \big((\gamma - \tfrac{1}{2})^{-1} \tau^{-1}\big)^{O(\log(1/\tau))}$, finishing the proof of \cref{thm:list}.

\subsection{Quadratic Goldreich--Levin}
We now use the list-decoding algorithm given in Theorem~\ref{thm:list} to construct our Quadratic Goldreich--Levin algorithm (\cref{thm:quadratic_GL}).

\begin{proof}[ of \cref{thm:quadratic_GL}]
The main idea is to apply the algorithm from Theorem~\ref{thm:list} with suitably chosen parameters to obtain a bounded-size list containing all ``good'' stabilizer states, and then replace each of these good stabilizer states by a bounded number of (classical) quadratic phase functions.
Each such quadratic phase~$(-1)^q$ is obtained from its associated stabilizer state~$\phi$ by extending its support from (a coset of) a subspace~$V$ to the whole domain~$\F_2^n$.
We end the proof by showing that, with high probability, one of the quadratic phases thus obtained has almost-maximal correlation with~$f$;
by querying~$f$ a bounded number of times, we can estimate all of these correlations and pick up the highest one.

The full algorithm is given as follows:
\begin{enumerate}
    \item Apply the algorithm from Theorem~\ref{thm:list} with parameters $\tau = \eps$ and $\gamma = 1/2 + \eps^2$.
    We obtain a list~$L$ of size $\log(1/\delta) (1/\eps)^{O(\log(1/\eps))}$ which, with probability at least $1-\delta$, contains all stabilizer states that are $(1/2+\eps^2)$-approximate local maximizers of correlation for~$f$ and have correlation at least~$\eps$ with~$f$.
    
    \item Remove from~$L$ every stabilizer state whose support has codimension larger than $2\log(1/\eps)$.
    If~$L$ becomes empty after this step, end the algorithm and return the constant function $p \equiv 0$.
    Otherwise, initialize a list~$L'$ to be empty and continue the algorithm.
    
    \item For each stabilizer state $\phi\in L$, do the following:

    \noindent Write $\phi(x) = 2^{(n-d)/2} \one_{u+V}(x) (-1)^{q(x)} i^{|c\circ x|}$, where~$V$ is a subspace of dimension~$d$, $q: \F_2^n \to \F_2$ is a quadratic polynomial, and $u, c \in \F_2^n$ are vectors such that either $c=0$ or $c\notin V^\perp$.
    (This decomposition is possible due to \cref{lem:stab_explicit} and the remark following it.)
    \begin{itemize}
        \item[$(a)$] If $c=0$, add to~$L'$ the quadratic functions $x\mapsto q(x)+y\cdot x$ for all $y\in V^\perp$.

        \item[$(b)$] If $c\notin V^\perp$, let $U = \{c\}^\perp$ and let $v\in \F_2^n$ satisfy $c\cdot v = 1$, so that any $x\in\F_2^n$ has a unique representation of the form $x = z+bv$ for some $z\in U$ and $b\in \F_2$.
        By evaluating the map $x = z+bv \mapsto i^{|c\circ z| - 2|c\circ z\circ bv|}$ on all vectors $x\in \F_2^n$ with weight $|x| \leq 2$, find the polynomial $r\in \F_2[x_1,\dots,x_n]$ of degree at most~2 such that\footnote{Such a polynomial exists because the function $x = z+bv \mapsto i^{|c\circ z| - 2|c\circ z\circ bv|}$ is a non-classical quadratic phase function taking values in $\pmset{}$, and thus equals a classical quadratic phase $(-1)^{r(x)}$.}
        $(-1)^{r(z+bv)} = i^{|c\circ z| - 2|c\circ z\circ bv|}$.
        Add to~$L'$ the quadratic functions
        $$x\mapsto r(x)+q(x)+y\cdot x \quad \text{and} \quad x\mapsto r(x)+q(x)+ (y+c)\cdot x$$
        for all $y\in V^{\perp}$.
    \end{itemize}
    
    \item Query~$f$ at $m = \poly(\tfrac{1}{\eps} \log \tfrac{1}{\delta})$ randomly chosen points $x_1, \dots, x_m \in \F_2^n$ and compute
    $$\Est_q := \frac{1}{m} \sum_{j=1}^m f(x_j) (-1)^{q(x_j)}$$
    for all quadratic functions~$q$ in~$L'$.
    Output the one that attains the maximum value of~$|\Est_q|$.
\end{enumerate}

Note that, for each $\phi\in L$, the number of quadratic functions we add to~$L'$ at step~$(3)$ is at most~$2^{n-d+1}$.
Since $n-d \leq 2\log(1/\eps)$ because of step~$(2)$, it follows that the final list~$L'$ has size at most
$$2^{n-d+1} |L| \leq 2|L|/\eps^2 = \log(1/\delta) (1/\eps)^{O(\log(1/\eps))}.$$
In addition to the list-decoding subroutine from \cref{thm:list}, the most expensive step in this algorithm is step $(3)$, which takes $O(n^3 + n/\eps^2)$ time for each stabilizer state in the list~$L$.
The query and time complexities of the algorithm above thus match those stated in \cref{thm:quadratic_GL}.

Denote the (random) quadratic function output by this algorithm by~$p$.
We will show that, with probability at least~$1-2\delta$, this function satisfies
\begin{equation} \label{eq:high_corr}
    |\langle f,\, (-1)^{p(\cdot)}\rangle| > \|f\|_{u^3} - \eps,
\end{equation}
where $\|f\|_{u^3}$ (with a minuscule~$u$) denotes the maximum correlation $|\langle f,\, (-1)^{q(\cdot)}\rangle|$ of~$f$ with a quadratic polynomial $q\in \F_2[x_1, \dots, x_n]$.
This will complete the proof of the theorem.

Note that we may assume $\|f\|_{u^3} \geq \eps$, as otherwise any quadratic function will satisfy equation~$\eqref{eq:high_corr}$.
We can also assume that~$\eps \leq 1/100$, which will allow us to bound certain expressions more easily.
The heart of the argument is given in the following result:

\begin{lemma} \label{lem:find_quad}
    Assume that~$\eps \leq 1/100$ and~$\|f\|_{u^3} \geq \eps$.
    Then, with probability at least~$1-\delta$, there exists a quadratic function~$q$ in~$L'$ satisfying
    $$|\langle f,\, (-1)^{q(\cdot)}\rangle| \geq \|f\|_{u^3} - \eps/2.$$
\end{lemma}

\begin{proof}
Let $p^*: \F_2^n \to \F_2$ be a quadratic function attaining maximum correlation with~$f$:
$$\big|\E_{x\in \F_2^n} f(x) (-1)^{p^*(x)}\big| = \|f\|_{u^3}.$$
Consider the stabilizer state $\phi_0 := (-1)^{p^*(\cdot)}$, and denote $\gamma = 1/2+\eps^2$.
If~$\phi_0$ is a $\gamma$-approximate local maximizer of correlation with~$f$, then with probability at least~$1-\delta$ it will appear in the list~$L$
(and thus~$p^*$ will appear in~$L'$).

Now suppose~$\phi_0$ is not a $\gamma$-approximate local maximizer of correlation with~$f$.
There must then exist a ``neighbor'' stabilizer state~$\phi_1$ satisfying
$$|\langle\phi_0,\, \phi_1\rangle|^2 = 1/2 \quad \text{and} \quad |\langle f,\, \phi_1 \rangle|^2 > \gamma^{-1} |\langle f,\, \phi_0 \rangle|^2.$$
If~$\phi_1$ is a $\gamma$-approximate local maximizer, then it will appear in the list~$L$ with probability at least~$1-\delta$.
Otherwise, we can keep choosing stabilizer states~$\phi_{i+1}$ satisfying
$$|\langle\phi_i,\, \phi_{i+1}\rangle|^2 = 1/2 \quad \text{and} \quad |\langle f,\, \phi_{i+1} \rangle|^2 > \gamma^{-1} |\langle f,\, \phi_i \rangle|^2$$
until at last we arrive at some~$\phi_t$ which is a $\gamma$-approximate local maximizer of correlation with~$f$.
This must stop at some point because $|\langle f,\, \phi_0 \rangle| = \|f\|_{u^3} \geq \eps$ and we always have $|\langle f,\, \phi_i \rangle| \leq \|f\|_2$ by Cauchy-Schwarz.
The final stabilizer state~$\phi_t$ will then appear in list~$L$ with probability at least~$1-\delta$, and it satisfies
\begin{equation} \label{eq:corr_phi_t}
    |\langle f,\, \phi_t \rangle|^2 > \gamma^{-t} |\langle f,\, \phi_0 \rangle|^2 = (1/2+\eps^2)^{-t} \|f\|_{u^3}^2.
\end{equation}

Let us write
$$\phi_t(x) = 2^{(n-d)/2} \one_{u+V}(x) (-1)^{q^*(x)} i^{|c\circ x|},$$
where~$V$ is a subspace of dimension $d$, $q^*: \F_2^n \to \F_2$ is a quadratic polynomial and $u, c \in \F_2^n$ are vectors such that either~$c=0$ or $c\notin V^{\perp}$.
Since
$$\eps \leq |\langle f,\, \phi_0 \rangle| \leq |\langle f,\, \phi_t \rangle| \leq 2^{(n-d)/2} \E_{x\in \F_2^n} \one_{u+V}(x) |f(x)| = 2^{-(n-d)/2},$$
we conclude that $n-d \leq 2\log(1/\eps)$, and so~$\phi_t$ will not get removed in step~$(2)$ of the~algorithm.

Next, we relate the dimension~$d$ of~$V$ with the number~$t$ of steps we took until we arrived at~$\phi_t$.
For each $0\leq i\leq t$, denote by $\dim(\phi_i)$ the dimension of the subspace on which the $i$-th stabilizer state~$\phi_i$ is supported.
Since $|\langle\phi_i,\, \phi_{i+1}\rangle|^2 = 1/2$ while $|\phi_j(\cdot)| = 2^{(n-\dim(\phi_j))/2} \one_{\supp(\phi_j)}(\cdot)$, we conclude that
$\dim(\phi_{i+1}) \geq \dim(\phi_i)-1$.
Moreover, in the case where $\dim(\phi_{i+1}) = \dim(\phi_i)-1$, the two stabilizer states~$\phi_i$ and~$\phi_{i+1}$ must be proportional to one another inside the support of~$\phi_{i+1}$.
As~$\phi_0 = (-1)^{p^*(\cdot)}$ while~$\phi_t$ has a nontrivial non-classical component~$i^{|c\circ x|}$ if $c\notin V^{\perp}$, it follows that $\dim(\phi_t) \geq \dim(\phi_0) - t + \one_{c\notin V^{\perp}}$.
We conclude that $t \geq n-d + \one_{c\notin V^{\perp}}$.

Using the fact that $\E_{y\in V^{\perp}} (-1)^{y\cdot x} = \one_V(x)$, we see that
\begin{align*}
    |\langle f,\, \phi_t \rangle|^2
    &= 2^{n-d} \big|\E_{x\in \F_2^n} f(x) \one_{V}(x+u) (-1)^{q^*(x)} i^{-|c\circ x|}\big|^2 \\
    &= 2^{n-d} \big|\E_{y\in V^{\perp}} \E_{x\in \F_2^n} f(x) (-1)^{y\cdot (x+u)} (-1)^{q^*(x)} i^{-|c\circ x|}\big|^2 \\
    &\leq 2^{n-d} \E_{y\in V^{\perp}} \big|\E_{x\in \F_2^n} f(x) (-1)^{q^*(x) + y\cdot x} i^{-|c\circ x|}\big|^2,
\end{align*}
and thus there exists some $y^* \in V^{\perp}$ such that
\begin{equation} \label{eq:ystar}
    \big|\E_{x\in \F_2^n} f(x) (-1)^{q^*(x) + y^* \cdot x} i^{-|c\circ x|}\big|^2 \geq 2^{-(n-d)} |\langle f,\, \phi_t \rangle|^2.
\end{equation}
Recall that, if $\phi_t \in L$ (which happens with probability at least~$1-\delta$), then the quadratic functions $x\mapsto r(x) + q^*(x) + y^*\cdot x$ and $x\mapsto r(x) + q^*(x) + (y^*+c)\cdot x$ will both be in~$L'$
(we will recall the definition of $r\in \F_2[x_1, \dots, x_n]$ below).
It then suffices to show that one of these functions has correlation at least $\|f\|_{u^3}-\eps/2$ with~$f$.

We separate the proof into two cases:
$c=0$ and $c\notin V^{\perp}$.
If~$c=0$, then $r\equiv 0$, and we obtain from combining~\eqref{eq:corr_phi_t} and~\eqref{eq:ystar} that
$$\big|\E_{x\in \F_2^n} f(x) (-1)^{q^*(x) + y^* \cdot x}\big|^2 \geq 2^{-(n-d)} (1/2+\eps^2)^{-t} \|f\|_{u^3}^2.$$
Using that $n-d \leq 2\log(1/\eps)$ and $t\geq n-d$, we conclude that
$$2^{-(n-d)} (1/2+\eps^2)^{-t} \geq 2^{-(n-d)} (1/2+\eps^2)^{-(n-d)} \geq (1+2\eps^2)^{-2\log(1/\eps)}.$$
This last expression is at least $(1-\eps/2)^2$ when $\eps \leq 1/100$, which implies that
$$\big|\E_{x\in \F_2^n} f(x) (-1)^{q^*(x) + y^* \cdot x}\big| \geq \|f\|_{u^3} - \eps/2$$
as wished.

In the case where $c\notin V^{\perp}$, let~$U \leq \F_2^n$ be the subspace orthogonal to~$c$, and let~$v \in \F_2^n\setminus U$.
Any $x\in \F_2^n$ can be written in a unique way as $x = z+bv$, where $z\in U$ and $b\in \F_2$.
Note that $|c\circ (z+bv)| = |c\circ z| + |c\circ bv| - 2|c\circ z\circ bv|$.
Define $h:\F_2^n\to \C$ by
$$h(z+bv) = i^{|c\circ z|  - 2|c\circ z\circ bv|} \quad \text{where $z\in U$, $b\in \F_2$.}$$
Then~$h$ is a non-classical quadratic phase function taking values in $\pmset{}$, which implies it must be classical:
there exists a polynomial $r\in \F_2[x_1,\dots,x_n]$ of degree at most~2 such that $h(x) = (-1)^{r(x)}$ for all~$x\in \F_2^n$.

For any function $g: \F_2^n \to \C$, we have
\begin{align*}
    \big|\E_{x\in \F_2^n} g(x) i^{-|c\circ x|}\big|^2
    &= \big|\E_{b\in \F_2} \E_{z\in U} g(z+bv) i^{-|c\circ (z+bv)|} \big|^2 \\
    &= \big|\E_{b\in \F_2} \E_{z\in U} g(z+bv) (-1)^{r(z+bv)} i^{-|c\circ bv|} \big|^2.
\end{align*}
By the Cauchy-Schwarz inequality, this expression is at most
\begin{equation*}
    \E_{b\in \F_2}\big|\E_{z\in U} g(z+bv)(-1)^{r(z+bv)}\big|^2.
\end{equation*}
By Parseval's identity on~$\F_2$ we get
\begin{align*}
    \E_{b\in \F_2} |\E_{z\in U} g(z+bv) (-1)^{r(z+bv)}|^2
    &= \sum_{a\in \F_2} |\E_{b\in \F_2} \E_{z\in U} g(z+bv) (-1)^{r(z+bv) + ab}|^2 \\
    &= \sum_{a\in \F_2} |\E_{b\in \F_2} \E_{z\in U} g(z+bv) (-1)^{r(z+bv) + a c\cdot (z+bv)}|^2 \\
    &= |\E_{x\in \F_2^n} g(x)(-1)^{r(x)}|^2 + |\E_{x\in \F_2^n} g(x) (-1)^{r(x) + c\cdot x}|^2,
\end{align*}
from which we conclude that
$$|\E_{x\in \F_2^n} g(x)(-1)^{r(x)}|^2 + |\E_{x\in \F_2^n} g(x) (-1)^{r(x) + c\cdot x}|^2 \geq \big|\E_{x\in \F_2^n} g(x) i^{-|c\circ x|}\big|^2.$$

Using this last inequality for the function $g(x) = f(x) (-1)^{q^*(x) + y^*\cdot x}$, we obtain
\begin{align*}
    \max\big\{|\E_{x\in \F_2^n} f(x) (-1)^{r(x) + q^*(x) + y^* \cdot x}|^2,\: &|\E_{x\in \F_2^n} f(x) (-1)^{r(x)+q^*(x) + (y^*+c) \cdot x}|^2 \big\} \\
    &\geq \frac{1}{2} \big|\E_{x\in \F_2^n} f(x) (-1)^{q^*(x) + y^* \cdot x} i^{-|c\circ x|}\big|^2 \\
    &\geq \frac{1}{2^{n-d+1}} |\langle f,\, \phi_t \rangle|^2 \\
    &\geq \frac{1}{2^{n-d+1}} \Big(\frac{1}{1/2+\eps^2}\Big)^t \|f\|_{u^3}^2,
\end{align*}
where we used inequalities~\eqref{eq:ystar} and~\eqref{eq:corr_phi_t} respectively.
Since in this case we have $t\geq n-d+1$ and $n-d\leq 2\log(1/\eps)$, the last expression is at least
$$\Big(\frac{1}{1 + 2\eps^2}\Big)^{2\log(1/\eps)+1} \|f\|_{u^3}^2 \geq (1-\eps/2)^2 \|f\|_{u^3}^2,$$
where we use that~$\eps \leq 1/100$. 
This concludes the proof of the lemma.
\end{proof}

Using our bound on the size of the list~$L'$ and the Chernoff bound we conclude that, with probability at least~$1-\delta$, we have
$$\big|\Est_q - \langle f,\, (-1)^{q(\cdot)}\rangle\big| < \eps/4 \quad \text{for all $q\in L'$.}$$
Recall that we denote by~$p$ the random polynomial output by the algorithm.
If the above estimate holds, we get
$$|\langle f,\, (-1)^{p(\cdot)}\rangle| > |\Est_p| - \frac{\eps}{4} = \max_{q \in L'} |\Est_q| - \frac{\eps}{4} > \max_{q \in L'} |\langle f,\, (-1)^{q(\cdot)}\rangle| - \frac{\eps}{2}.$$
By Lemma~\ref{lem:find_quad}, we have that $\max_{q \in L'} |\langle f,\, (-1)^{q(\cdot)}\rangle| \geq \|f\|_{u^3} - \eps/2$ with probability at least $1-\delta$.
This implies that inequality~\eqref{eq:high_corr} holds with probability at least $1-2\delta$, finishing the proof of the theorem.
\end{proof}

\subsection{Algorithmic PGI}
Next, we use the polynomial Quadratic Goldreich--Levin algorithm to obtain an algorithmic polynomial inverse theorem for the Gowers $U^3$ norm.

\begin{proof}[ of \cref{thm:algoPGI}]
By \cref{thm:pgi}, there exists a constant $c>1$ such that the following holds:
whenever a 1-bounded function $f: \F_2^n \to \C$ satisfies $\|f\|_{U^3} \geq \gamma$, there exists a quadratic polynomial $q: \F_2^n \to \F_2$ with
$\big|\E_{x\in \F_2^n} f(x) (-1)^{q(x)}\big| \geq (\gamma/2)^c$.
Apply Theorem~\ref{thm:quadratic_GL} to~$f$ with $\eps = \gamma^c/2^{c+1}$ and $\delta = 1/3$;
we obtain a quadratic polynomial~$p$ which, with probability at least~$2/3$, satisfies
$$\big|\E_{x\in \F_2^n} f(x) (-1)^{p(x)}\big| > \big|\E_{x\in \F_2^n} f(x) (-1)^{q(x)}\big| - \frac{\gamma^c}{2^{c+1}} \geq (\gamma/2)^{c+1}.$$
The result follows.
\end{proof}

\subsection{Self-correcting Reed-Muller codes}

We next obtain an optimal self-corrector algorithm for quadratic Reed-Muller codes over $\F_2^n$ that is agnostic to the error rate:

\begin{corollary}[Optimal self-correction of quadratic Reed-Muller codes] \label{cor:ReedMuller}
    There is a query algorithm~$\A$ with the following guarantees.
    Given $\eps>0$ and query access to a Boolean function $f: \F_2^n\to \F_2$, $\A$ makes $n^2 \log n\cdot (1/\eps)^{O(\log 1/\eps)}$ queries to~$f$ and, with probability at least~$2/3$, outputs a quadratic polynomial $p: \F_2^n \to \F_2$ satisfying
    $$\dist(f,\, p) < \min_{q \text{ quadratic}} \dist(f,\, q) + \eps,$$
    where $\dist$ denotes the normalized Hamming distance.
    In addition to the $\widetilde{O}_{\eps}(n^2)$ queries, algorithm~$\A$ has runtime $\widetilde{O}_{\eps}(n^3)$.
\end{corollary}

\begin{proof}
Query access to a Boolean function $f: \F_2^n \to \F_2$ gives query access to the bounded function $g(x) := (-1)^{f(x)}$.
Note that, for any Boolean function $q: \F_2^n \to \F_2$, we have
\begin{equation} \label{eq:corr_to_dist}
    \E_{x\in \F_2^n} g(x) (-1)^{q(x)} = 1 - 2\Pr_{x\in \F_2^n}[f(x) \neq q(x)] = 1-2\dist(f,\, q).
\end{equation}
Applying Theorem~\ref{thm:quadratic_GL} to~$g$ (with~$\eps$ substituted by~$\eps/4$ and $\delta=1/6$), we obtain a quadratic polynomial $p: \F_2^n \to \F_2$ which, with probability at least~$5/6$, satisfies
$$\big|\Exp_{x\in \F_2^n} g(x) (-1)^{p(x)}\big| > \max_{q \text{ quadratic}} \big|\Exp_{x\in \F_2^n} g(x) (-1)^{q(x)}\big| - \eps/4.$$
Using~$O(1/\eps^2)$ further queries to~$g$, we can differentiate (with probability at least~$5/6$) between the two cases
$$\Exp_{x\in \F_2^n} g(x) (-1)^{p(x)} \geq \eps/8 \quad \text{and} \quad \Exp_{x\in \F_2^n} g(x) (-1)^{p(x)} < -\eps/8;$$
in the latter case, we replace~$p$ by its negation~$\one+p$.
The guarantees stated in the corollary immediately follow from those of \cref{thm:quadratic_GL} together with equation~\eqref{eq:corr_to_dist}.
\end{proof}

\subsection{Quadratic decompositions}

Finally, we obtain our algorithmic structure-versus-randomness decomposition result by combining algorithmic~$\PGI$ with the framework developed by Tulsiani and Wolf~\cite{TulsianiW2014}.
This gives rise to an algorithmic ``quadratic decomposition theorem'' which efficiently decomposes a bounded function~$f$ into a sum of $\poly(1/\eps)$-many quadratic phase function, plus errors of $U^3$ norm and $L^1$ norm at most~$\eps$.

\begin{corollary}[Efficient quadratic decomposition] \label{cor:decomposition}
    There is a randomized algorithm that, when given query access to a $1$-bounded function $f: \F_2^n \to \C$ and some $\eps>0$, outputs with probability at least $2/3$ a decomposition
    $$f = c_1 (-1)^{p_1(\cdot)} + \dots + c_r (-1)^{p_r(\cdot)} + g + h$$
    where the~$c_i$ are constants, the~$p_i$ are quadratic polynomials, $r \leq \poly(1/\eps)$, $\|g\|_{U^3} \leq \eps$ and $\|h\|_1 \leq \eps$.
    The algorithm makes $n^2 \log n\cdot (1/\eps)^{O(\log 1/\eps)}$ queries to~$f$ and runs in time $n^3 \log n\cdot (1/\eps)^{O(\log 1/\eps)}$.
\end{corollary}

\begin{proof}
Denote $B := 1/(2\eps)$.
\cref{thm:algoPGI} provides an algorithm which, when given query access to a function $f: \F_2^n \to \{z\in \C:\: |z| \leq B\}$ satisfying $\|f\|_{U^3} \geq \eps$, outputs with probability at least $1-\delta$ a quadratic function $p: \F_2^n \to \F_2$ such that $|\langle f,\, (-1)^{q}\rangle| \geq \poly(\eps)$.
This algorithm makes $\widetilde{O}(n^2)$ queries to~$f$ and takes $\widetilde{O}(n^3)$ time.
The result now follows by applying \cite[Theorem~3.1]{TulsianiW2014} to this algorithm and the norm~$\|\cdot\|_{U^3}$.
\end{proof}

We note that, by replacing our use of \cite[Theorem~3.1]{TulsianiW2014} by \cite[Theorem~3.3]{kim2023cubic}, it is possible to do away with the $L^1$-error function~$h$ in this decomposition at the price of increasing the number of quadratic phase functions to $\exp(\poly(1/\eps))$.
It is at present unclear whether there exists a decomposition that attains the best of both worlds, even if one is to ignore the algorithmic aspects.

\section{Proof of the algorithmic PFR theorems} \label{sec:algo_PFR}

In this section, we use our Quadratic Goldreich--Levin theorem (Theorem~\ref{thm:quadratic_GL}) to prove an algorithmic version of the $\PFR$ theorem and its equivalent formulations.
To make the proofs clearer, we will use variations of~$P$ to denote specific positive polynomials $P: \R_+ \to \R_+$ related to our results.
For instance, we will write~$P_1$ to denote the polynomial promised to exist in the~$\PFR$ theorem (\cref{thm:marton_conjecture}):
whenever $A \subseteq \F_2^{n}$ satisfies $|A + A| \leq K |A|$, it can be covered by $P_1(K)$ translates of a subspace $V \leq \F_2^n$ of size $|V|\leq |A|$.

We begin by recalling a few results in additive combinatorics that will be needed for our algorithms.
The first is a version of the original Freiman--Ruzsa theorem with optimal bounds, which was proven by Even--Zohar~\cite{even2012sums}.

\begin{theorem}[Freiman--Ruzsa theorem]
\label{thm:spanAsize}
    Let $A\subseteq \F_2^n$ be a set with doubling constant at most~$K$:
    $|A+A|\leq K\cdot |A|$.
    Then,
    $$|\vspan(A)| \leq  \frac{2^{2K}}{2K} |A|.$$
\end{theorem}

The next three results are quite standard, and proofs for all of them can be found in Tao and Vu's textbook~\cite{tao2006additive}.
For an additive set~$S$ and $k\in \N$, we write $kS$ to denote the $k$-fold sumset of~$S$.

\begin{lemma}[Pl\"unnecke's inequality, special case]
\label{lem:4Asize}
    If $A \subseteq \F_2^{n}$ satisfies $|2A|/|A| \leq K$, then $|4A|/|A| \leq K^4$.
\end{lemma}

\begin{lemma}[Ruzsa's covering lemma]
\label{lem:ruzsacovering}
    If $S, T \subseteq \F_2^n$ satisfy $|T+S| \leq K|S|$, then there is a subset $X \subseteq T$ of size $|X| \leq K$ such that $T \subseteq X + 2S$.
\end{lemma}

\begin{theorem}[Balog--Szemer\'{e}di--Gowers theorem]
\label{thm:BSG}
    Let $A \subseteq \F_2^n$ be a set such that
    $$E(A) = \big|\big\{(x_1, x_2, x_3, x_4)\in A^4:\: x_1+x_2 = x_3+x_4\big\}\big| \geq |A|^3/K.$$
    Then, there exists a set $A' \subseteq A$ such that
    $$|A'| \geq |A|/P_{BSG}^{(1)}(K) \quad \text{and} \quad |A'+A'| \leq P_{BSG}^{(2)}(K) |A|.$$
\end{theorem}

\subsection{Algorithmic dense model and sparse set localization}
\label{sec:lemmas}

We now provide a couple of algorithmic primitives that will be needed for our main results.
The first is an efficient randomized algorithm for ``localizing'' a sparse set $A \subset \F_2^n$.

\begin{lemma}[Sparse set localization]
\label{lem:spansample}
   Let $\varepsilon, \delta >0$ and $A \subseteq \F_2^n$.
   Let $m = \log |\vspan(A)|$ and $k = \lceil 2m/\varepsilon\rceil \cdot \lceil\log(1/\delta)\rceil$.
    If $v_1, \dots, v_k$ are uniformly random elements of~$A$, then, with probability at least $1-\delta$, we have
    $$|A\cap \vspan(\{v_1, \dots, v_k\})| \geq (1-\varepsilon)|A|.$$
\end{lemma}

\begin{proof}
Let $\ell \geq 2$ be an integer to be chosen later, and let $v_1, \dots, v_\ell$ be $\ell$ independent random elements of $A$.
Let $V_0 = \{0\}$ and, for each $1\leq i\leq \ell$, denote the linear span of the first $i$ random elements $v_1, \dots, v_i$ by $V_i$. Suppose first that
\begin{equation}
\label{eq:probbound}
    \Pr_{v_1, \dots, v_\ell\in A}\big[|A\cap V_\ell| \geq (1-\varepsilon) |A|\big] < 1/2.
\end{equation}
Then $\Pr_{v_1, \dots, v_\ell\in A}\big[|A\setminus V_\ell| > \varepsilon |A|\big] \geq 1/2$, and so
\begin{align}
    \label{eq:lowerboundonA-Vi}
    \Pr_{v_1, \dots, v_i\in A}\big[|A\setminus V_i| \geq \varepsilon |A|\big] > 1/2 \quad \text{for all $0\leq i \leq \ell$.}
\end{align}
It follows that
\begin{align*}
    \Exp_{v_1, \dots, v_\ell\in A} \big[\dim(V_\ell)\big] &= \sum_{i=1}^\ell \Pr_{v_1, \dots, v_i\in A}\big[v_i \notin V_{i-1}\big] \\
    &\geq \sum_{i=1}^\ell \Pr_{v_1, \dots, v_{i-1}\in A}\Big[|A\setminus V_{i-1}| \geq \varepsilon |A|\Big]\\
    &\qquad\quad \cdot \hspace{1mm} \Pr_{v_1, \dots, v_i\in A}\Big[v_i \notin V_{i-1}\: :\: |A\setminus V_{i-1}| \geq \varepsilon |A|\Big] \\
    &> \sum_{i=1}^\ell 1/2 \cdot \varepsilon
    = \varepsilon \ell/2,
\end{align*}
where the final inequality used equation~\eqref{eq:lowerboundonA-Vi}. 
Since $V_\ell \subseteq \vspan(A)$, we must have $\dim(V_\ell) \leq \log |\vspan(A)| = m$, and thus $\ell < 2m/\varepsilon$ is required for equation~\eqref{eq:probbound} to hold. Denoting $t := \lceil 2m/\varepsilon\rceil$, we conclude~that
\begin{equation*}
    \Pr_{v_1, \dots, v_t\in A}\big[|A\cap \langle v_1, \dots, v_t\rangle| \geq (1-\varepsilon) |A|\big] \geq 1/2.
\end{equation*}

Repeating this sampling process $\lceil\log(1/\delta)\rceil$ times independently at random, the probability that we succeed at least once is at least $1-\delta$.
Setting $k = t \lceil\log(1/\delta)\rceil$, it follows that
$$
\Pr_{v_1, \dots, v_k\in A}\Big[|A\cap \langle v_1, \dots, v_k\rangle| \geq (1-\varepsilon) |A|\Big] \geq 1-\delta,
$$
proving the lemma.
\end{proof}

We next wish to obtain a \emph{dense model} of a set $A\in \F_2^n$ with bounded doubling constant~$K$;
that is, a set $S \subseteq \F_2^m$ that is ``additively equivalent'' to $A$ (as captured by the notion of Freiman isomorphisms) but which has density at least $\poly(1/K)$ inside its ambient space $\F_2^m$.
Recall the notions of Freiman homomorphism and isomorphism:

\begin{definition}[Freiman homomorphism]
For a set $A\subseteq \F_2^n$, a function $\phi: A\rightarrow \F_2^m$ is a \emph{Freiman homomorphism} if, for every additive quadruple $x_1, x_2, x_3, x_4\in A$ such that $x_1+x_2=x_3+x_4$, we have that $\phi(x_1)+\phi(x_2)=\phi(x_3)+\phi(x_4)$.
\end{definition}

\begin{definition}[Freiman isomorphism]
   A \emph{Freiman isomorphism} is a bijective Freiman homomorphism $\phi$ such that its inverse is also a Freiman homomorphism.
\end{definition}

We obtain our algorithmic dense model by showing that, for a suitable choice of~$m$, a uniformly random linear map $\pi: \F_2^n \to \F_2^m$ is a Freiman isomorphism from $A$ to $\pi(A)$ with high probability.

\begin{lemma}[Algorithmic dense model]
\label{lem:randomFreiman}
    Let $\delta > 0$,  $A\subseteq \F_2^n$ and  let $m \geq \log |4A| + \log 1/\delta$ be an integer.
    Suppose $\pi: \F_2^n \rightarrow \F_2^m$ is a random linear map.
    Then, $A$ is Freiman-isomorphic to $\pi(A)$ with probability at least $1-\delta$.
\end{lemma}

\begin{proof}
One easily sees that $\pi$ is a Freiman isomorphism between $A$ and $\pi(A)$ iff
$$
\forall a, b, c, d\in A:\: a+b+c+d = 0 \iff \pi(a)+\pi(b) + \pi(c)+\pi(d) = 0.
$$
(Note that this property implies that~$\pi$ is bijective.)
If $\pi: \F_2^n \rightarrow \F_2^m$ is a linear map, then the forward implication is automatically satisfied, and moreover
$$\pi(a)+\pi(b) + \pi(c)+\pi(d) = \pi(a+b+c+d).$$
It then suffices to check that
$$\forall a, b, c, d\in A:\: \pi(a+b+c+d) = 0 \implies a+b+c+d = 0,$$
which is equivalent to requiring that $\pi(x) \neq 0$ for all nonzero $x\in 4A$.

Now let $\pi: \F_2^n \rightarrow \F_2^m$ be a uniformly random linear map.
Then, for each $x\in 4A\setminus\{0\}$ individually, $\pi(x)$ is uniformly distributed over $\F_2^m$.
It follows from the union bound that
$$
\Pr\big[\exists x\in 4A\setminus\{0\}:\: \pi(x) = 0\big] \leq \frac{|4A|-1}{2^m},
$$
which is less than $\delta$ if $2^m \geq |4A|/\delta$.
This concludes the proof.
\end{proof}

\subsection{Algorithmic restricted homomorphism}
We will need the following restricted version of a ``homomorphism-testing'' formulation of the $\PFR$ theorem.
For completeness, we include a proof based on~\cite[Proposition~2.6]{green_montreal}, where we replace the use of the Freiman--Ruzsa theorem with the $\PFR$ theorem to obtain polynomial bounds.

\begin{lemma}[Restricted homomorphism testing]
\label{lem:closeaffine}
Suppose $S\subseteq \F_2^m$ and $f: S \to \F_2^n$ satisfy
$$\big|\big\{(x_1, x_2, x_3, x_4)\in S^4:\: x_1+x_2 = x_3+x_4 \,\text{ and }\, f(x_1)+f(x_2) = f(x_3)+f(x_4)\big\}\big| \geq 2^{3m}/K.$$
Then, there exists an affine-linear function $\psi: \F_2^m \to \F_2^n$ such that $f(x) = \psi(x)$ for at least $2^m/P_2(K)$ values of $x\in S$.
\end{lemma}

\begin{proof}
Consider the graph set
$$\Gamma = \big\{(x, f(x)):\: x\in S\big\} \subseteq \F_2^{m+n}.$$
Then $|\Gamma| = |S| \leq 2^m$ and $E(\Gamma) \geq 2^{3m}/K \geq |\Gamma|^3/K$.
By the Balog-Szemer\'{e}di-Gowers theorem (\cref{thm:BSG}), there exists a set $\Gamma' \subseteq \Gamma$ such that
$$|\Gamma'| \geq |\Gamma|/P_{BSG}^{(1)}(K) \quad \text{and} \quad |\Gamma'+\Gamma'| \leq P_{BSG}^{(2)}(K)\cdot |\Gamma|.$$
By the combinatorial $\PFR$ theorem (Theorem~\ref{thm:marton_conjecture}), we have that $\Gamma'$ can be covered~by
$$K' := P_1\big(P_{BSG}^{(1)}(K) P_{BSG}^{(2)}(K)\big)$$
translates of a subspace $H\leq \F_2^{m+n}$ of size $|H| \leq |\Gamma'|$, say
$$
\Gamma' \subseteq \bigcup_{i=1}^{K'} (u_i + H).
$$
Let $\pi: \F_2^{m+n} \to \F_2^m$ denote the projection map onto the first $m$ coordinates:
$\pi(x, y) = x$ for $x\in \F_2^m$, $y\in \F_2^n$.
Let $\ker_H(\pi) = H \cap \big(\{0^m\} \times \F_2^n\big)$ be the kernel of $\pi$ restricted to $H$ and let $H'$ be a complemented subspace of $\ker_H(\pi)$ in $H$, so that $H = H' \oplus \ker_H(\pi)$.
By linearity and the injectivity of $\pi$ on $H'$, there exists a matrix $M\in \F_2^{n\times m}$ such that
\begin{equation}
\label{eq:Hprime}
    H' = \big\{(x, Mx):\: x\in \pi(H)\big\},
\end{equation}
and by the rank-nullity theorem we have that
$$|H| = |\ker_H(\pi)|\cdot |\pi(H)|.$$
Moreover, since $\Gamma'$ is a graph, for each $i\in [K']$, we have
$$\big|\Gamma' \cap (u_i+H)\big| = \big|\pi\big(\Gamma' \cap (u_i+H)\big)\big| \leq |\pi(u_i+H)| = |\pi(H)|,$$
and thus
$$|\Gamma'| = \Bigg|\Gamma'\cap \bigcup_{i=1}^{K'} (u_i+H)\Bigg| \leq \sum_{i=1}^{K'} \big|\Gamma'\cap (u_i+H)\big| \leq K' |\pi(H)|,$$
from which we conclude that $|\pi(H)| \geq |\Gamma'|/K'$.
Finally, since $H = H' \oplus \ker_H(\pi)$, we~have
$$|\Gamma'| = \Bigg|\Gamma'\cap \bigcup_{i=1}^{K'} \bigcup_{v\in \ker_H(\pi)} (u_i+v+H')\Bigg| \leq \sum_{i=1}^{K'} \sum_{v\in \ker_H(\pi)} \big|\Gamma'\cap (u_i+v+H')\big|.$$
There must then exist some translate $u_i$ and some $v\in \ker_H(\pi)$ such that
$$\big|\Gamma'\cap (u_i+v+H')\big| \geq \frac{|\Gamma'|}{K' |\ker_H(\pi)|}.$$
Using the assumption $|H| \leq |\Gamma'|$, the identity $|H| = |\ker_H(\pi)|\cdot |\pi(H)|$ and the bound $|\pi(H)| \geq |\Gamma'|/K'$, we conclude from the last inequality that
$$\big|\Gamma'\cap (u_i+v+H')\big| \geq \frac{|H|}{K' |\ker_H(\pi)|} = \frac{|\pi(H)|}{K'} \geq \frac{|\Gamma'|}{(K')^2}.$$

We can now easily conclude.
Fixing $u_i = (x_1, y_1)$, $v = (x_2, y_2) \in \F_2^{m+n}$ such that the above inequality holds, we obtain from the description of $H'$ (equation~\eqref{eq:Hprime}) that
\begin{align*}
    \Gamma'\cap (u_i+v+H') &= \Gamma' \cap \big\{(x+x_1+x_2,\, Mx+y_1+y_2):\: x\in \pi(H)\big\} \\
    &= \Gamma' \cap \big\{(x,\, Mx + Mx_1 + Mx_2 + y_1 + y_2):\: x\in \pi(H)+x_1+x_2\big\}.
\end{align*}
There must then be at least $|\Gamma'|/(K')^2$ values of $x\in S$ such that $(x, f(x)) \in \Gamma'$ and
$$f(x) = Mx + Mx_1 + Mx_2 + y_1 + y_2.$$
Denote $\psi(x) = Mx + Mx_1 + Mx_2 + y_1 + y_2$.
Recalling that $|\Gamma'| \geq |\Gamma|/P_{BSG}^{(1)}(K)$ and
$$2^{3m}/K \leq E(\Gamma) \leq |\Gamma|^3,$$
we obtain that $f(x) = \psi(x)$ for at least
$$\frac{|\Gamma'|}{(K')^2} = \frac{|\Gamma'|}{P_1\big(P_{BSG}^{(1)}(K) P_{BSG}^{(2)}(K)\big)^2} \geq \frac{2^m}{K P_{BSG}^{(1)}(K) P_1\big(P_{BSG}^{(1)}(K) P_{BSG}^{(2)}(K)\big)^2}$$
values of $x\in S$.
Taking $P_2(K) = K P_{BSG}^{(1)}(K) P_1\big(P_{BSG}^{(1)}(K) P_{BSG}^{(2)}(K)\big)^2$ concludes the proof of the lemma.
\end{proof}

We next provide an algorithmic version of this last lemma, which relies on our optimal Quadratic Goldreich--Levin theorem (Theorem~\ref{thm:quadratic_GL}).

\begin{lemma}[Algorithmic restricted homomorphism]
\label{lem:findingaffine}
Suppose $S\subseteq \F_2^m$ and $f: S \to \F_2^n$ satisfy
$$\big|\big\{(x_1, x_2, x_3, x_4)\in S^4:\: x_1+x_2 = x_3+x_4 \,\text{ and }\, f(x_1)+f(x_2) = f(x_3)+f(x_4)\big\}\big| \geq 2^{3m}/K.$$
There is a randomized algorithm that makes $K^{O(\log K)} (m+n)^2 \log (m+n)$ queries to $S$ and to $f$, runs in $K^{O(\log K)} (m+n)^3 \log (m+n)$ time and, with probability at least $0.7$, returns $M\in \F_2^{n\times m}$, $v\in \F_2^n$ such that
$$\big|\big\{x\in S:\: f(x) = Mx+v\big\}\big| \geq 2^m/P_2'(K).$$
\end{lemma}

\begin{proof}
Define the function $g: \F_2^{m+n} \to \{-1, 0, 1\}$ by
$$g(x, y) = \one_S(x)\cdot (-1)^{f(x)\cdot y}.$$
Note that one query to $g$ can be made using one query to $S$, one query to $f$, and $O(n)$ time.
We first show that $g$ correlates well with a quadratic function:

\begin{claim}
\label{claim:highu3}
    There exists a quadratic polynomial $p: \F_2^{m+n} \to \F_2$ such that
    $$\Big|\Exp_{\substack{x\in \F_2^m,\\ y\in \F_2^n}} g(x,y) (-1)^{p(x,y)}\Big| \geq \frac{1}{P_2(K)},$$
    where $P_2(\cdot)$ is the polynomial promised by Lemma~\ref{lem:closeaffine}.
\end{claim}

\begin{proof}
From Lemma~\ref{lem:closeaffine}, we know there exists an affine-linear function $\psi: \F_2^m \to \F_2^n$ such that
$$\Pr_{x\in \F_2^m} \big[x\in S \,\text{ and }\, f(x) = \psi(x)\big] \geq \frac{1}{P_2(K)}.$$
Let $E$ be the set where $f$ and $\psi$ agree:
$$E = \big\{x\in S:\: f(x) = \psi(x)\big\}.$$
Note that $g(x, y) = (-1)^{\psi(x)\cdot y}$ for all $x\in E$, $y\in \F_2^n$, and so by Cauchy-Schwarz
\begin{align*}
    \Exp_{x\in \F_2^m} \one_E(x)
    &= \Exp_{x\in \F_2^m} \Big(\one_E(x) \cdot \Exp_{y\in \F_2^n} g(x,y) (-1)^{\psi(x)\cdot y}\Big) \\
    &\leq \Big(\Exp_{x\in \F_2^m} \one_E(x)^2 \Big)^{1/2} \Big(\Exp_{x\in \F_2^m}\Big(\Exp_{y\in \F_2^n} g(x,y) (-1)^{\psi(x)\cdot y}\Big)^2\Big)^{1/2} \\
    &= \Big(\Exp_{x\in \F_2^m} \one_E(x) \Big)^{1/2} \Big(\Exp_{x\in \F_2^m} \Exp_{y, y'\in \F_2^n} g(x,y) g(x,y') (-1)^{\psi(x)\cdot (y+y')}\Big)^{1/2} \\
    &= \Big(\Exp_{x\in \F_2^m} \one_E(x) \Big)^{1/2} \Big(\Exp_{x\in \F_2^m} \Exp_{z\in \F_2^n} \one_S(x) (-1)^{f(x) \cdot z} (-1)^{\psi(x)\cdot z}\Big)^{1/2}.
\end{align*}
We conclude that
$$\Big|\Exp_{x\in \F_2^m} \Exp_{z\in \F_2^n} g(x,z) (-1)^{\psi(x)\cdot z}\Big| \geq \Exp_{x\in \F_2^m} \one_E(x) = \Pr_{x\in \F_2^m}[x\in E] \geq \frac{1}{P_2(K)}.$$
The quadratic function $p: (x,z) \mapsto \psi(x)\cdot z$ thus satisfies the claim.
\end{proof}

We now use the Quadratic Goldreich--Levin theorem (Theorem~\ref{thm:quadratic_GL}) with $f$ replaced by $g$ and $\varepsilon := 1/(2P_2(K))$.
We conclude that, in $(m+n)^3 \log(m+n) \cdot K^{O(\log(K))}$ time and using $(m+n)^2 \log(m+n) \cdot K^{O(\log(K))}$ queries to $g$, we can obtain a quadratic function $q: \F_2^{m+n} \rightarrow \F_2$ which satisfies the following with probability at least $0.9$:
\begin{equation}
\label{eq:highcorrelation}
    \Big|\Exp_{x\in \F_2^m,\, y\in \F_2^n} \one_S(x) (-1)^{f(x)\cdot y} (-1)^{q(x,y)}\Big| \geq \frac{1}{2 P_2(K)}.
\end{equation}
Assume that this inequality holds, and write
$$q(x, y) = (x,y)^\tp A (x,y) + u\cdot x + u'\cdot y + b,$$
where $A \in \F_2^{(m+n) \times (m+n)}$, $u\in \F_2^m$, $u'\in \F_2^n$ and $b\in \F_2$.
Denote the $(m\times n)$-submatrix of $A$ defined by its first $m$ rows and last $n$ columns by $A_{12}$, and the $(n\times m)$-submatrix of $A$ defined by its last $n$ rows and first $m$ columns by $A_{21}$.
We claim that $f$ agrees often with an affine-linear function whose linear part equals $(A_{12}^\tp + A_{21}) x$:

\begin{claim}
    If equation~\eqref{eq:highcorrelation} holds, then there exists some $z_0\in \F_2^n$ such that
    \begin{equation}
    \label{eq:interim_agreement_funcs}
    \big|\big\{x\in S:\: f(x) = (A_{12}^\tp + A_{21}) x + z_0\big\}\big| \geq \frac{2^m}{64 P_2(K)^3}.
\end{equation}
\end{claim}

\begin{proof}
Define the bilinear form $B: \F_2^m \times \F_2^n \to \F_2$ by
$$B(x,y) = q(x, y) - q(x, 0) - q(0, y) + q(0, 0).$$
From the definition of~$q$, one easily checks that $B(x,y) = y^\tp (A_{12}^\tp + A_{21}) x$.

Denote $\sigma := 1/(2 P_2(K))$ and $M := A_{12}^\tp + A_{21}$ for convenience, so that $B(x,y) = Mx \cdot y$.
Plugging in
$$q(x,y) = Mx\cdot y + q(x,0) + q(0,y) - q(0,0)$$
into  equation~\eqref{eq:highcorrelation}, we obtain
$$\Big|\Exp_{x\in \F_2^m} \Exp_{y\in \F_2^n} \one_S(x) (-1)^{f(x)\cdot y} (-1)^{Mx\cdot y} (-1)^{q(x,0) + q(0,y) - q(0,0)}\Big| \geq \sigma.$$
By the triangle inequality, we conclude that
\begin{equation*}
\sum_{x \in S} \Big| \Exp_{y \in \mathbb{F}_2^n} (-1)^{f(x)\cdot y} (-1)^{Mx\cdot y} (-1)^{q(0,y)} \Big| \geq \sigma\cdot 2^m.
\end{equation*}

Defining the function $h: \F_2^n \rightarrow \{-1,1\}$ by $h(y)=(-1)^{q(0, y)}$, one can rewrite the last equation as
$$\sum_{x \in S} \big| \widehat{h}\big(f(x) + Mx\big) \big| \geq \sigma\cdot 2^m.$$
Since $|\widehat{h}(z)| \leq 1$ for all $z\in \F_2^n$, this implies that there exist at least $(\sigma/2)\cdot 2^m$ many $x \in S$ such that $\big| \widehat{h}\big(f(x) + Mx\big) \big| \geq \sigma/2$.
Let us define the set $T = \big\{z \in \F_2^n:\: |\widehat{h}(z)| \geq \sigma/2\big\}$, so that
$$\big|\big\{x\in S:\: f(x)+Mx \in T\big\}\big| \geq \frac{\sigma 2^m}{2}.$$
Then
$$\frac{\sigma 2^m}{2} \leq \sum_{z\in T} \big|\big\{x\in S:\: f(x)+Mx = z\big\}\big| \leq |T| \cdot \max_{z_0\in T} \big|\big\{x\in S:\: f(x)+Mx = z_0\big\}\big|.$$
Since $h$ is a Boolean function, by Parseval we have that
$$1 = \sum_{z \in \F_2^n} |\widehat{h}(z)|^2 \geq \sum_{z \in T} (\sigma/2)^2,$$
and thus $|T| \leq 4/\sigma^2$.
We conclude there exists some $z_0\in T$ such that
$$\big|\big\{x\in S:\: f(x)+Mx = z_0\big\}\big| \geq \frac{1}{|T|} \frac{\sigma 2^m}{2} \geq \frac{\sigma^3 2^m}{8},$$
which proves the claim.
\end{proof}

It now suffices to find such a vector $z_0\in \F_2^n$ such that  equation~\eqref{eq:interim_agreement_funcs} holds.
We do this by sampling $x_1, x_2, \ldots, x_t$ uniformly at random from $\F_2^m$, checking whether $x_i \in S$ and then computing the difference $d(x_i) = f(x_i) - (A_{12}^\tp + A_{21}) x_i$. 
For each $z \in \{d(x_i)\}_{i \in [t]}$, we then estimate $\Pr_{x \in S}\big[f(x) = (A_{12}^\tp + A_{21})x + z\big]$ and output the value~$z^*$ which maximizes the agreement.
To complete the argument, let us now comment on the value of $t$ required to determine a good value of $z^*$.
First, note that equation~\eqref{eq:interim_agreement_funcs} implies
$$\Pr_{x \in \F_2^m}[d(x) = z_0] \geq \frac{1}{64 P_2(K)^3}.$$
Thus, by sampling $t = O(P_2(K)^3)$ times, we ensure that $v_0 \in \{d(x_i)\}_{i \in [t]}$ with probability at least $0.9$.
Finally, we determine $z^*$ as mentioned before by estimating $\Pr_{x \in S}\big[f(x) = (A_{12}^\tp + A_{21})x + z\big]$ for each $z \in \{d(x_i)\}_{i \in [t]}$, which can be done up to error $1/(128 P_2(K)^3)$ with probability at least $1-0.1/t$ using an empirical estimator\footnote{For Boolean functions $f,g$, one can estimate $\Pr_{x}[f(x) = g(x)]$ up to error $\varepsilon$ with probability at least $1-\delta$ using the empirical estimate
$\mathrm{Est}_m := \frac{1}{m} \sum_{j =1}^m f(x_j) g(x_j)$,
which can be computed by querying $f,g$ at uniformly random  $x_1,\ldots,x_m \in \F_2^n$ and for $m = \poly(1/\varepsilon \log(1/\delta))$.}
that uses $O(\log(K) P_2(K)^3)$ samples from $\F_2^n$ and queries to $S$ and $f$ for each $i \in [t]$.
In total, this procedure consumes $O(\log(K) P_2(K)^{6})$ queries to $S$ and to $f$, and succeeds with probability at least $0.8$ (after taking the union bound).

We then return $M = A_{12}^\tp + A_{21}$ and $v = z^*$ as given above.
With probability at least $0.7$, the guarantee of the statement is satisfied with $P_2'(K) = 128 P_2(K)^3$.
The overall query and time complexities of the algorithm are dominated by the complexity of the algorithm in Theorem~\ref{thm:quadratic_GL}. This completes the proof of \cref{lem:findingaffine}.
\end{proof}

\subsection{Algorithmic PFR theorems}
\label{sec:classical}

We are finally ready to prove our algorithmic versions of the $\PFR$ theorem, corresponding to its equivalent formulations given in \cite[Proposition~10.2]{green2004finite}.\footnote{Note that formulations~$(1)$ and~$(3)$ in this proposition immediately follow from formulation~$(2)$, and will thus be omitted.}
We start with the original formulation, corresponding to our \cref{thm:APFR}, which is restated more precisely below.

\begin{theorem}[Algorithmic $\PFR$]
\label{thm:algoPFR1}
    Suppose $A \subseteq \mathbb{F}_2^n$ satisfies $|A + A| \leq K|A|$.
    There is a randomized algorithm that takes $O(\log|A| + K)$ random samples from $A$, makes $2^{O(K)} (\log |A|)^2 \log\log |A|$ queries to $A$, runs in time $K^{O(\log K)} n^4 \log n + 2^{O(K)} n^3 \log n$ and has the following guarantee:
    with probability at least $2/3$, it outputs a basis for a subspace $V \leq \mathbb{F}_2^n$ of size $|V| \leq |A|$ such that $A$ can be covered by $P_1'(K)$ translates of $V$.
\end{theorem}

\begin{proof}
We first describe the algorithm to find $V$: 
\begin{enumerate}
    \item Sample $t = 28 \log|A| + 56K$ uniformly random elements from $A$, and denote their linear span by $U$.
    Let $A' := A\cap U$.
    \item Take a random linear map $\pi: U \to \mathbb{F}_2^m$ where $m = \log|A| + 4\log K + 10$. Let $S = \pi(A')$ denote the image of $A'$ under $\pi$, and let $f: S \to U$ be the inverse of $\pi$ when restricted to $S$.\footnote{In our analysis we show that this inverse is well-defined with high probability.}
    \item Apply Lemma~\ref{lem:findingaffine} to obtain an affine-linear map $\psi: \mathbb{F}_2^m \to U$ such that $f(x) = \psi(x)$ for at least $|A|/P_2'\big(2^{34}K^{13}\big)$ values $x \in S$.
    \item Take a subspace $V$ of $\textrm{Im}(\psi) + \psi(0)$ having size at most~$|A|$, and output a basis for~$V$.
\end{enumerate}

We proceed to analyze the correctness and complexity of this algorithm.
For Step $(1)$, note that Theorem~\ref{thm:spanAsize} directly implies that $|\vspan(A)|\leq 2^{2K}\cdot |A|$.
Now, by our choice of $t$, Lemma~\ref{lem:spansample} implies that  $|A'| \geq |A|/2$ with probability at least $0.99$.
Supposing this is the case, we have that
$$
|A' + A'|\leq |A+A|\leq  K|A| \leq 2K |A'|.
$$
Moreover, by Lemma~\ref{lem:4Asize} we conclude that  $|4A'| \leq |4A| \leq K^4 |A| \leq 2K^4 |A'|$.

For Step $(2)$, note that Lemma~\ref{lem:randomFreiman} shows that, with probability at least $0.99$, $\pi$ is a Freiman isomorphism from $A'$ to $S = \pi(A')$.
In this case, the inverse map $f: S\to A'$ is a Freiman isomorphism and $|S| = |A'|$.

In Step $(3)$ we wish to apply Lemma~\ref{lem:findingaffine}, which requires us to bound from below the quantity
$$\big|\big\{(x_1, x_2, x_3, x_4)\in S^4:\: x_1+x_2 = x_3+x_4 \,\text{ and }\, f(x_1)+f(x_2) = f(x_3)+f(x_4)\big\}\big|.$$
We claim that this is at least $|A'|^3/(2K)$:

\begin{claim}
    If $f: S \to A'$ is a Freiman isomorphism and $|A'+A'| \leq 2K |A'|$, then
    $$\big|\big\{(x_1, x_2, x_3, x_4)\in S^4:\: x_1+x_2 = x_3+x_4 \,\text{ and }\, f(x_1)+f(x_2) = f(x_3)+f(x_4)\big\}\big| \geq \frac{|A'|^3}{2K}.$$
\end{claim}

\begin{proof}
If $f$ is a Freiman isomorphism, then the quantity above equals
\begin{align*}
    &\big|\big\{(x_1, x_2, x_3, x_4)\in S^4:\: f(x_1)+ f(x_2) = f(x_3)+f(x_4)\big\}\big| \\
    &\qquad = \big|\big\{(y_1, y_2, y_3, y_4)\in (A')^4:\: y_1+y_2 = y_3+y_4\big\}\big| \\
    &\qquad = E(A').
\end{align*}
Note that
$$\sum_{z\in 2A'} \big|\big\{(y_1, y_2)\in (A')^2:\: y_1+y_2 = z\big\}\big| = |A'|^2$$
and
\begin{align*}
    &\sum_{z\in 2A'} \big|\big\{(y_1, y_2)\in (A')^2:\: y_1+y_2 = z\big\}\big|^2 \\
    &\qquad = \sum_{z\in 2A'} \big|\big\{(y_1, y_2, y_3, y_4)\in (A')^4:\: y_1+y_2 = z = y_1+y_2\big\}\big|^2 \\
    &\qquad = E(A'),
\end{align*}
hence, by Cauchy-Schwarz,
$$|A'|^2 \leq |2A'|^{1/2} E(A')^{1/2} \implies E(A') \geq |A'|^4/|2A'|.$$
The claim now follows from the assumption $|2A'| \leq 2K |A'|$.
\end{proof}

Next we note that, by assumption and by our choice for $m$, we have
$$|A'| \geq \frac{|A|}{2} \geq \frac{2^m}{2^{11}K^4}.$$
From the claim above, we conclude that $S$ and $f$ satisfy the hypothesis of Lemma~\ref{lem:findingaffine} with $K$ substituted by $K' := 2^{34}K^{13}$.
We then obtain an affine-linear map $\psi: \F_2^m \to U$ such that, with probability at least $0.7$,
\begin{align}
    \label{eq:promiseofpsi=phi}
    \big|\big\{x\in S:\: f(x) = \psi(x)\big\}\big| \geq \frac{2^m}{P_2'\big(2^{34}K^{13}\big)}.
\end{align}
It remains to argue how one can simulate queries to $S$ and $f$, as required by the statement of Lemma~\ref{lem:findingaffine}.
To this end, observe that we have a full description of the linear map~$\pi: U\to \F_2^m$, so in time $O(m^2 n)$ we can find $\ker(\pi)= \{v\in U:\: \pi(v)=0\}$.
We first make three observations about this:
$(a)$ $\ker(\pi)$ is a subspace of size
$$\frac{|U|}{|\mathrm{Im}(\pi)|} \leq \frac{|\vspan(A)|}{|S|} \leq \frac{2|\vspan(A)|}{|A|} \leq 2^{2K},$$
where we used Theorem~\ref{thm:spanAsize} in the final inequality;
$(b)$ for every $x\in \mathrm{Im}(\pi)$, we have that $\pi^{-1}(x)$ is a translate of $\ker(\pi)$;
$(c)$ in $O(m^2 n)$ time, we can find the inverse map $\pi^{-1}: \mathrm{Im}(\pi) \to U/\ker(\pi)$.
Using item $(b)$, we can check whether $x\in S$ (i.e., $\pi^{-1}(x) \cap A \neq \emptyset$) by enumerating over all $y\in \pi^{-1}(x)$ and checking if $y\in A$ or not.
By item $(a)$, this takes at most $2^{2K}$ queries to $A$.
Hence, after computing $\ker(\pi)$ and $\pi^{-1}$, one can make one query to $S$ and to $f$ using $2^{2K}$ queries to $A$ and $O(mn + 2^{2K}n)$ time.

Now define the affine subspace $V' = \mathrm{Im}(\psi)$.
By definition, we have that $|V'| \leq 2^m \leq 2^{10} K^4|A|$.
Since $f: S \to A'$ is injective, from equation~\eqref{eq:promiseofpsi=phi} we conclude that
$$|A\cap V'| = |\mathrm{Im}(f) \cap \mathrm{Im}(\psi)| \geq \big|\big\{x\in S:\: f(x) = \psi(x)\big\}\big| \geq \frac{2^m}{P_2'\big(2^{34}K^{13}\big)}.$$
It follows that
$$|A + (A\cap V')| \leq |A+A| \leq K |A| \leq 2^m \leq P_2'\big(2^{34}K^{13}\big) |A\cap V'|.$$
Applying Ruzsa's covering lemma (Lemma~\ref{lem:ruzsacovering}), we obtain that $A$ can be covered by $P_2'\big(2^{34}K^{13}\big)$ translates of $2(A\cap V') \subseteq V' + V' = \psi(0)+V'$.

In Step~$(4)$, we can choose a subspace $V \leq V' + \psi(0)$ of size between~$|A|/2$ and~$|A|$, which will then cover~$V'$ using at most $2^{11} K^4$ cosets.
This subspace~$V$ covers~$A$ using at most $P_1'(K) := 2^{11}K^4 P_2'\big(2^{34}K^{13}\big)$ translates, as wished.

Overall, the complexity of the algorithm is as follows.
We use $O(K+\log |A|)$ random samples from $A$.
The number of queries to~$A$ is as given by Lemma~\ref{lem:findingaffine}, where each query to~$f$ and to~$S$ costs $2^{2K}$ queries to~$A$;
using that $m = \log|A| + O(\log K)$ and $\log|U| = O(\log|A| + K)$, we then require at most
$$2^{2K} \cdot K^{O(\log K)} (m+\log|U|)^2 \log (m+\log|U|) = 2^{O(K)} (\log |A|)^2 \log \log |A|$$
queries to~$A$.
The total runtime is the cost of Lemma~\ref{lem:findingaffine}, the cost of inverting $\pi$, and the cost for making the queries to $f$ and $S$, i.e.,
$$
K^{O(\log K)} (m+n)^3 \log (m+n) + O(m^2 n) + K^{O(\log K)} (m+n)^2 \log (m+n) \cdot O(mn + 2^{2K}n).
$$
This scales as $K^{O(\log K)} n^4 \log n + 2^{O(K)} n^3 \log n$, finishing the proof.
\end{proof}

We proceed to state and prove algorithmic versions of two structural theorems whose existential version were shown to be equivalent to the the $\PFR$ theorem \cite{green2004finite, green2005notes}.

\begin{theorem}[Homomorphism testing]
\label{thm:algoPFR2'}
    Suppose $f: \F_2^m \to \F_2^n$ satisfies
    $$\Pr_{x_1+x_2 = x_3+x_4} \big[f(x_1)+f(x_2) = f(x_3)+f(x_4)\big] \geq 1/K.$$
    There is a randomized algorithm that makes $K^{O(\log K)} (m+n)^2 \log (m+n)$ queries to $f$, runs in $K^{O(\log K)} (m+n)^3 \log (m+n)$ time and, with probability at least $2/3$, outputs a matrix $M \in \F_2^{n\times m}$ and a vector $v\in \F_2^n$ such that
    $$\Pr_{x\in \F_2^m} \big[f(x) = Mx + v\big] \geq 1/P_2'(K).$$
\end{theorem}

\begin{proof}
This follows immediately from Lemma~\ref{lem:findingaffine} with $S = \F_2^m$.
\end{proof}

\begin{theorem}[Structured approximate homomorphism]
\label{thm:algoPFR3'}
    Suppose $f: \F_2^m \to \F_2^n$ satisfies
    $$\big|\big\{f(x)+f(y)-f(x+y):\: x, y\in \F_2^m \big\}\big| \leq K.$$
    There is a randomized algorithm that makes $K^{O(\log K)} (m+n)^2 \log (m+n)$ queries to $f$, runs in $K^{O(\log K)} (m+n)^3 \log (m+n)$ time and, with probability at least $2/3$, outputs a matrix $M \in \F_2^{n\times m}$ such that
    $$|\{f(x) - Mx:\: x\in \F_2^m\}| \leq P_3'(K).$$
\end{theorem}

\begin{proof}
We first show that the property in the statement implies that
\begin{equation} \label{eq:propPFR2}
    \Pr_{x_1+x_2 = x_3+x_4} \big[f(x_1)+f(x_2) = f(x_3)+f(x_4)\big] \geq \frac{1}{K}.
\end{equation}
Indeed, denote $\Delta f := \big\{f(x)+f(y)-f(x+y):\: x, y\in \F_2^m \big\}$, so that $|\Delta f| \leq K$ by assumption.
Then
\begin{align*}
    &\Exp_{b\in \Delta f} \Exp_{x\in \F_2^m} \Exp_{y\in \F_2^m} \one \big[f(x)+f(y)-f(x+y) = b\big] \\
    &\qquad = \frac{1}{|\Delta f|} \Exp_{x, y\in \F_2^m} \sum_{b\in \Delta f} \one \big[f(x)+f(y)-f(x+y) = b\big] \\
    &\qquad = \frac{1}{|\Delta f|} \Exp_{x, y\in \F_2^m} 1 \\
    &\qquad \geq \frac{1}{K},
\end{align*}
and so by Cauchy-Schwarz
\begin{align*}
    \frac{1}{K^2}
    &\leq \Exp_{b\in \Delta f} \Exp_{x\in \F_2^m} \Big(\Exp_{y\in \F_2^m} \one \big[f(x)+f(y)-f(x+y) = b\big]\Big)^2 \\
    &= \Exp_{b\in \Delta f} \Exp_{x\in \F_2^m} \Exp_{y, z\in \F_2^m} \one \big[f(x)+f(y)-f(x+y) = b = f(x)+f(z)-f(x+z)\big] \\
    &= \frac{1}{K} \Exp_{x, y, z\in \F_2^m} \sum_{b\in \Delta f} \one \big[f(y)-f(x+y) = b-f(x) = f(z)-f(x+z)\big] \\
    &= \frac{1}{K} \Exp_{x, y, z\in \F_2^m} \one \big[f(y)-f(x+y) = f(z)-f(x+z)\big] \\
    &= \frac{1}{K} \Exp_{x_1+x_2 = x_3+x_4} \one \big[f(x_1)+f(x_2) = f(x_3)+f(x_4)\big],
\end{align*}
which gives inequality~\eqref{eq:propPFR2} as desired.

We may then apply Lemma~\ref{lem:findingaffine} (with $S = \F_2^m$) to obtain a matrix $M \in \F_2^{n\times m}$ and a vector $v\in \F_2^n$ such that, with probability at least $0.7$, we have
\begin{equation}
\label{eq:conclPFR2}
    \Pr_{x\in \F_2^m} \big[f(x) = Mx + v\big] \geq 1/P_2'(K).
\end{equation}
We claim that, if this inequality holds (and $|\Delta f| \leq K$), then
\begin{equation}
\label{eq:conclPFR3}
    |\{f(x) - Mx:\: x\in \F_2^m\}| \leq K^2 P_2'(K),
\end{equation}
which is the property we want with $P_3'(K) = K^2 P_2'(K)$.
It then suffices to prove~\eqref{eq:conclPFR3}.

Denote $E := \big\{x\in \F_2^m:\: f(x) = Mx + v\big\}$, so that $|E| \geq 2^m/P_2'(K)$ by  equation~\eqref{eq:conclPFR2}.
Then
$$|\F_2^m + E| = 2^m \leq P_2'(K)\cdot |E|,$$
so we may use Ruzsa's covering lemma (Lemma~\ref{lem:ruzsacovering}, with $S = E$ and $T = \F_2^m$) to conclude there exists a set $X\subseteq \F_2^m$ of size $P_2'(K)$ such that $\F_2^m \subseteq X + 2E$.
In other words, every element of $\F_2^m$ can be written as $x+y+z$ with $x\in X$ and $y, z\in E$, where $|X| \leq P_2'(K)$.

Now, for every $x\in X$, $y, z\in E$, by definition of the set $\Delta f$ there exist $b, b'\in \Delta f$ such~that
$$f(x+y)-f(x)-f(y) = b \quad \text{and} \quad f(x+y+z)-f(x+y)-f(z) = b'.$$
Summing these two identities, we conclude that
\begin{align*}
    f(x+y+z)
    &= f(x) + f(y) + f(z) + b+b' \\
    &= f(x) + My + Mz + b+b' \\
    &= f(x) + M(x+y+z) - Mx + b+b',
\end{align*}
and so
$$f(x+y+z) - M(x+y+z) = f(x) - Mx + b+b' \in \big\{f(x') - Mx':\: x'\in X\big\} + \Delta f+\Delta f$$
is contained in a set of size at most $|X|\cdot |\Delta f|^2 \leq K^2 P_2'(K)$.
This gives  equation~\eqref{eq:conclPFR3} and concludes the proof of the theorem.
\end{proof}

\section{Quantum algorithmic PFR theorem}
\label{sec:quantum}

In this section, we provide our quantum algorithm for the $\PFR$ theorem.
We start by introducing the relevant quantum information notation and the concepts and results needed for our proof.

\subsection{Quantum information} Let $\ket{0}=\tvect{1}{0}$ and $\ket{1}=\tvect{0}{1}$ be the basis for $\mathbb{C}^2$, the space in which single qubits live. An arbitrary pure single qubit state is a \emph{superposition} of $\ket{0},\ket{1}$ and has the form $\alpha\ket{0}+\beta\ket{1}=\tvect{\alpha}{\beta}$ where $\alpha,\beta\in \mathbb{C}$ and $|\alpha|^2+|\beta|^2=1$. To define multi-qubit quantum states, we will work with the basis of the Hilbert space $\mathbb{C}^{2^n}$ defined by $\ket{x} = \otimes_{i=1}^n \ket{x_i}$ for $x \in \{0,1\}^n$ built from the $n$-fold tensor product of $\ket{0},\ket{1}$.
An arbitrary $n$-qubit quantum state $\ket{\psi}\in \mathbb{C}^{2^n}$ can then be written as $\ket{\psi}=\sum_{x\in \{0,1\}^n}\alpha_x \ket{x}$ where $\alpha_x\in \mathbb{C}$ and $\sum_x |\alpha_x|^2=1$. Similarly, one can define $\bra{\psi}$ as the complex-conjugate transpose of the state $\ket{\psi}$. A valid \emph{quantum operation} on quantum states can be expressed as a \emph{unitary matrix} $U$ (which satisfies $UU^\dagger =U^\dagger U=\mathbb{I}$ with $U^\dagger$ denoting the complex-conjugate transpose of $U$). An  application of a unitary $U$ to the state $\ket{\psi}$ results in another quantum state $U\ket{\psi}$. In order to obtain classical information from  a quantum state, one can \emph{measure} the quantum state in the computational basis (i.e., $\{\ket{x}\}_{x \in \{0,1\}^n}$) to obtain a classical bit string $z\in \{0,1\}^n$ according to the probability distribution $\{|\alpha_z|^2\}_z$. We will work with the metric of \emph{infidelity} between two $n$-qubit pure quantum states $\ket{\psi}$ and $\ket{\phi}$ defined as $1 - |\langle \psi | \phi \rangle|^2$. It will also be convenient to work simply with fidelity, defined as $|\langle \psi | \phi \rangle|^2$. We refer the interested reader to~\cite{nielsen2010quantum} for more on quantum information.

\textbf{Clifford gates.} 
Clifford circuits are those generated by Hadamard gate, $S$ gate and \textsf{CNOT} gate defined as below
$$
H=\frac{1}{\sqrt{2}}\begin{pmatrix}
1 & 1\\
1 & -1
\end{pmatrix},\: S=\begin{pmatrix}
1 & 0\\
0 & i
\end{pmatrix},\: \textsf{CNOT}=\begin{pmatrix}
1 & 0 & 0 & 0\\
0 & 1 & 0 & 0\\
0 & 0 & 0 & 1\\
0 & 0 & 1 & 0
\end{pmatrix}.
$$
We will need one additional non-Clifford gate in this section, the Toffoli gate. To describe this, first observe the action of the \textsf{CNOT} gate:
$$
\textsf{CNOT}:\ket{a,b}\mapsto\ket{a,a\oplus b}\quad \text{ for all }a,b\in \{0,1\}.
$$
This is a $2$-qubit gate as it acts on the two qubits $\ket{a,b}$, which in particular, flips the second qubit if $a=1$ and keeps the second qubit as it is if $a=0$. The Toffoli gate, denoted by \textsf{CCNOT}, can then be defined as
$$
\textsf{CCNOT}:\ket{a_1,a_2,b}\mapsto \ket{a_1,a_2,b\oplus a_1\cdot a_2} \quad \text{ for all } a_1,a_2,b\in \{0,1\},
$$
i.e., the gate flips the last qubit \emph{if and only if}  the first $2$ qubits are $1$. 

The states produced by Clifford circuits acting on the input $\ket{0^n}$ are \emph{stabilizer states}, which have the following characterization.
(Recall that we write $|\cdot|: \F_2 \to \bset{} \subset \Z$ for the natural identification map.)

\begin{theorem}[Stabilizer state formula \cite{dehaene2003clifford, nest2008classical}]
\label{thm:neststabilizer}
    Every $k$-qubit stabilizer state can be expressed as
    $$
    \frac{1}{\sqrt{|A|}}\sum_{x\in A}i^{|\ell(x)|}(-1)^{q(x)}\ket{x},
    $$
    for some affine subspace $A\subseteq \F_2^k$, quadratic polynomial $q$ and linear polynomial $\ell$ in the variables $(x_1,\ldots,x_k)\in \F_2^k$. 
\end{theorem}
Notably, stabilizer states encode non-classical quadratic functions over an affine subspace, as noted earlier in Section~\ref{sec:symplectic}. Our quantum algorithms will revolve around stabilizer states.

Our quantum algorithmic $\PFR$ theorem will crucially use the agnostic learnability of stabilizer states. Informally the task here is as follows: supposing an arbitrary quantum state $\ket{\psi}$ was $\tau$-close to an \emph{unknown} stabilizer state $\ket{\phi}$ in fidelity (i.e., $|\langle \phi| \psi \rangle|^2 \geq \tau$), output the ``nearest" stabilizer state $\ket{\phi'}$ that is $(\tau-\varepsilon)$-close. Recently, Chen, Gong, Ye and Zhang~\cite{chen2024stabilizer} gave an agnostic learning algorithm that runs in time quasipolynomial in $1/\tau$ and polynomial in the other parameters. Formally, their result is stated in the following theorem.

\begin{theorem}[Agnostic stabilizer learning~\cite{chen2024stabilizer}]
\label{thm:sitanbootstrapping}
Let $\Stab_n$ be  the class of stabilizer states on $n$ qubits. Let $0<\varepsilon\leq \tau$ and $\delta \in (0,1)$.
There is an algorithm that, given access to copies of an $n$-qubit pure state $\ket{\psi}$ with $\max_{\ket{\phi'} \in \Stab_n} |\langle \phi' | \psi \rangle|^2 \geq \tau$, outputs a $\ket{\phi} \in \Stab_n$ such that $|\langle \phi | \psi \rangle|^2 \geq \tau - \varepsilon$ with probability at least $1-\delta$.
The algorithm performs single-copy and two-copy measurements on at most $n\cdot \poly(1/\varepsilon,(1/\tau)^{\log 1/\tau})$ copies of $\ket{\psi}$ and runs in time $n^3 \poly(1/\varepsilon, (1/\tau)^{\log 1/\tau})$.
\end{theorem}
We will also require the following subroutines for estimating the overlap between two states and obtaining unitaries that prepare stabilizer states.
\begin{lemma}[\textsf{SWAP} test~\cite{nielsen2010quantum}]
\label{lem:swap_test}
Let $\varepsilon,\delta \in (0,1)$. Given two arbitrary $n$-qubit quantum states $\ket{\psi}$ and~$\ket{\phi}$, there is a quantum algorithm that estimates $|\langle \psi | \phi \rangle|^2$ up to error $\varepsilon$ with probability at least $1-\delta$ using $O(1/\varepsilon^2\cdot \log(1/\delta))$ copies of $\ket{\psi},\ket{\phi}$ and runs in $O(n/\varepsilon^2\cdot \log(1/\delta))$ time.
\end{lemma}

\begin{lemma}[Clifford synthesis~\cite{aaranson2004sim}]\label{lem:clifford_synthesis}
Given the classical description of an $n$-qubit stabilizer state $\ket{\phi}$, there is a quantum algorithm that outputs a Clifford circuit $U$ that prepares $\ket{\phi}$, using $O(n^2/\log n)$ many single-qubit and two-qubit Clifford~gates.
\end{lemma}

\subsection{The algorithm}

We now give a quantum algorithm that is quadratically better in the query complexity compared to the classical algorithm shown in the section above. We restate the statement of the quantum result in more detail below.

\begin{theorem}[Quantum algorithmic $\PFR$]
\label{thm:algoquantumPFR1}
    Suppose $A \subseteq \mathbb{F}_2^n$ satisfies $|A + A| \leq K|A|$.
    There is a quantum algorithm that takes $O(\log|A| + K)$ random samples from $A$, makes $2^{O(K)} \log |A|$ quantum queries to $A$, runs in time $K^{O(\log K)} n^3 + 2^{O(K)}n^2$ and has the following guarantee:
    with probability at least $2/3$, it outputs a basis for a subspace $V \leq \mathbb{F}_2^n$ of size $|V| \leq |A|$ such that $A$ can be covered by $P_1'(K)$ translates of $V$.
\end{theorem}

To prove the above theorem, we will reprove Lemma~\ref{lem:findingaffine} in the quantum setting, but now taking advantage of the main result (Theorem~\ref{thm:sitanbootstrapping}) of~\cite{chen2024stabilizer}, which allows us to find the closest stabilizer state to a given unknown $n$-qubit quantum state. Formally, the quantum version of Lemma~\ref{lem:findingaffine} is as follows.

\begin{lemma}
\label{lem:findingfreimanquantum}
Suppose $S\subseteq \F_2^m$ and $f: S \to \F_2^n$ satisfy
$$\big|\big\{(x_1, x_2, x_3, x_4)\in S^4:\: x_1+x_2 = x_3+x_4 \,\text{ and }\, f(x_1)+f(x_2) = f(x_3)+f(x_4)\big\}\big| \geq 2^{3m}/K.$$
There is a quantum algorithm that makes $K^{O(\log K)} (m+n)$ quantum queries to $S$ and to $f$, runs in $K^{O(\log K)}(m+n)^3$ time and, with probability at least $0.7$, returns $M\in \F_2^{n\times m}$, $v\in \F_2^n$ such that
$$\big|\big\{x\in S:\: f(x) = Mx+v\big\}\big| \geq 2^m/P_2'(K).$$
\end{lemma}

To prove Lemma~\ref{lem:findingfreimanquantum} and describe its corresponding algorithm, we need a quantum protocol to prepare the quantum state that encodes the function 
$$
g_S(x,y) = \one_S(x) (-1)^{f(x)\cdot y},
$$
which from Claim~\ref{claim:highu3}, we know has high correlation with a quadratic function.

\begin{claim}\label{claim:prep_psi}
Consider the context of Lemma~\ref{lem:findingfreimanquantum}. Let $\delta \in (0,1)$. Suppose we have quantum query access to $S$ via the oracle $O_S$ and query access to  $f: S \rightarrow \F_2^n$ via the oracle $O_f$ as~follows
$$
\ket{x,0} \stackrel{O_S}\longrightarrow \ket{x,\one_S(x)}, \quad \ket{x,0^n} \stackrel{O_f}\longrightarrow \ket{x,f(x)}.
$$
There is a quantum algorithm that makes $O(K \log(1/\delta))$ queries to $O_S, O_f$ and, with probability at least $1-\delta$, prepares an $(m+n)$-qubit state $\ket{\psi}$ encoding $g_S(x,y)$ as
$$
\ket{\psi}= \frac{1}{\sqrt{2^n |S|}}\sum_{\substack{x\in S,\\y\in \F_2^n}}(-1)^{f(x)\cdot y}\ket{x,y}.
$$
This algorithm takes $O(K (m+n) \log(1/\delta))$ time to prepare one copy of $\ket{\psi}$.
\end{claim}
\begin{proof}
First, given quantum query access to $S$, the algorithm prepares
$$
\frac{1}{\sqrt{2^m}}\sum_{x\in \F_2^m} \ket{x,0} \stackrel{O_S}{\longrightarrow} \frac{1}{\sqrt{2^m}}\sum_{x\in \F_2^m} \ket{x,\one_S(x)},
$$
and measures the second register. With probability $|S|/2^m \geq 1/K$, the algorithm obtains $1$, in which case the resulting state is $\ket{S}=\frac{1}{\sqrt{|S|}}\sum_{x\in S}\ket{x}$. So, making $O(K \log(1/\delta))$ quantum queries, one can prepare $\ket{S}$ with probability at least $1-\delta/2$.

The algorithm then simply performs the~following
\begin{align*}
\frac{1}{\sqrt{|S|}}\sum_{x\in S}\ket{x}\otimes \frac{1}{\sqrt{2^n}}\sum_{y\in \F_2^n}\ket{y} &\stackrel{O_f}\longrightarrow \frac{1}{\sqrt{2^n|S|}}\sum_{\substack{x\in S,\\y\in \F_2^n}}\ket{x,y,f(x)}\\
&\longrightarrow \frac{1}{\sqrt{2^n|S|}}\sum_{\substack{x\in S,\\y\in \F_2^n}}\ket{x,y,f(x)}\otimes_{i=1}^n\ket{f(x)_i\cdot y_i}\\
&\longrightarrow \frac{1}{\sqrt{2^n|S|}}\sum_{\substack{x\in S,\\y\in \F_2^n}}\ket{x,y}\ket{f(x)}\otimes_{i=1}^n\ket{f(x)_i\cdot y_i} \ket{f(x)\cdot y}.
\end{align*}
where the second operation is by applying  $n$ many \textsf{CCNOT} gates with the control qubits being $y_i,f(x)_i$ applied onto the target qubit $\ket{0}_i$,
and the third operation is by applying $n$ \textsf{CNOT} gates between the control qubit $\ket{f(x)_i\cdot y_i}$ and target qubit $\ket{0}$. After obtaining the final state above, the algorithm applies a single-qubit Hadamard on the last qubit and measures it in the computational basis. If the result is $1$, the algorithm continues. First note that, if the last qubit was $1$, then the resulting quantum state is
\begin{equation*}
     \frac{1}{\sqrt{2^n |S|}}\sum_{\substack{x\in S,\\y\in \F_2^n}}(-1)^{f(x)\cdot y}\ket{x,y}\ket{f(x)}\otimes_{i=1}^n\ket{f(x)_i\cdot y_i} \ket{1}.
\end{equation*}
Furthermore, the probability of obtaining $1$ is exactly $1/2$.
If the result of the measurement is~$0$, we repeat the Hadamard-and-measure process for $O(\log(1/\delta))$ times until the result is~$1$.

Upon succeeding, the algorithm inverts the $n$ many \textsf{CCNOT} gates and the query operator $O_f$ to obtain the~state
$$
\ket{\psi}= \frac{1}{\sqrt{2^n |S|}}\sum_{\substack{x\in S,\\y\in \F_2^n}}(-1)^{f(x)\cdot y}\ket{x,y}.
$$
The algorithm uses $O(K(m+n) \log(1/\delta))$ time and $O(K \log(1/\delta))$ queries to prepare $\ket{\psi}$ with probability $1-\delta$.
\end{proof}

We are now ready to prove Lemma~\ref{lem:findingfreimanquantum}.

\begin{proof}[ of Lemma~\ref{lem:findingfreimanquantum}]
The proof will be similar to the classical proof in Lemma~\ref{lem:findingaffine}.
As in that case, we are guaranteed by Claim~\ref{claim:highu3} that there exists a \emph{quadratic} polynomial $q: \F_2^{m} \times \F_2^n \rightarrow \F_2$ which has high correlation with $g_S(x,y) := \one_S(x) (-1)^{f(x)\cdot y},$ i.e.,
$$
\left| \Exp_{(x,y) \in \mathbb{F}_2^m \times \mathbb{F}_2^n} \one_S(x)(-1)^{f(x)\cdot y} (-1)^{q(x, y)} \right| \geq \frac{1}{P_2(K)},
$$
where $P_2(\cdot)$ is the polynomial promised by Lemma~\ref{lem:closeaffine}. For simplicity in notation, let us denote $\sigma:=1/P_2(K)$. In particular, defining the quantum states 
$$
\ket{\psi} = \frac{1}{\sqrt{2^n |S|}}\sum_{x\in S,\\y\in \F_2^n}(-1)^{f(x)\cdot y}\ket{x,y}, \quad \ket{\phi_q}= \frac{1}{\sqrt{2^{m+n}}} \sum_{x\in \F_2^m,y\in \F_2^n}(-1)^{q(x, y)}\ket{x,y},
$$
we have that $|\langle \psi|\phi_q\rangle|^2 \geq \sigma^{2}$. Moreover by Theorem~\ref{thm:neststabilizer}, we note that the quantum state $\ket{\phi_q}$ is a \emph{stabilizer state},\footnote{We remark that $\ket{\phi_q}$ is in fact a degree-$2$ phase state (i.e., the subspace is $\F_2^{m+n}$ and there are no complex phases), but we will not use that here.} and thus the stabilizer fidelity of $\ket{\psi}$ is also at least~$\sigma^{2}$.

We now use Theorem~\ref{thm:sitanbootstrapping} on copies of $\ket{\psi}$ prepared using Claim~\ref{claim:prep_psi}, with the error instantiated as $\varepsilon=\sigma^{2}/2$, to learn a stabilizer state $\ket{s}$ such that $|\langle s| \psi \rangle|^2 \geq \sigma^{2}/2$. By Theorem~\ref{thm:neststabilizer}, we can write this stabilizer state as
\begin{equation}\label{eq:stab_state_from_bootstrapping}
\ket{s} = \frac{1}{\sqrt{|A_s|}} \sum_{z \in A_s} i^{|\ell_s(z)|} (-1)^{q_s(z)} \ket{z},    
\end{equation}
where $A_s \subseteq \F_2^{m+n}$ is an affine subspace, $\ell_s$ is a linear polynomial and $q_s$ is a quadratic polynomial.
Denote $T:= S \times \F_2^n$.
To lower bound the size of $A_s$, we will lower bound the size of $A_s \cap T$:
\begin{align*}
\frac{\sigma}{\sqrt{2}} \leq |\langle \psi | s \rangle| 
&= \Big|\frac{1}{\sqrt{2^n |S|\cdot |A_s|}} \sum_{\substack{x \in S,\, y \in \F_2^n \\ (x,y) \in A_s}} i^{|\ell(x,y)|} (-1)^{q_s(x,y) + f(x)\cdot y} \Big| \\
&\leq \frac{1}{\sqrt{|A_s| \cdot |S|\cdot 2^n}} \sum_{(x,y) \in A_s \cap T} \Big| i^{|\ell(x,y)|} (-1)^{q_s(x,y) + f(x)\cdot y} \Big|\\
& \leq \frac{\sqrt{|A_s \cap T|}}{\sqrt{|S|\cdot 2^n}},    
\end{align*}
where we have used the triangle inequality in the second line and noted that each internal term is at $1$ in the final inequality along with using $|A_s| \geq |A_s \cap T|$. The above result implies that $|A_s|$ is large, i.e.,
\begin{equation}\label{eq:lb_size_As}
|A_s| \geq |A_s \cap T| \geq (\sigma^{3}/2) 2^{m+n},
\end{equation}
as $|S| \geq 2^m/K \geq \sigma \cdot 2^m$. Writing $A_s = a + H_s$ where $H_s$ is a linear subspace, we then have $\textsf{codim}(H_s) \leq \log(2/\sigma^{3})$. To obtain a quadratic phase state $\ket{\phi_p}$ corresponding to a quadratic phase polynomial $p:\F_2^{m+n} \rightarrow \F_2$ that has high fidelity with $\ket{\psi}$ from the description of $\ket{s}$, we make the following observations which will inform our approach.

Let us denote the orthogonal complement of $H_s$ as $H_s^\perp = \{x \in \F_2^{m+n} : x \cdot h = 0, \,\,\forall h \in H_s\}$. The Fourier decomposition of $\one_{A_s}(x)$ is given by
\begin{equation}\label{eq:fourier_indicator_coset}
\one_{A_s}(x) = \frac{|H_s|}{2^{m+n}}\sum_{\lambda \in H_s^\perp} (-1)^{\lambda \cdot (a+x)},    
\end{equation}
which follows from the observation that
\begin{align*}
\Exp_x[\one_{A_s}(x) (-1)^{\lambda \cdot x}] = 2^{-(m+n)} \sum_{x \in H_s} (-1)^{\lambda \cdot (a+x)} 
&= |H_s| 2^{-(m+n)} (-1)^{\lambda \cdot a} \Exp_{x \in H_s}[(-1)^{\lambda \cdot x}] \\
&= |H_s| 2^{-(m+n)} (-1)^{\lambda \cdot a} \one\{\lambda \in H_s^\perp\}.    
\end{align*}
Recalling that $g_S(x,y) = \one_S(x) (-1)^{f(x)\cdot y}$, we then observe that
\begin{align*}
    \sigma/\sqrt{2} \leq |\la \psi | s \ra |
    &= \frac{1}{\sqrt{2^n |S|\cdot |A_s|}} \Big| \sum_{z \in \F_2^{m+n}} \one_{A_s}(z) g_S(z) (-1)^{q_s(z)} i^{|\ell_s(z)|} \Big| \\
    &= \frac{|H_s|}{\sqrt{2^n |S|\cdot |A_s|}} \Big| \sum_{\lambda \in H_s^\perp} \Exp_{z \in \F_2^{m+n}} (-1)^{\lambda \cdot (a+z)} g_S(z) (-1)^{q_s(z)} i^{|\ell_s(z)|} \Big| \\
    &\leq \frac{|H_s|\cdot|H_s^\perp|}{\sqrt{2^n |S|\cdot |A_s|}} \max_{\lambda \in H_s^\perp} \Big|  \Exp_{z \in \F_2^{m+n}} (-1)^{\lambda \cdot z} g_S(z) (-1)^{q_s(z)} i^{|\ell_s(z)|} \Big| \\
    &\leq \frac{\sqrt{2}}{\sigma^{2}} \max_{\lambda \in H_s^\perp} \Big|  \Exp_{z \in \F_2^{m+n}} (-1)^{\lambda \cdot z} g_S(z) (-1)^{q_s(z)} i^{|\ell_s(z)|} \Big|,
\end{align*}
where we used the Fourier decomposition of $\one_{A_s}(z)$ from equation~\eqref{eq:fourier_indicator_coset} in the second line, applied the triangle inequality along with considering the $\lambda \in H_s^\perp$ which maximizes the expectation in the third line, and finally used equation~\eqref{eq:lb_size_As} as well as noting $|H_s|\cdot|H_s^\perp| = 2^{m+n}$ and $|S|\geq \sigma \cdot 2^m$.
From this chain of inequalities, we conclude there exists $\lambda^\star \in H_s^\perp$ such that
\begin{equation} \label{eq:gS_corr}
    \Big|  \Exp_{z \in \F_2^{m+n}} g_S(z) (-1)^{q_s(z) + \lambda^\star \cdot z} i^{|\ell_s(z)|} \Big| \geq \sigma^{3}/2.
\end{equation}
Define the function $h(z) := g_S(z) (-1)^{q_s(z) + \lambda^\star \cdot z}$.
Additionally, we denote
\begin{equation*}
    R_h = \Exp_{z\in \F_2^{m+n}} h(z) \one\{\ell_s(z)=0\}, \quad I_h = \Exp_{z\in \F_2^{m+n}} h(z) \one\{\ell_s(z)=1\},
\end{equation*}
so that by equation~\eqref{eq:gS_corr} we have
$$\sqrt{R_h^2 + I_h^2} = |R_h + iI_h| \geq \sigma^{3}/2.$$

Now, we consider the two candidate quadratic polynomials $p_0(z) := q_s(z) + \lambda^\star \cdot z$ and $p_1(z) := q_s(z) + \lambda^\star \cdot z + \ell_s(z)$, where $q_s$ and $\ell_s$ are the quadratic and linear polynomials corresponding to the stabilizer state $\ket{s}$ in hand~(equation~\eqref{eq:stab_state_from_bootstrapping}) and $\lambda^* \in H_s^\perp$ satisfies equation~\eqref{eq:gS_corr}.
We observe that the quadratic phase states
$$\ket{\phi_{p_b}} = \frac{1}{\sqrt{2^{m+n}}} \sum_{x\in \F_2^{m+n}}(-1)^{p_b(z)}\ket{z}, \quad b\in \bset{},$$
satisfy
\begin{align*}
    |\la \psi | \phi_{p_b}\ra|
    &= \frac{1}{\sqrt{2^{m+2n} |S|}} \bigg|\sum_{z \in \F_2^{m+n}} g_S(z) (-1)^{q_s(z) + \lambda^\star \cdot z + b\ell_s(z)} \bigg| \\
    &= \sqrt{\frac{2^m}{|S|}} \,\Big|\Exp_{z \in \F_2^{m+n}} h(z) (-1)^{b\ell_s(z)} \Big| \\
    &= \sqrt{\frac{2^m}{|S|}} \big| R_h + (-1)^b I_h \big| \\
    &\geq \big| R_h + (-1)^b I_h \big|.
\end{align*}
Noting that $\max\{|u+v|,|u-v|\} = |u| + |v| \geq \sqrt{u^2 + v^2}$, we then have
\begin{equation}\label{eq:promise_quadratics}
    \max\big\{|\la \psi | \phi_{p_0} \ra|,\, |\la \psi | \phi_{p_1} \ra|\big\} \geq \sqrt{R_h^2 + I_h^2} \geq \sigma^{3}/2.
\end{equation}
In other words, one of the quadratic polynomials $p_0$ or $p_1$ has high correlation with~$g_S(z)$.

To determine this quadratic polynomial, we now use the following approach. We create the list of candidate quadratic polynomials $L$ where we add the polynomials $p_0^\lambda(z) := q_s(z) + \lambda \cdot z$ and $p_1^\lambda(z) := q_s(z) + \lambda \cdot z + \ell_s(z)$ for each $\lambda \in H_s^\perp$. This list will be of size $|L| = 2 |H_s^\perp| \leq 4/\sigma^{3}$, where we have used $\textsf{codim}(H_s) \leq \log(2/\sigma^{3})$. For each $p \in L$, we prepare copies of the quadratic phase state $\ket{\phi_p}$ (which is also a stabilizer state) using Lemma~\ref{lem:clifford_synthesis} and then measure $|\la \psi | \phi_p \ra|^2$ using the \textsf{SWAP} test (Lemma~\ref{lem:swap_test}) up to error $\sigma^{3}/4$ and output the quadratic polynomial $p^\star$ that maximizes the fidelity. This consumes $\poly(1/\sigma)$ sample complexity and $O(n^2/\log n \cdot\poly(1/\sigma))$ time complexity. We are guaranteed by equation~\eqref{eq:promise_quadratics} that $(-1)^{p^\star}$ satisfies
\begin{equation}
\Big| \Exp_{(x,y) \in \F_2^{m} \times \F_2^n} \one_S(x) (-1)^{p^\star(x,y) + f(x)\cdot y} \Big| \geq \frac{\sigma^{3}}{4} = \frac{1}{4P_2(K)^3},    
\end{equation}
where we have substituted back $\sigma=1/P_2(K)$ set earlier. Having determined the polynomial $p^\star$, we now proceed as in Lemma~\ref{lem:findingaffine} to determine the affine linear function $\varphi$ that agrees with $f$ on many values $x \in S$. This completes the proof of the lemma. The query complexity and time complexity are determined by Theorem~\ref{thm:sitanbootstrapping}.
\end{proof}

With this lemma, we are finally ready to prove the main theorem of this section.

\begin{proof}[ of Theorem~\ref{thm:algoquantumPFR1}]
We proceed similarly to the proof of \cref{thm:algoPFR1}.
The algorithm is given as follows:
\begin{enumerate}
    \item Sample $t = 28 \log|A| + 56K$ uniformly random elements from $A$, and denote their linear span by $U$.
    Let $A' := A\cap U$.
    \item Take a random linear map $\pi: U \to \mathbb{F}_2^m$ where $m = \log|A| + 4\log K + 10$. Let $S = \pi(A')$ denote the image of $A'$ under $\pi$, and let $f: S \to U$ be the inverse of~$\pi$ restricted to $S$.
    \item Apply Lemma~\ref{lem:findingfreimanquantum} to obtain an affine-linear map $\psi: \mathbb{F}_2^m \to U$ such that $f(x) = \psi(x)$ for at least $|A|/P_2'\big(2^{34}K^{13}\big)$ values $x \in S$.
    \item Take a subspace $V$ of $\textsf{Im}(\psi) + \psi(0)$ having size at most $|A|$, and output a basis for~$V$.
\end{enumerate}
The only difference between the classical and quantum algorithms is in Step $(3)$.  So, we do not reproduce the correctness analysis and refer the reader to the classical proof of Theorem~\ref{thm:algoPFR1}.

Overall, the complexity of the algorithm is as follows.
The sample complexity to the set $A$ is $O(K+\log |A|)$, as given in step $(1)$.
Computing $\ker(\pi) \leq U$ and $\pi^{-1}: \mathrm{Im}(\pi) \to U/\ker(\pi)$ takes $O(m^2n)$ time and, after this is done, each query to $S$ and $f$ takes $2^{2K}$ queries to $A$ and $O(mn)$ time.
The total number of queries to $A$ needed to apply Lemma~\ref{lem:findingfreimanquantum} is then
$$2^{2K} \cdot K^{O(\log K)} (m+\log|U|) = 2^{O(K)} \log |A|,$$
where we used that $m = \log |A| + O(\log K)$ and $\log|U| = O(\log|A| + K)$.
The total runtime is the cost of Lemma~\ref{lem:findingfreimanquantum}, the cost of inverting $\pi$ and the cost of making queries to $S$ and $f$, i.e.
$$K^{O(\log K)}(m+\log|U|)^3 + O(m^2 n) + K^{O(\log K)} (m+\log|U|) \cdot O(mn + 2^{2K}n),$$
which scales as $K^{O(\log K)} n^3 + 2^{O(K)}n^2$, concluding the proof of the theorem.
\end{proof}

\subsection*{Acknowledgments}

The authors thank David Gross and Markus Heinrich for illuminating discussions regarding the stabilizer formalism and representation theory.
JB was supported by the Dutch Research Council (NWO) as part of the NETWORKS programme (Grant No.~024.002.003).
DCS was supported by the Engineering and Physical Sciences Research Council grant on Robust and Reliable Quantum Computing (grant reference EP/W032635/1). TG was supported by ERC Starting Grant 101163189 and UKRI Future Leaders Fellowship MR/X023583/1.

\bibliographystyle{plainurl}
\bibliography{references}

\end{document}